\documentclass[a4paper]{article}
\usepackage{amsmath,amsthm,amsfonts,thmtools}
\usepackage{amssymb}
\usepackage[cmintegrals]{newtxmath}
\usepackage{graphicx}
\usepackage{listings}
\usepackage[font=footnotesize]{caption}
\usepackage[font=scriptsize]{subcaption}
\usepackage{euscript}
\usepackage{units}
\usepackage{enumerate}
\usepackage{enumitem}
\usepackage{xcolor}
\usepackage{todonotes}
\usepackage{booktabs}
\usepackage{tikz}
\usepackage{comment}
 
\DeclareMathAlphabet{\pazocal}{OMS}{zplm}{m}{n}

\usepackage[bookmarks=true,hidelinks]{hyperref}
\usepackage{bbm}
\usepackage{multirow}

\usepackage{mathtools}
\usepackage{subcaption} 
\usepackage{graphicx}
\usepackage{pgfplots}
\usepackage{float}
\usepackage{algpseudocode}
\usepackage{bbold}
\usepackage{cite}
\usepackage{microtype}
\usepackage{rotating}
\usepackage{enumitem}
\usepackage{bm}
\usepackage{hyperref}

\numberwithin{equation}{section}
\newtheorem{theorem}{Theorem}[section]

\newtheorem{lemma}[theorem]{Lemma}

\newtheorem{assumption}[theorem]{Assumption}
\newtheorem{algorithm}[theorem]{Algorithm}

\numberwithin{equation}{section}

\theoremstyle{definition}

\newtheoremstyle{myremarkstyle}{}{}{}{}{\bfseries}{.}{ }{}
\theoremstyle{myremarkstyle}
\declaretheorem[name=Remark,qed=$\blacksquare$,numberlike=theorem]{remark}

\makeatletter
\newcommand*{\intavg}{%
  \mint@l{-}{}%
}
\newcommand*{\mint@l}[2]{%
  \@ifnextchar\limits{%
    \mint@l{#1}%
  }{%
    \@ifnextchar\nolimits{%
      \mint@l{#1}%
    }{%
      \@ifnextchar\displaylimits{%
        \mint@l{#1}%
      }{%
        \mint@s{#2}{#1}%
      }%
    }%
  }%
}
\newcommand*{\mint@s}[2]{%
  \@ifnextchar_{%
    \mint@sub{#1}{#2}%
  }{%
    \@ifnextchar^{%
      \mint@sup{#1}{#2}%
    }{%
      \mint@{#1}{#2}{}{}%
    }%
  }%
}
\def\mint@sub#1#2_#3{%
  \@ifnextchar^{%
    \mint@sub@sup{#1}{#2}{#3}%
  }{%
    \mint@{#1}{#2}{#3}{}%
  }%
}
\def\mint@sup#1#2^#3{%
  \@ifnextchar_{%
    \mint@sub@sup{#1}{#2}{#3}%
  }{%
    \mint@{#1}{#2}{}{#3}%
  }%
}
\def\mint@sub@sup#1#2#3^#4{%
  \mint@{#1}{#2}{#3}{#4}%
}
\def\mint@sup@sub#1#2#3_#4{%
  \mint@{#1}{#2}{#4}{#3}%
}
\newcommand*{\mint@}[4]{%
  \mathop{}%
  \mkern-\thinmuskip
  \mathchoice{%
    \mint@@{#1}{#2}{#3}{#4}%
        \displaystyle\textstyle\scriptstyle
  }{%
    \mint@@{#1}{#2}{#3}{#4}%
        \textstyle\scriptstyle\scriptstyle
  }{%
    \mint@@{#1}{#2}{#3}{#4}%
        \scriptstyle\scriptscriptstyle\scriptscriptstyle
  }{%
    \mint@@{#1}{#2}{#3}{#4}%
        \scriptscriptstyle\scriptscriptstyle\scriptscriptstyle
  }%
  \mkern-\thinmuskip
  \int#1%
  \ifx\\#3\\\else_{#3}\fi
  \ifx\\#4\\\else^{#4}\fi  
}
\newcommand*{\mint@@}[7]{%
  \begingroup
    \sbox0{$#5\int\m@th$}%
    \sbox2{$#5\int_{}\m@th$}%
    \dimen2=\wd0 %
    \let\mint@limits=#1\relax
    \ifx\mint@limits\relax
      \sbox4{$#5\int_{\kern1sp}^{\kern1sp}\m@th$}%
      \ifdim\wd4>\wd2 %
        \let\mint@limits=\nolimits
      \else
        \let\mint@limits=\limits
      \fi
    \fi
    \ifx\mint@limits\displaylimits
      \ifx#5\displaystyle
        \let\mint@limits=\limits
      \fi
    \fi
    \ifx\mint@limits\limits
      \sbox0{$#7#3\m@th$}%
      \sbox2{$#7#4\m@th$}%
      \ifdim\wd0>\dimen2 %
        \dimen2=\wd0 %
      \fi
      \ifdim\wd2>\dimen2 %
        \dimen2=\wd2 %
      \fi
    \fi
    \rlap{%
      $#5%
        \vcenter{%
          \hbox to\dimen2{%
            \hss
            $#6{#2}\m@th$%
            \hss
          }%
        }%
      $%
    }%
  \endgroup
}

\def\XXint#1#2#3{{\setbox0=\hbox{$#1{#2#3}{\int}$ }
		\vcenter{\hbox{$#2#3$ }}\kern-.6\wd0}}

\renewcommand{\geq}{\geqslant}
\renewcommand{\leq}{\leqslant}

\renewcommand{\epsilon}{\varepsilon}
\renewcommand{\phi}{\varphi}

\newcommand{\nl}{{\mathbf n}}

\newcommand{\IP}{\mathbb{P}}

\newcommand{\Cont}{\EuScript{C}}

\newcommand{\R}{\mathbb{R}}
\newcommand{\N}{\mathbb{N}}

\newcommand{\bu}{{\bf u}}		

\newcommand{\n}{\mathbf{n}}
\newcommand{\f}{\mathbf{f}}

\newcommand{\IE}{\mathbb{E}}

\newcommand{\train}{\EuScript{S}}

\newcommand{\val}{\EuScript{V}}

\newcommand{\er}{\EuScript{E}}

\newcommand{\df}{\EuScript{D}}
\newcommand{\dom}{\mathbb{D}}
\newcommand{\bv}{\mathbf{v}}

\newcommand{\res}{\EuScript{R}}
\newcommand{\vtrain}{{\mathbf S}}
\newcommand{\bD}{\partial D}

\begin{document}

\date{\today}

\title{Estimates on the generalization error of Physics Informed Neural Networks (PINNs) for approximating PDEs.}

\author{Siddhartha Mishra \thanks{Seminar for Applied Mathematics (SAM), D-Math \newline
  ETH Z\"urich, R\"amistrasse 101, 
  Z\"urich-8092, Switzerland} and
  Roberto Molinaro \thanks{Seminar for Applied Mathematics (SAM), D-Math \newline
  ETH Z\"urich, R\"amistrasse 101, 
  Z\"urich-8092, Switzerland.}}

\maketitle
\begin{abstract}
Physics informed neural networks (PINNs) have recently been widely used for robust and accurate approximation of PDEs. We provide upper bounds on the generalization error of PINNs approximating solutions of the forward problem for PDEs. An abstract formalism is introduced and stability properties of the underlying PDE are leveraged to derive an estimate for the generalization error in terms of the training error and number of training samples. This abstract framework is illustrated with several examples of nonlinear PDEs. Numerical experiments, validating the proposed theory, are also presented. 
\end{abstract}

\section{Introduction}
Deep learning has emerged as a central tool in science and technology in the last few years. It is based on using deep artificial neural networks (DNNs), which are formed by composing many layers of affine transformations and scalar non-linearities. These deep neural networks have been applied with tremendous success \cite{DLnat} in a variety of tasks such as image and text classification, speech and natural language processing, robotics, game intelligence and protein folding \cite{Dfold}, among others. 

Partial differential equations (PDEs) model a vast array of natural and manmade phenomena in all areas of science and technology. Explicit solution formulas are only available for very specific types and examples of PDEs. Hence, numerical simulations are necessary for most practical applications featuring PDEs. A diverse set of methods for approximating PDEs numerically are available, such as finite difference, finite element, finite volume and spectral methods. Although very successful in practice, it is still challenging to numerically simulate problems such as Uncertainty quantification (UQ), multi-scale and multi-physics problems, Inverse and constrained optimization problems, PDEs in domains with very complex geometries and PDEs in very high dimensions. Could deep learning assist in the computations of these hitherto difficult to simulate problems involving PDEs? 

Deep learning techniques are being increasingly used in the numerical approximations of PDEs. A brief and very incomplete survey of this rapidly emerging literature follows: one approach in using deep neural networks (DNNs) for numerically approximating PDEs is based on explicit (or semi-implicit) representation formulas such as the Feynman-Kac formula for parabolic (and elliptic) PDEs, whose compositional structure is in turn utilized to facilitate approximation by DNNs. This approach is presented and analyzed for a variety of (parametric) elliptic and parabolic PDEs in \cite{E1,HEJ1,Jent1} and references therein, see a recent paper \cite{Pet1} for a similar approach to approximating linear transport equations with deep neural networks. 

Another strategy is to enhance existing numerical methods by adding deep learning inspired modules into them, for instance by learning free parameters of numerical schemes from data \cite{SM1,DR1} and references therein. 

A third approach consists of using deep neural networks to learn \emph{observables} or quantities of interest of the solutions of the underlying PDEs, from data. This approach has been described in \cite{LMR1,LMM1,MR1,LMPR1} in the context of uncertainty quantification and PDE constrained optimization and \cite{QUAT1} for model order reduction, among others. 

Finally, deep neural networks possess the so-called \emph{universal approximation property} \cite{Bar1,Hor1,Cy1}, namely any continuous, even measurable, function can be approximated by DNNs, see \cite{YAR1} for very precise descriptions of the required neural network architecture for functions with sufficient Sobolev regularity. Hence, it is natural to use deep neural networks as ansatz spaces for the solutions of PDEs, in particular by collocating the PDE residual at training points (see section \ref{sec:2}, Algorithm \ref{alg:PINN} for the detailed description of this approach). This approach was first proposed in \cite{Lag1,Lag2}. However, it has been revived and developed in significantly greater detail in the pioneering contributions of Karniadakis and collaborators, starting with \cite{KAR1,KAR2}. These authors have termed the underlying neural networks as \emph{Physics Informed Neural Networks} (PINNs) and we will continue to use this, by now, widely accepted terminology in this paper. There has been an explosive growth of papers that present algorithms with PINNs for various applications to both forward and inverse problems for PDEs and a very incomplete list of references include \cite{KAR4,KAR5,KAR6,KAR7} and references therein. Needless to say, PINNs have emerged as a very successful paradigm for approximating different aspects of solutions of PDEs. 

Why do PINNs approximate a wide variety of PDEs so well? Although many heuristic reasons have been proposed in some of the afore cited papers, particularly by highlighting the role played by (even small amounts of) data in driving the neural network towards the target and the role of the residual in changing training modes, there is very little rigorous justification of why PINNs work. With the exception of the very recent paper \cite{DAR1}, there are few rigorous bounds on the approximation error due to PINNs. In \cite{DAR1}, the authors prove consistency of PINNs by showing convergence, under reasonable hypothesis, to solutions of linear elliptic and parabolic PDEs as the number of training samples is increased. A detailed comparison between \cite{DAR1} and the current article is provided in section \ref{sec:6}. 

The main goal of this paper is to provide some rationale of why PINNs are so efficient at approximating solutions for the \emph{forward problem} for PDEs, under reasonable and verifiable hypothesis on the underlying PDE. The key issue that we wish to address is \emph{to understand the mechanisms by which minimizing the PDE residuals at collocation points, which is the main ingredient of the PINNs training algorithm, might lead to control (bounds) on the overall error}. To this end, we will present an abstract framework for PINNs that encompasses a wide variety of potential applications, including to nonlinear PDEs, and prove estimates on the so-called \emph{generalization error} i.e, the error of the neural network on predicting unseen data. This abstract estimate bounds the generalization error in terms of the underlying training error, number of training (collocation) points and stability bounds for the underlying PDE. Our generalization error estimate will show that the error due to approximating the underlying PDE with a trained PINN will be sufficiently low as long as 
\begin{itemize}
    \item The training error is low i.e, the PINN has been trained well. This error is computed and monitored during the training process. Hence, it is available \emph{a posteriori}.
    \item The number of training (collocation) points is sufficiently large. This number is determined by the error due to an underlying quadrature rule.
    \item The solution of the underlying PDE is stable (with respect to perturbations of inputs) in a very precise manner. For nonlinear PDEs, these stability bounds might require that the solutions of the underlying PDEs (and the PINNs) are sufficiently regular. \item Implicit constants that arise in the stability and quadrature error estimates, which depend on the underlying PINNs, need to be controlled in a suitable manner.
\end{itemize}
Thus, with the derived error estimate, we identify possible mechanisms by which PINNs are able to approximate PDEs so well and provide a firm mathematical foundation for approximations by PINNs. 

We will also provide three concrete examples to illustrate our abstract framework and error estimate, namely linear and semi-linear parabolic equations, one-dimensional scalar quasilinear parabolic (and hyperbolic) conservation laws and the incompressible Euler equations of fluid dynamics. The abstract error estimate is described in concrete terms for each example and numerical experiments are presented to support the proposed theory. The aim is to convince the reader of why PINNs, when correctly formulated, are successful at approximating the forward problem for PDEs numerically. 

The rest of the paper is organized as follows: in section \ref{sec:2}, we formulate PINNs for an abstract PDE and prove the estimate on the generalization error. In sections \ref{sec:3}, \ref{sec:4} and \ref{sec:5}, the abstract framework and error estimate is worked out in the concrete examples of semi-linear Parabolic PDEs, viscous scalar conservation laws and the incompressible Euler equations of fluid dynamics. 
\section{An abstract framework for Physics informed Neural Networks}
\label{sec:2}
\subsection{The underlying abstract PDE}
\label{sec:21}
Let $X,Y$ be separable Banach spaces with norms $\| \cdot \|_{X}$ and $\|\cdot\|_{Y}$, respectively. For definiteness, we set $Y = L^p(\dom;\R^m)$ and $X= W^{s,q}(\dom;\R^m)$, for $m \geq 1$, $1 \leq p,q < \infty$ and $s \geq 0$, with $\dom \subset \R^{\bar{d}}$, for some $\bar{d} \geq 1$. In particular, we will also consider space-time domains with $\dom = (0,T) \times D \subset \R^d$ with $d \geq 1$. In this case $\bar{d} = d +1$. Let $X^{\ast} \subset X$ and $Y^{\ast} \subset Y$ be closed subspaces with norms $\|\cdot \|_{X^{\ast}}$ and $\|\cdot\|_{Y^{\ast}}$, respectively.

We start by considering the following abstract formulation of our underlying PDE:
\begin{equation}
    \label{eq:pde}
    \df(\bu) = \f.
\end{equation}
Here, the \emph{differential operator} is a mapping, $\df: X^{\ast} \mapsto Y^{\ast}$ and the \emph{input} $\f \in Y^{\ast}$, such that 
\begin{equation}
\label{eq:assm1}
\begin{aligned}
&(H1): \quad \|\df(\bu)\|_{Y^{\ast}} < +\infty, \quad \forall~ \bu \in X^{\ast}, ~{\rm with}~\|\bu\|_{X^{\ast}} < +\infty. \\
&(H2):\quad \|\f\|_{Y^{\ast}} < +\infty. 
\end{aligned}
\end{equation}
Moreover, we assume that for all $\f \in Y^{\ast}$, there exists a unique $\bu \in X^{\ast}$ such that \eqref{eq:pde} holds. Furthermore, the solutions of the abstract PDE \eqref{eq:pde} satisfy the following stability bound, let $Z \subset X^{\ast} \subset X$ be a closed subspace with norm $\|\cdot\|_{Z}$.
\begin{assumption}For any $\bu,\bv \in Z$, the differential operator $\df$ satisfies
\begin{equation}
    \label{eq:assm2}
(H3):\quad    \|\bu - \bv\|_{X} \leq C_{pde}\left(\|\bu\|_Z, \|\bv\|_Z \right) \|\df(\bu) - \df(\bv)\|_{Y}.
\end{equation}
\end{assumption}
Here, the constant $C_{pde} > 0$ explicitly depends on $\|\bu\|_Z$ and $\|\bv\|_Z$.

As a first example of a PDE with solutions satisfying (H3) \eqref{eq:assm2}, consider the linear differential operator $\df: X \mapsto Y$ i.e $\df(\alpha \bu + \beta \bv) = \alpha \df(\bu) + \beta \df(\bv)$, for any $\alpha,\beta \in \R$. For simplicity, let $X^{\ast} = X$ and $Y^{\ast} = Y$. By the assumptions on the existence and uniqueness of the underlying linear PDE \eqref{eq:pde}, there exists an \emph{inverse} operator $\df^{-1}: Y \mapsto X$. Note that the assumption \eqref{eq:assm2} is satisfied if the inverse is bounded i.e, $\|\df^{-1}\| \leq C < +\infty$, with respect to the natural norm on linear operators from $Y$ to $X$. Thus, the assumption \eqref{eq:assm2} on stability boils down to the boundedness of the inverse operator for linear PDEs. Many well-known linear PDEs possess such bounded inverses \cite{DL1}. 

As a second example, we will consider a nonlinear PDE \eqref{eq:pde}, but with a well-defined linearization i.e, there exists an operator $\overline{\df}: X^{\ast} \mapsto Y^{\ast}$, such that 
\begin{equation}
    \label{eq:lin}
    \df(\bu) - \df(\bv) = \overline{\df}_{(\bu,\bv)}\left(\bu - \bv\right), \quad \forall \bu,\bv \in X^{\ast}.
\end{equation}
Again for simplicity, we will assume that $X^{\ast} = X$ and $Y^{\ast} = Y$. We further assume that the inverse of $\overline{\df}$ exists and is bounded in the following manner,
\begin{equation}
    \label{eq:lin1}
    \|\left(\overline{\df}_{(\bu,\bv)}\right)^{-1}\| \leq C\left(\|\bu\|_{X},\|\bv\|_{X}  \right) < +\infty, \quad \forall \bu,\bv \in X,
    \end{equation}
with the norm of $\overline{\df}^{-1}$ being an operator norm, induced by linear operators from $Y$ to $X$. Then a straightforward calculation shows that \eqref{eq:lin1} suffices to establish the stability bound \eqref{eq:assm2}. 

Further and more concrete examples of PDEs satisfying the above assumptions will be provided in the following sections.
\begin{remark}
We note that initial and boundary conditions are implicitly included in the PDE \eqref{eq:pde}. In particular, the regularity of the boundary of the underlying domain is implicit in our formulation and the boundary needs to be sufficiently regular in order to yield a regular enough solution. Moreover, by assuming that $Y = L^p(\dom)$, we are implicitly assuming some regularity for the solutions of the abstract PDE \eqref{eq:pde}. In particular, depending on the order of derivatives in the differential operator $\df$, we can expect that the solution $\bu$ satisfies the PDE \eqref{eq:pde} in a classical sense i.e, pointwise. 
\end{remark}
\subsection{Quadrature rules}
\label{sec:22}
In the following section, we need to consider approximating integrals of functions. Hence, we need an abstract formulation for quadrature. To this end, we consider a mapping $g: \dom \mapsto \R^m$, such that $g \in Z^{\ast} \subset Y^{\ast}$. We are interested in approximating the integral,
$$
\overline{g}:= \int\limits_{\dom} g(y) dy,
$$
with $dy$ denoting the $\bar{d}$-dimensional Lebesgue measure. In order to approximate the above integral by a quadrature rule, we need the quadrature points $y_{i} \in \dom$ for $1 \leq i \leq N$, for some $N \in \N$ as well as weights $w_i$, with $w_i \in \R_+$. Then a quadrature is defined by,
\begin{equation}
    \label{eq:quad}
    \overline{g}_N := \sum\limits_{i=1}^N w_i g(y_i),
\end{equation}
for weights $w_i$ and quadrature points $y_i$. We further assume that the quadrature error is bounded as,
\begin{equation}
    \label{eq:assm3}
    \left|\overline{g} - \overline{g}_N\right| \leq C_{quad}
    \left(\|g\|_{Z^{\ast}},\bar{d} \right) N^{-\alpha},
\end{equation}
for some $\alpha > 0$. 

As long as the domain $\dom$ is in reasonably low dimension i.e $\bar{d} \leq 4$, we can use standard (composite) Gauss quadrature rules on an underlying grid. In this case, the quadrature points and weights depend on the order of the quadrature rule \cite{SBbook} and the rate $\alpha$ depends on the regularity of the underlying integrand i.e, on the space $Z^{\ast}$.

On the other hand, these grid based quadrature rules are not suitable for domains in high dimensions. For moderately high dimensions i.e $4 \leq \bar{d} \approx 20$, we can use \emph{low discrepancy sequences}, such as the Sobol and Halton sequences, as quadrature points \cite{CAF1}. As long as the integrand $g$ is of bounded \emph{Hardy-Krause variation} \cite{owen}, the error in \eqref{eq:assm3} converges at a rate $(\log(N))^{\bar{d}}N^{-1}$. One can also employ sparse grids and Smolyak quadrature rules \cite{sgbook} in this regime. 

For problems in very high dimensions $\bar{d} \gg 20$, Monte-Carlo quadrature is the numerical integration method of choice \cite{CAF1}. In this case, the quadrature points are randomly chosen, independent and identically distributed (with respect to a scaled Lebesgue measure). The estimate \eqref{eq:assm3} holds in the root mean square (RMS) sense and the rate of convergence is $\alpha = \frac{1}{2}$.

\subsection{PINNs}
\label{sec:23}
In this section, we will describe physics-informed neural networks (PINNs). We start with a description of neural networks which form the basis of PINNs.
\subsubsection{Neural Networks.}
Given an input $y \in \dom$, a feedforward neural network (also termed as a multi-layer perceptron), shown in figure \ref{fig:1}, transforms it to an output, through a layer of units (neurons) which compose of either affine-linear maps between units (in successive layers) or scalar non-linear activation functions within units \cite{DLbook}, resulting in the representation,
\begin{equation}
\label{eq:ann1}
\bu_{\theta}(y) = C_K \circ\sigma \circ C_{K-1}\ldots \ldots \ldots \circ\sigma \circ C_2 \circ \sigma \circ C_1(y).
\end{equation} 
Here, $\circ$ refers to the composition of functions and $\sigma$ is a scalar (non-linear) activation function. A large variety of activation functions have been considered in the machine learning literature \cite{DLbook}. Popular choices for the activation function $\sigma$ in \eqref{eq:ann1} include the sigmoid function, the hyperbolic tangent function and the \emph{ReLU} function.

For any $1 \leq k \leq K$, we define
\begin{equation}
\label{eq:C}
C_k z_k = W_k z_k + b_k, \quad {\rm for} ~ W_k \in \R^{d_{k+1} \times d_k}, z_k \in \R^{d_k}, b_k \in \R^{d_{k+1}}.
\end{equation}
For consistency of notation, we set $d_1 = \bar{d}$ and $d_K = m$. 

Thus in the terminology of machine learning (see also figure \ref{fig:1}), our neural network \eqref{eq:ann1} consists of an input layer, an output layer and $(K-1)$ hidden layers for some $1 < K \in \N$. The $k$-th hidden layer (with $d_k$ neurons) is given an input vector $z_k \in \R^{d_k}$ and transforms it first by an affine linear map $C_k$ \eqref{eq:C} and then by a nonlinear (component wise) activation $\sigma$. A straightforward addition shows that our network contains $\left(\bar{d} + m + \sum\limits_{k=2}^{K-1} d_k\right)$ neurons. 
We also denote, 
\begin{equation}
\label{eq:theta}
\theta = \{W_k, b_k\},~ \theta_W = \{ W_k \}\quad \forall~ 1 \leq k \leq K,
\end{equation} 
to be the concatenated set of (tunable) weights for our network. It is straightforward to check that $\theta \in \Theta \subset \R^M$ with
\begin{equation}
\label{eq:ns}
M = \sum\limits_{k=1}^{K-1} (d_k +1) d_{k+1}.
\end{equation}
 \begin{figure}[htbp]
\centering
\includegraphics[width=8cm]{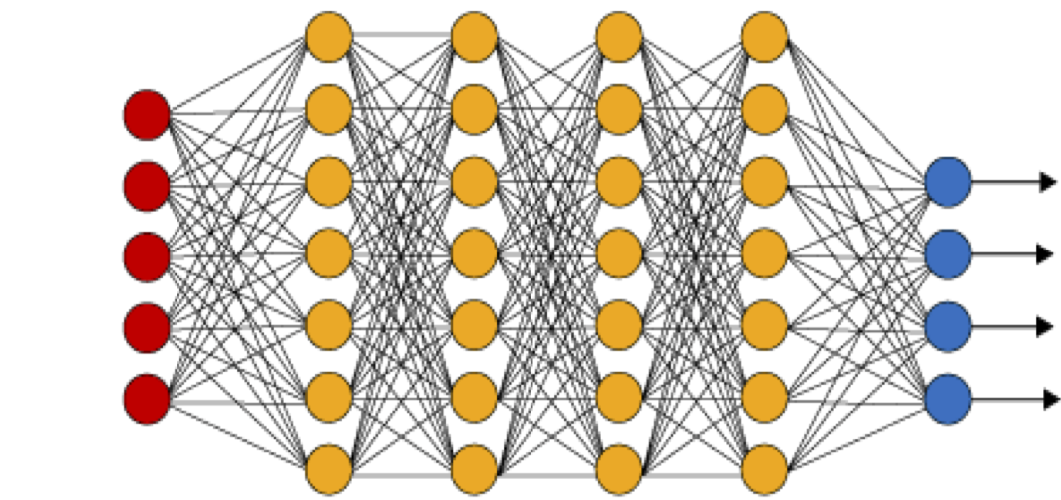}
\caption{An illustration of a (fully connected) deep neural network. The red neurons represent the inputs to the network and the blue neurons denote the output layer. They are
connected by hidden layers with yellow neurons. Each hidden unit (neuron) is connected by affine linear maps between units in different layers and then with nonlinear (scalar) activation functions within units.}
\label{fig:1}
\end{figure}
\subsubsection{Training PINNs: Loss functions and optimization}
The neural network $\bu_{\theta}$ \eqref{eq:ann1} depends on the tuning parameter $\theta \in \Theta$ of weights and biases. Within the standard paradigm of deep learning \cite{DLbook}, one \emph{trains} the network by finding tuning parameters $\theta$ such that the loss (error, mismatch, regret) between the neural network and the underlying target is minimized. Here, our target is the solution $\bu \in X^{\ast}$ of the abstract PDE \eqref{eq:pde} and we wish to find the tuning parameters $\theta$ such that the resulting neural network $\bu_{\theta}$ approximates $\bu$. 

Following standard practice of machine learning, one obtains training data $\bu(y)$, for all $y \in \train$, with training set $\train \subset \dom$ and then minimizes a loss function of the form $\sum\limits_{\train} \|\bu(y) - \bu_{\theta}(y)\|_{X}$ to find the neural network approximation for $\bu$. However, obtaining this training data requires possibly expensive numerical simulations of the underlying PDE \eqref{eq:pde}. In order to circumvent this issue, the authors of \cite{Lag1} suggest a different strategy. An abstract paraphrasing of this strategy runs as follows: we assume that for every $\theta \in \Theta$, the neural network $\bu_{\theta} \in X^{\ast}$ and $\|\bu_{\theta} \|_{X^{\ast}} < +\infty$. We define the following \emph{residual}:
\begin{equation}
    \label{eq:res1}
    \res_{\theta} = \res(\bu_\theta):= \df\left(\bu_{\theta}\right) - \f. 
\end{equation}
By assumptions (H1),(H2) \eqref{eq:assm1}, we see that $\res \in Y^{\ast}$ and $\|\res\|_{Y^{\ast}} < +\infty$ for all $\theta \in \Theta$. Note that $\res(\bu) = \df(\bu) - \f \equiv 0$, for the solution $\bu$ of the PDE \eqref{eq:pde}. Hence, the term \emph{residual} is justified for \eqref{eq:res1}. 

The strategy of PINNs, following \cite{Lag1}, is to minimize the \emph{residual} \eqref{eq:res1}, over the admissible set of tuning parameters $\theta \in \Theta$ i.e 
\begin{equation}
    \label{eq:pinn1}
    {\rm Find}~\theta^{\ast} \in \Theta:\quad \theta^{\ast} = {\rm arg}\min\limits_{\theta \in \Theta} \|\res_{\theta}\|_{Y}.
\end{equation}
Realizing that $Y = L^p(\dom)$ for some $1 \leq p < \infty$, we can equivalently minimize,
\begin{equation}
    \label{eq:pinn2}
    {\rm Find}~\theta^{\ast} \in \Theta:\quad \theta^{\ast} = {\rm arg}\min\limits_{\theta \in \Theta} \|\res_{\theta}\|^p_{L^p(\dom)} = {\rm arg}\min\limits_{\theta \in \Theta} \int\limits_{\dom} |\res_{\theta}(y)|^p dy. 
\end{equation}
As it will not be possible to evaluate the integral in \eqref{eq:pinn2} exactly, we need to approximate it numerically by a quadrature rule. To this end, we use the quadrature rules \eqref{eq:quad} discussed earlier and select the \emph{training set} $\train = \{y_n\}$ with $y_n \in \dom$ for all $1 \leq n \leq N$ as the quadrature points for the quadrature rule \eqref{eq:quad} and consider the following \emph{loss function}:
\begin{equation}
    \label{eq:lf1}
    J(\theta):= \sum\limits_{n=1}^N w_n |\res_{\theta}(y_n)|^p = \sum\limits_{n=1}^N w_n \left| \df(\bu_{\theta}(y_n)) - \f(y_n) \right|^p.
\end{equation}
It is common in machine learning \cite{DLbook} to regularize the minimization problem for the loss function i.e we seek to find,
\begin{equation}
\label{eq:lf2}
\theta^{\ast} = {\rm arg}\min\limits_{\theta \in \Theta} \left(J(\theta) + \lambda_{reg} J_{reg}(\theta) \right).
\end{equation}  
Here, 
\begin{equation}
\label{eq:reg}
    J_{reg}:\Theta \to \R
\end{equation} is a \emph{regularization} (penalization) term. A popular choice is to set  $J_{reg}(\theta) = \|\theta_W\|^q_q$ for either $q=1$ (to induce sparsity) or $q=2$. The parameter $0 \leq \lambda_{reg} \ll 1$ balances the regularization term with the actual loss $J$ \eqref{eq:lf1}. 

The above minimization problem amounts to finding a minimum of a possibly non-convex function over a subset of $\R^M$ for possibly very large $M$. We will follow standard practice in machine learning and solve this minimization problem approximately by either (first-order) stochastic gradient descent methods such as ADAM \cite{adam} or even higher-order optimization methods such as LBFGS \cite{lbfgs}. 

For notational simplicity, we denote the (approximate, local) minimum in \eqref{eq:lf2} as $\theta^{\ast}$ and the underlying deep neural network $\bu^{\ast}= \bu_{\theta^{\ast}}$ will be our physics-informed neural network (PINN) approximation for the solution $\bu$ of the PDE \eqref{eq:pde}.  

The proposed algorithm for computing this PINN is given below,
\begin{algorithm} 
\label{alg:PINN} {\bf Finding a physics informed neural network to approximate the solution of the PDE \eqref{eq:pde}}. 
\begin{itemize}
\item [{\bf Inputs}:] Underlying domain $\dom$, differential operator $\df$ and input source term $\f$ for the PDE \eqref{eq:pde}, quadrature points and weights for the quadrature rule \eqref{eq:quad}, non-convex gradient based optimization algorithms.
\item [{\bf Goal}:] Find PINN $\bu^{\ast}= \bu_{\theta^{\ast}}$ for approximating the PDE \eqref{eq:pde}. 
\item [{\bf Step $1$}:] Choose the training set $\train = \{y_n\}$ for $y_n \in \dom$, for all $1 \leq n \leq N$ such that $\{y_n\}$ are quadrature points for the underlying quadrature rule \eqref{eq:quad}.
\item [{\bf Step $2$}:] For an initial value of the weight vector $\overline{\theta} \in \Theta$, evaluate the neural network $\bu_{\overline{\theta}}$ \eqref{eq:ann1}, the PDE residual \eqref{eq:res1}, the loss function \eqref{eq:lf2} and its gradients to initialize the underlying optimization
algorithm.
\item [{\bf Step $3$}:] Run the optimization algorithm till an approximate local minimum $\theta^{\ast}$ of \eqref{eq:lf2} is reached. The map $\bu^{\ast} = \bu_{\theta^{\ast}}$ is the desired PINN for approximating the solution $\bu$ of the PDE \eqref{eq:pde}. 
\end{itemize}
\end{algorithm}
\begin{remark}
\label{rem:1}
The standard practice in machine learning is to approximate the solution $\bu$ of \eqref{eq:pde} from training data $\big(z_j,\bu(z_j)\big)$, for $z_j \in \dom$ and $1 \leq j \leq N_d$, then one would like to minimize the so-called \emph{data loss}:
\begin{equation}
    \label{eq:ld}
    J_{d}(\theta) := \frac{1}{N_d} \sum\limits_{j=1}^{N_d} |\bu(z_j) - \bu_{\theta}(z_j)|^p.
\end{equation}
Comparing the loss functions \eqref{eq:lf1} and \eqref{eq:ld} reveals the essence of PINNs, i.e, for PINNs, one does not necessarily need any training data for the solution $\bu$ of \eqref{eq:pde}, only the residual \eqref{eq:res1} needs to be evaluated, for which knowledge of the differential operator and inputs for \eqref{eq:pde} suffice. Here, we distinguish between initial and boundary data, which is implicitly included into the formulation of the PDE \eqref{eq:pde} and which are necessary for any numerical solution of \eqref{eq:pde}, from other types of data, for instance values of the solution $\bu$, from the interior of the domain $\dom$. Thus, the PINN in this formulation, which is closer in spirit to the original proposal of \cite{Lag1}, can be purely thought of as a numerical method for the PDE \eqref{eq:pde}. In this form, one can consider PINNs as an example of \emph{unsupervised learning} \cite{MLbook2}, as no explicit data is necessary. 

On the other hand, the authors of \cite{KAR2} and subsequent papers, have added further flexibility to PINNs by also including training data by augmenting the loss function \eqref{eq:lf2} with the data loss \eqref{eq:ld} and seeking neural networks with tuning parameters defined by,
\begin{equation}
\label{eq:lf3}
\theta^{\ast} = {\rm arg}\min\limits_{\theta \in \Theta} \left(J_d(\theta) + \lambda J(\theta) + \lambda_{reg} J_{reg}(\theta) \right),
\end{equation} 
with an additional hyperparameter $\lambda$ that balances the data loss with the residual. This paradigm has been very successful for inverse problems, where boundary and initial data may not be known \cite{KAR2,KAR4}. However, for the rest of the paper, we focus on the forward problem and only consider PINNs, trained with algorithm \ref{alg:PINN}, that minimize the residual based loss function \ref{eq:lf2}, without needing additional data. 
\end{remark}

\begin{remark}
Algorithm \ref{alg:PINN} requires the residual \eqref{eq:res1}, for the neural network $\bu_{\theta}$, to be evaluated pointwise for every training step. Hence, one needs the neural network to be sufficiently regular. Depending on the order of derivatives in $\df$ of \eqref{eq:pde}, this can be ensured by requiring sufficient smoothness for the activation function. Hence, the ReLU activation function, which is only Lipschitz continuous, might not be admissible in this framework. On the other hand, smooth activation functions such as logistic and $\tanh$ are always admissible. 
\end{remark}
\subsection{An abstract estimate on the generalization error}
In this section, we will estimate the error due to the PINN (generated by algorithm \ref{alg:PINN}) in approximating the solution $\bu$ of the PDE \eqref{eq:pde}. The relevant concept of error is the \emph{total error}, often referred to as the so-called \emph{generalization error} (see \cite{MLbook}):
    \begin{equation}
    \label{eq:egen}
    \er_G= \er_{G} (\theta^{\ast};\train) := \|\bu-\bu^{\ast}\|_{X}
\end{equation}
Clearly, the generalization error depends on the chosen training set $\train$ and the trained neural network with tuning parameters $\theta^{\ast}$, found by algorithm \ref{alg:PINN}. However, we will suppress this dependence due to notational convenience. We remark that the generalization error \eqref{eq:egen} is the error emanating from approximating the solution $\bu$ of \eqref{eq:pde} by the PINN $\bu^{\ast}$, generated by the algorithm \ref{alg:PINN}. 

Note that there is no computation of the generalization error during the training process. On the other hand, we monitor the so-called \emph{training error} given by,
\begin{equation}
    \label{eq:train}
    \er_T:= \left(\sum\limits_{n=1}^N w_n |\res_{\theta^{\ast}}(y_n)|^p \right)^{\frac{1}{p}} = J(\theta^{\ast})^{\frac{1}{p}}.
\end{equation}
Hence, the training error $\er_T$ can be readily computed, after training has been completed, from the loss function \eqref{eq:lf2}. The generalization error \eqref{eq:egen} can be estimated in terms of the training error in the following theorem.
\begin{theorem}
\label{thm:1}
Let $\bu \in Z \subset X^{\ast}$ be the unique solution of the PDE 
\eqref{eq:pde} and assume that the stability hypothesis \eqref{eq:assm2} holds. Let $\bu^{\ast} \in Z \subset X^{\ast}$ be the PINN generated by algorithm \ref{alg:PINN}, based on the training set $\train$ of quadrature points corresponding to the quadrature rule \eqref{eq:quad}. Further assume that the residual $\res_{\theta^{\ast}}$, defined in \eqref{eq:res1}, be such that $\res_{\theta^{\ast}} \in Z^{\ast}$ and the quadrature error satisfies \eqref{eq:assm3}. Then the following estimate on the generalization error holds,
\begin{equation}
    \label{eq:egenb}
    \er_G \leq C_{pde}\er_T + C_{pde}C_{quad}^{\frac{1}{p}}N^{-\frac{\alpha}{p}},
\end{equation}
with constants $C_{pde} = C_{pde}\left(\|\bu\|_{Z},\|\bu^{\ast}\|_{Z}\right)$ and $C_{quad} = C_{quad}\left(\left\||\res_{\theta^{\ast}}|^p\right\|\right)$ stemming from \eqref{eq:assm2} and \eqref{eq:assm3}, respectively. Note that these constants $C_{pde},C_{quad}$ depend on the underlying PINN $\bu^{\ast}$, which in turn can depend on the number of training points $N$. 
\end{theorem}
\begin{proof}
In the following, we denote $\res = \res_{\theta^{\ast}}$, the residual \eqref{eq:res1}, corresponding to the trained neural network $\bu^{\ast}$. As $\bu$ solves the PDE \eqref{eq:pde} and $\res$ is defined by \eqref{eq:res1}, we easily see that, 
\begin{equation}
\label{eq:pf1}
\res = \df(\bu^{\ast}) -\df(\bu).
\end{equation}
Hence, we can directly apply the stability bound \eqref{eq:assm2} to yield,
\begin{align*}
\er_G &=    \|\bu - \bu^{\ast}\|_{X} \quad ({\rm by}~\eqref{eq:egen}), \\
&\leq C_{pde} \|\df(\bu^{\ast}) - \df(\bu) \|_{Y}  \quad ({\rm by}~\eqref{eq:assm2}), \\
&\leq C_{pde} \|\res\|_{Y} \quad ({\rm by}~\eqref{eq:pf1}).
\end{align*}
By the fact that $Y = L^p(\dom)$, the definition of the training error \eqref{eq:train} and the quadrature rule \eqref{eq:quad}, we see that,
\begin{align*}
\|\res\|^p_{Y} \approx \left(\sum\limits_{n=1}^N w_n |\res_{\theta^{\ast}}|^p \right) = \er_T^p.
\end{align*}
Hence, the training error is a quadrature for the residual \eqref{eq:res1} and the resulting quadrature error, given by \eqref{eq:assm3} translates to,
\begin{align*}
   \|\res\|^p_{Y} \leq \er_T^p + C_{quad}N^{-\alpha}. 
\end{align*}
Therefore,
\begin{align*}
    \er_G^p &\leq C^p_{pde} \|\res\|^p_{Y} \\
    &\leq C^p_{pde} \left(\er_T^p + C_{quad}N^{-\alpha} \right)\\
    \Rightarrow \quad \er_G &\leq  C_{pde}\er_T + C_{pde}C_{quad}^{\frac{1}{p}}N^{-\frac{\alpha}{p}},
\end{align*}
which is the desired estimate \eqref{eq:egenb}.
\end{proof}
Several remarks about the above theorem are in order.
\begin{remark}
\label{rem:11}
An inspection of the estimate \eqref{eq:egenb} reveals that the generalization error for the PINN is small as long as the following hold,
\begin{itemize}
    \item The training error $\er_T << 1$ has to be sufficiently small. Note that we have no a priori control on the training error but can compute it \emph{a posteriori}. 
    \item The quadrature error depends on the number of quadrature (training) points $N$ as well as on the quadrature constant $C_{quad}$, which in turn, depends on the residual of the underlying PINN $\bu^{\ast}$ and indirectly, on the number of training points $N$. In particular, it might grow with increasing $N$. Thus, one might need to choose the number of quadrature points $N$ large enough such that $C_{quad}N^{-\alpha} << 1$. More pertinently, the constant $C_{quad}$ depends on the architecture of the underlying neural network, In particular, the evaluation of $C_{quad}$ depends on the details on the underlying PDE and quadrature rule and cannot be worked out in the abstract setup of Theorem \ref{thm:1}, see \cite{DRM1} for an example in the context of specific linear Kolmogorov PDEs. 
    \item The constant $C_{pde}$ that encodes the stability of the underlying PDE and depends on both the underlying exact solution $\bu$ as well as the trained PINN $\bu^{\ast}$ needs to be bounded. 
\end{itemize}
\end{remark}

\begin{remark}
\label{rem:2}
The main point of the error estimate \eqref{eq:egenb} is to relate the total (generalization) error to the training error. As mentioned before, it is not at all obvious that minimizing the PDE residual \eqref{eq:res1} can lead to any control on the generalization error \eqref{eq:egen}. The error estimate \eqref{eq:egenb} provides this connection by leveraging the stability of PDEs to derive an error estimate in terms of the PDE residual. The bound \eqref{eq:egenb} further estimates the PDE residual in terms of the training error and the quadrature error. 
\end{remark}
\begin{remark}
We observe that the error estimate \eqref{eq:egenb} holds for any function $\bu^{\ast}$, defined in terms of the residual \eqref{eq:res1}, as it only relies on the stability of the underlying PDE and on the accuracy of the underlying quadrature rule. We have not used the structure of neural networks in the proof of Theorem \ref{thm:1}, nor have we used any specific details of the training in algorithm \ref{alg:PINN}. In particular, this estimate \eqref{eq:egenb} holds for any neural network of the form \eqref{eq:ann1}, including those realized during the training process in the algorithm \ref{alg:PINN}. However, there is no guarantee that the training error $\er_T$ is small for such neural networks. On the other hand, one could expect that the training error for the trained PINN $\bu^{\ast}$ is small, which leads to a small generalization error as long as there is some control on the constants $C_{pde}$ and $C_{quad}$.  \end{remark}

\subsubsection{Case of random training points.}
\label{sec:mc}
The quadrature rule \eqref{eq:quad} (and the quadrature error \eqref{eq:assm3}) play an important role in the error estimate \eqref{eq:egenb}. In principle, as long as the underlying solution (and PINN) are sufficiently regular, one can use standard grid based (composite) Gauss quadrature rules and obtain low quadrature errors, resulting in a large $\alpha$ in the estimate \eqref{eq:egenb}. However, the constant $C_{quad}$, depends explicitly on the dimension $\bar{d}$ of $\dom$ and suffers from the so-called \emph{curse of dimensionality}. This can be alleviated for moderately high dimensions by using low-discrepancy sequences as the quadrature (and training) points. As long as the solution $\bu$ (and PINN $\bu^{\ast}$) are of bounded Hardy-Krause variation \cite{owen}, we know that $\alpha =1$ and the constant $C_{quad} = (\log(N))^{\bar{d}}$ for the resulting Quasi-Monte Carlo quadrature error \cite{CAF1}. Note that the logarithmic correction can be dominated as long $N \approx e^{\bar{d}}$. However, if the underlying dimension is very high, using low-discrepancy sequences might not be efficient, see \cite{MR1} for a similar observation in the context of standard neural networks. 

Hence, for problems in very high dimensions, one has to use random quadrature points as the resulting Monte Carlo quadrature does not suffer from the curse of dimensionality. However, as we use the quadrature points as training points for the PINN in algorithm \ref{alg:PINN}, we cannot directly use error estimates for Monte Carlo quadrature (central limit theorem) for approximating the integral in $\|\res_{\theta^{\ast}}\|^p_Y$ as $\res_{\theta^{\ast}}(y_n)$ in the training error \eqref{eq:train} could be \emph{correlated} (and are not independent) for different training points $y_n$. In general, advanced tools from statistical learning theory \cite{CS1,MLbook} such as Rademacher complexity, VC dimension etc are used to circumvent these correlations and in practice, lead to large overestimates on the generalization error \cite{AR1}. Here, we follow the approach of \cite{LMM1} and references therein and adopt a pragmatic framework in estimating the generalization error in terms of the training error. 

For simplicity of notation, we set the domain $\dom = [0,1]^{\bar{d}}$ and $Y=L^1(\dom)$ and start by recalling that the points in the training set $\EuScript{S}$ are chosen randomly from domain $\dom$, independently and identically distributed with the underlying Lebesgue measure $dy$. We will identify the training set $\train$ with the vector $\vtrain \in \dom^N$, defined by 
 $$
 \vtrain = \left[y_1,y_2,\ldots y_N \right],
 $$
 which is distributed according to the measure  $d\vtrain:= dy\otimes\cdot\cdot dy$ (a $N$-fold product Lebesgue measure). We also denote $(\Omega, \Sigma, \IP)$ as the underlying complete probability space from which random draws are made. Expectation with respect to this underlying probability measure $\IP$ is denoted as $\IE$
 
 Note that as the training set is randomly chosen, the 
 trained PINN $\bu^{\ast}$, explicitly depends on the training set $\vtrain$ i.e, $\bu^{\ast} = \bu^{\ast}(\vtrain)$. Consequently, the generalization and training errors now explicitly depend on the training set i.e,
 \begin{align*}
     \er_G(\vtrain)= \er_{G} (\theta^{\ast};\vtrain) := \|\bu-\bu^{\ast}(\vtrain)\|_{X}, \quad 
      \er_T(\vtrain)= \er_T(\theta^{\ast};\vtrain):= \frac{1}{N} \sum\limits_{n=1}^N  |\res_{\theta^{\ast}}(y_n;\vtrain)|
 \end{align*}
 
Next, we follow \cite{LMM1} and define the so-called \emph{cumulative} (average over training sets) generalization and training errors as,
 \begin{equation}
    \label{eq:ecgen}
    \bar{\er}_G = \int\limits_{\dom^N} \er_{G} (\vtrain) d\vtrain = \int\limits_{\dom^N} \|\bu - \bu^{\ast}(\vtrain) \|_X d\vtrain,
    \end{equation}
and
\begin{equation}
    \label{eq:ectrain}
    \bar{\er}_T =  \int\limits_{\dom^N} \er_{T} (\vtrain) d\vtrain.
\end{equation}
Note that the cumulative errors $\bar{\er}_{G,T}$ are deterministic quantities. In \cite{LMM1} and references therein, one followed standard practice in machine learning and also computed the so-called \emph{validation set},
\begin{equation}
    \label{eq:val}
    \val = \{z_j \in \dom, ~ 1\leq j \leq N, \quad z_j ~i.i.d~wrt~dy\}.
\end{equation}
Note that the number of validation points need not necessarily be set to the number of training points. We do so here for the sake of notational simplicity. 

The validation set is chosen before the start of the training process and is independent of the training sets. We can define the \emph{cumulative validation error} as, 
\begin{equation}
\label{eq:ecval}
\bar{\er}_V= \frac{1}{N}\int_{\dom^N}\sum\limits_{j=1}^N |\res_{\theta^{\ast}}(z_n;\vtrain) | d\vtrain.
\end{equation}
We observe that as the set $\val$ is drawn randomly from $\dom$ with underlying distribution $dy$, the cumulative validation error is a random quantity, $\bar{\er}_V = \bar{\er}_V(\omega)$ with $\omega \in \Omega$. We suppress this $\omega$-dependence for notational convenience. Finally, we introduce the \emph{validation gap}:
\begin{equation}
    \label{eq:vgap}
    \er_{TV}:= \IE\left(|\bar{\er}_{T} - \bar{\er}_V|\right):= \int\limits_{\Omega}|\bar{\er}_{T} - \bar{\er}_V(\omega)|d\IP(\omega) 
\end{equation}
Then, we have the following lemma on estimating the cumulative generalization error,
\begin{lemma}
\label{lem:11}
Let $\bu \in X^{\ast} \subset X$ be the unique solution of the PDE \eqref{eq:pde}. Let $\bu^{\ast} = \bu_{\theta^{\ast}}(\vtrain)$ be a PINN generated by algorithm \ref{alg:PINN}, with $N$ randomly chosen training points, identified by $\vtrain \in \dom^N$. Assume that there exists a constant $C_{pde}$ such that the stability bound \eqref{eq:assm2} is uniformly satisfied for PINNs, generated by the algorithm \ref{alg:PINN} for every training set $\train \subset \dom$, chosen independently and identically distributed with respect to the Lebesgue measure, then the cumulative generalization error \eqref{eq:ecgen} is bounded by, 
\begin{equation}
    \label{eq:ecgenb}
    \bar{\er}_G \leq C_{pde}\left( \bar{\er}_T + \er_{TV}
+ \frac{std(|\res_{\theta^{\ast}}|)}{\sqrt{N}}\right),
\end{equation}
with cumulative training error $\bar{\er}_T$ \eqref{eq:ectrain}, validation gap $\er_{TV}$ \eqref{eq:vgap} and 
\begin{equation}
\label{eq:std}
 std(|\res_{\theta^{\ast}}|):= \sqrt{\IE\left(\int_{\dom^N}|\res(z(\omega),\vtrain)| d\vtrain - \int_{\dom}\int_{\dom^N}|\res(z,\vtrain)| d\vtrain dz\right)^2}
\end{equation}
\end{lemma}
\begin{proof}
We use the shorthand notation $\res(\vtrain) = \res_{\theta^{\ast}}(\vtrain)$ for the residual \eqref{eq:res1}. By definition of the cumulative generalization error \eqref{eq:ecgen}, we have
\begin{align*}
    \bar{\er}_G &= \IE\left(\|\bu - \bu^{\ast}\|_{X}\right) \\
    &\leq C_{pde}\IE\left(\|\res(\vtrain)\|_{1}\right), \quad {\rm by}~\eqref{eq:assm2}, \\
    &\leq C_{pde} \IE\left(|\IE\left(\|\res(\vtrain)\|_{1}\right)-\bar{\er}_V + \bar{\er}_{V} -\bar{\er}_T+\bar{\er}_T|        \right) \\
    &\leq C_{pde}\left(\bar{\er}_T + \er_{TV} + \sqrt{\IE\left(|\IE(\|\res(\vtrain)\|_{1}) - \bar{\er}_{V}|^2\right)}\right)
\end{align*}
As the residual evaluated on the validation points is independent, the term inside the square root can be easily estimated in terms of the Monte Carlo quadrature to obtain the desired estimate \eqref{eq:ecgenb}.
\end{proof}
We remark that the estimate on the cumulative generalization error \eqref{eq:ecgenb} requires the computation of the so-called validation set. This is standard practice in machine learning, albeit with a smaller validation set than the training set. Usually, the validation set is chosen within a \emph{supervised learning} paradigm to monitor possible over-fitting.

However, in this case, choosing a validation set serves a different purpose. Our motivation in choosing a validation set is to evaluate an error estimate of the form \eqref{eq:ecgenb}. Note that as no explicit data (other than initial data and boundary conditions) is involved in the PINNs algorithm \ref{alg:PINN}, we can readily and cheaply evaluate the residual to compute the validation error and the validation gap \eqref{eq:vgap}. 

The error estimates \eqref{eq:egenb} (and \eqref{eq:ecgenb}) are for the abstract PDE formulation \eqref{eq:pde} and are meant to illustrate possible mechanisms for low approximation errors with PINNs. A lot of information is implicit in them, for instance the exact form of the function spaces $X,X^{\ast}, Y, Y^{\ast},Z,Z^{\ast}$ and (initial) boundary conditions. These will be made explicit in each of the subsequent concrete examples.
\section{Semi-linear Parabolic equations}
\label{sec:3}
\subsection{The underlying PDEs}
Let $D \subset \R^d$ be a domain i.e, an open connected bounded set with a $C^k$ boundary $\bD$. We consider the following model semi-linear parabolic equation,
\begin{equation}
    \label{eq:heat}
    \begin{aligned}
    u_t &= \Delta u + f(u), \quad \forall x\in D~ t \in (0,T), \\
    u(x,0) &= \bar{u}(x), \quad \forall x \in D, \\
    u(x,t) &= 0, \quad \forall x\in \bD, ~ t \in (0,T). 
    \end{aligned}
\end{equation}
Here, $u_0 \in H^{\bar{s}}(D;\R)$ is the initial data, $u \in H^s(((0,T)\times D);\R)$ is the solution and $f:\R \times \R$ is the non-linear source (reaction) term. We assume that the non-linearity is globally Lipschitz i.e, there exists a constant $C_f$ (independent of $v,w$) such that 
\begin{equation}
    \label{eq:assf}
    |f(v) - f(w)| \leq C_f|v-w|, \quad v,w \in \R.
\end{equation}
In particular, the homogeneous linear heat equation with $f(u) \equiv 0$ and the linear source term $f(u) = c_f u$ are examples of \eqref{eq:heat}. Semilinear heat equations with globally Lipschitz nonlinearities arise in several models in biology and finance \cite{Jent1}. 

The existence, uniqueness and regularity of the semi-linear parabolic equations with Lipschitz non-linearities such as \eqref{eq:heat} can be found in classical textbooks such as \cite{Frdbook}. For our purposes here, we will choose $\bar{s}>k +d/2 + 1$ such that the initial data $\bar{u} \in C^k(D)$ and we obtain $u \in C^k([0,T] \times D)$, with $k\geq 2$ as the classical solution of the semi-linear parabolic equation \eqref{eq:heat}. 
\subsection{PINNs}
In order to complete the PINNs algorithm \ref{alg:PINN}, we need to specify the training set $\train$ and define the appropriate residual, which we do below. 
\subsubsection{Training set}
\label{sec:tset}
Let $D_T = (0,T) \times D$ be the space-time domain. As in section \ref{sec:2}, we will choose the training set $\train \subset [0,T] \times \bar{D}$, based on suitable quadrature points. We have to divide the training set into the following three parts,
\begin{itemize}
    \item Interior training points $\train_{int}=\{y_n\}$ for $1 \leq n \leq N_{int}$, with each $y_n = (x,t)_n \in D_T$. These points can be the quadrature points, corresponding to a suitable space-time grid-based composite Gauss quadrature rule as long $d \leq 3$ or correspond to low-discrepancy sequences for moderately high dimensions or randomly chosen points in very high dimensions. See figure \ref{fig:2} for an example of randomly chosen training points for $d=1$. 
    \item Spatial boundary training points $\train_{sb} = \{z_n\}$ for $1 \leq n \leq N_{sb}$ with each $z_n = (x,t)_n$ and each $x_n \in \bD$. Again the points can be chosen from a grid-based quadrature rule on the boundary, as low-discrepancy sequences or randomly. 
    \item Temporal boundary training points $\train_{tb} = \{x_n\}$, with $1 \leq n \leq N_{tb}$ and each $x_n \in D$, chosen either as grid points, low-discrepancy sequences or randomly chosen in $D$.
\end{itemize}
The full training set is $\train = \train_{int} \cup \train_{sb} \cup \train_{tb}$. An example for the full training set is shown in figure \ref{fig:2}. 
\subsubsection{Residuals}
In the algorithm \ref{alg:PINN} for generating PINNs, we need to define appropriate residuals. For the neural network $u_{\theta} \in C^k([0,T]\times \bar{D})$, with continuous extensions of the derivatives to the boundaries, defined by \eqref{eq:ann1}, with a smooth activation function such as $\sigma = \tanh$ and $\theta \in \Theta$ as the set of tuning parameters, we define the following residual,   
\begin{itemize}
    \item Interior Residual given by,
    \begin{equation}
        \label{eq:hres1}
        \res_{int,\theta}(x,t):= \partial_t u_{\theta}(x,t) - \Delta u_{\theta}(x,t) - f(u_{\theta}(x,t)).
    \end{equation}
Here $\Delta = \Delta_x$ is the spatial Laplacian. Note that the residual is well defined and $\res_{int,\theta} \in C^{k-2}([0,T]\times \bar{D})$ for every $\theta \in \Theta$. 
\item Spatial boundary Residual given by,
\begin{equation}
    \label{eq:hres2}
    \res_{sb,\theta}(x,t):= u_{\theta}(x,t), \quad \forall x \in \bD, ~ t \in (0,T]. 
\end{equation}
Given the fact that the neural network is smooth, this residual is well defined. 
\item Temporal boundary Residual given by,
\begin{equation}
    \label{eq:hres3}
    \res_{tb,\theta}(x):= u_{\theta}(x,0) - \bar{u}(x), \quad \forall x \in D. 
\end{equation}
Again this quantity is well-defined and $\res_{tb,\theta} \in C^k(D)$ as both the initial data and the neural network are smooth. 
\end{itemize}
\subsubsection{Loss function}
As in section \ref{sec:2}, we need a loss function to train the PINN. To this end, we set the following loss function,
\begin{equation}
    \label{eq:hlf}
    J(\theta):= \sum\limits_{n=1}^{N_{tb}} w^{tb}_n|\res_{tb,\theta}(x_n)|^2 + \sum\limits_{n=1}^{N_{sb}} w^{sb}_n|\res_{sb,\theta}(x_n,t_n)|^2 + \lambda \sum\limits_{n=1}^{N_{int}} w^{int}_n|\res_{int,\theta}(x_n,t_n)|^2 .
\end{equation}
Here the residuals are defined by \eqref{eq:hres3}, \eqref{eq:hres2}, \eqref{eq:hres1}, $w^{tb}_n$ are the $N_{tb}$ quadrature weights corresponding to the temporal boundary training points $\train_{tb}$, $w^{sb}_n$ are the $N_{sb}$ quadrature weights corresponding to the spatial boundary training points $\train_{sb}$ and $w^{int}_n$ are the $N_{int}$ quadrature weights corresponding to the interior training points $\train_{int}$. Furthermore, $\lambda$ is a hyperparameter for balancing the residuals, on account of the PDE and the initial and boundary data, respectively. 
\subsection{Estimate on the generalization error.}
The algorithm \ref{alg:PINN} for training a PINN to approximate the semilinear parabolic equation \eqref{eq:heat} is now completely specified and can be run to generate the required PINN, which we denote as $u^{\ast} = u_{\theta^{\ast}}$, where $\theta^{\ast} \in \Theta$ is the (approximate) minimizer of the optimization problem, corresponding to the loss function \eqref{eq:lf2}, \eqref{eq:hlf}. 

We are interested in estimating the generalization error \eqref{eq:egen} for this PINN. In this case, $X = L^2(D \times (0,T))$ and the generalization error is concretely defined as,
\begin{equation}
    \label{eq:hegen}
    \er_{G}:= \left(\int\limits_0^T \int\limits_D |u(x,t) - u^{\ast}(x,t)|^2 dx dt \right)^{\frac{1}{2}}.
\end{equation}
As for the abstract PDE \eqref{eq:pde}, we are going to estimate the generalization error in terms of the \emph{training error} that we define as,
\begin{equation}
    \label{eq:hetrain}
    \er^2_{T}:= \underbrace{\sum\limits_{n=1}^{N_{tb}} w^{tb}_n|\res_{tb,\theta^{\ast}}(x_n)|^2}_{(\er_T^{tb})^2} + \underbrace{\sum\limits_{n=1}^{N_{sb}} w^{sb}_n|\res_{sb,\theta^{\ast}}(x_n,t_n)|^2}_{(\er_T^{sb})^2} +  \lambda\underbrace{\sum\limits_{n=1}^{N_{int}} w^{int}_n|\res_{int,\theta^{\ast}}(x_n,t_n)|^2}_{(\er_T^{int})^2}.
\end{equation}
Note that the training error can be readily computed \emph{a posteriori} from the loss function \eqref{eq:lf2},\eqref{eq:hlf}. 

We also need the following assumptions on the quadrature error, analogous to \eqref{eq:assm2}. For any function $g \in C^k(D)$, the quadrature rule corresponding to quadrature weights $w^{tb}_n$ at points $x_n \in \train_{tb}$, with $1 \leq n \leq N_{tb}$, satisfies 
\begin{equation}
    \label{eq:hquad1}
    \left| \int\limits_{D} g(x) dx - \sum\limits_{n=1}^{N_{tb}} w^{tb}_n g(x_n)\right| \leq C^{tb}_{quad}(\|g\|_{C^k}) N_{tb}^{-\alpha_{tb}}.
\end{equation}
For any function $g \in C^k(\bD \times [0,T])$, the quadrature rule corresponding to quadrature weights $w^{sb}_n$ at points $(x_n,t_n) \in \train_{sb}$, with $1 \leq n \leq N_{sb}$, satisfies 
\begin{equation}
    \label{eq:hquad2}
    \left| \int\limits_0^T \int\limits_{\bD} g(x,t) ds(x) dt - \sum\limits_{n=1}^{N_{sb}} w^{sb}_n g(x_n,t_n)\right| \leq C^{sb}_{quad}(\|g\|_{C^k}) N_{sb}^{-\alpha_{sb}}.
\end{equation}
Finally, for any function $g \in C^\ell(D \times [0,T])$, the quadrature rule corresponding to quadrature weights $w^{int}_n$ at points $(x_n,t_n) \in \train_{int}$, with $1 \leq n \leq N_{int}$, satisfies 
\begin{equation}
    \label{eq:hquad3}
    \left| \int\limits_0^T \int\limits_{D} g(x,t) dx dt - \sum\limits_{n=1}^{N_{int}} w^{int}_n g(x_n,t_n)\right| \leq C^{int}_{quad}(\|g\|_{C^\ell}) N_{int}^{-\alpha_{int}}.
\end{equation}
In the above, $\alpha_{int},\alpha_{sb},\alpha_{tb} > 0$ and in principle, different order quadrature rules can be used. We estimate the generalization error for the PINN in the following,
\begin{theorem}
\label{thm:heat}
Let $u \in C^k(\bar{D} \times [0,T])$ be the unique classical solution of the semilinear parabolic euqation \eqref{eq:heat} with the source $f$ satisfying \eqref{eq:assf}. Let $u^{\ast} = u_{\theta^{\ast}}$ be a PINN generated by algorithm \ref{alg:PINN}, corresponding to loss function \eqref{eq:lf2}, \eqref{eq:hlf}. Then the generalization error \eqref{eq:hegen} can be estimated as, 
\begin{equation}
    \label{eq:hegenb}
    \er_G \leq C_1 \left(\er_T^{tb}+\er_T^{int}+C_2(\er_T^{sb})^{\frac{1}{2}} + (C_{quad}^{tb})^{\frac{1}{2}}N_{tb}^{-\frac{\alpha_{tb}}{2}} +  (C_{quad}^{int})^{\frac{1}{2}}N_{int}^{-\frac{\alpha_{int}}{2}} + C_2  (C_{quad}^{sb})^{\frac{1}{4}}N_{sb}^{-\frac{\alpha_{sb}}{4}}             \right),
\end{equation}
with constants given by,
\begin{equation}
    \label{eq:hct}
    \begin{aligned}
        C_1 &= \sqrt{T + (1+2C_f)T^2e^{(1+2C_f)T}}, \quad C_2 = \sqrt{C_{\bD}(u,u^{\ast})T^{\frac{1}{2}}}, \\
        C_{\bD} &= |\bD|^{\frac{1}{2}}\left(\|u\|_{C^1([0,T] \times \bD)} + \|u^{\ast}\|_{C^1([0,T] \times \bD)}\right), \\
\end{aligned}
\end{equation}
and $C_{quad}^{tb} = C_{quad}^{tb}(\|\res^2_{tb,\theta^{\ast}}\|_{C^k})$, $C_{quad}^{sb} = C_{quad}^{tb}(\|\res^2_{sb,\theta^{\ast}}\|_{C^k})$ and $C_{quad}^{int} = C_{quad}^{int}(\|\res^2_{int,\theta^{\ast}}\|_{C^{k-2}})$ are the constants defined by the quadrature error \eqref{eq:hquad1}, \eqref{eq:hquad2}, \eqref{eq:hquad3}, respectively. 
\end{theorem}
\begin{proof}
By the definitions of the residuals \eqref{eq:hres1}, \eqref{eq:hres2}, \eqref{eq:hres3} and the underlying PDE \eqref{eq:heat}, we can readily verify that the error $\hat{u}: u^{\ast} - u$ satisfies the following (forced) parabolic equation,
\begin{equation}
    \label{eq:herr}
\begin{aligned}
    \hat{u}_t &= \Delta \hat{u} + f(u^{\ast}) - f(u) + \res_{int}, \quad \forall x\in D~ t \in (0,T), \\
    \hat{u}(x,0) &= \res_{tb}(x), \quad \forall x \in D, \\
    u(x,t) &= \res_{sb}(x,t), \quad \forall x\in \bD, ~ t \in (0,T). 
    \end{aligned}    
\end{equation}
Here, we have denoted $\res_{int} = \res_{int,\theta^{\ast}}$ for notational convenience and analogously for the residuals $\res_{tb},\res_{sb}.$

Multiplying both sides of the PDE \eqref{eq:herr} with $\hat{u}$, integrating over the domain and integrating by parts, denoting $\n$ as the unit outward normal, yields,
\begin{align*}
    \frac{1}{2} \frac{d}{dt}\int_{D}|\hat{u}(x,t)|^2 dx &=  - \int_{D} |\nabla \hat{u}|^2 dx + 
    \int_{\bD} \res_{sb}(x,t) (\nabla \hat{u}\cdot \n) ds(x) + \int_{D} \hat{u} (f(u^{\ast}) - f(u)) dx + \int_{D} \res_{int} \hat{u} dx. \\
    &\leq \int_{D} |\hat{u}||f(u^{\ast}) - f(u)| dx + \frac{1}{2} \int_{D} \hat{u}(x,t)^2 dx + \frac{1}{2}\int_{D} |\res_{int}|^2 dx \\
   &+  \underbrace{|\bD|^{\frac{1}{2}}\left(\|u\|_{C^1([0,T] \times \bD)} + \|u^{\ast}\|_{C^1([0,T] \times \bD)}\right)}_{C_{\bD}(u,u^{\ast})}\left(\int_{\bD} |\res_{sb}(x,t)|^2 ds(x) \right)^{\frac{1}{2}} \\
   &\leq (C_f+\frac{1}{2}) \int_{D} |\hat{u}(x,t)|^2 dx +   \frac{1}{2}\int_{D} |\res_{int}|^2 dx  + C_{\bD}(u,u^{\ast})\left(\int_{\bD} |\res_{sb}(x,t)|^2 ds(x) \right)^{\frac{1}{2}}~({\rm by}~\eqref{eq:assf}). 
\end{align*}
Integrating the above inequality over $[0,\bar{T}]$ for any $\bar{T} \leq T$ and the definition \eqref{eq:hres3} together with Cauchy-Schwarz inequality, we obtain,
\begin{align*}
    \int_{D}|\hat{u}(x,\bar{T})|^2 dx &\leq \int_{D} |\res_{tb}(x)|^2 dx + (1+C_f)\int_0^{\bar{T}} \int_{D} |\hat{u}(x,t)|^2 dx dt +  \int_0^T\int_{D} |\res_{int}|^2 dx dt \\
    &+ C_{\bD}(u,u^{\ast})T^{\frac{1}{2}}\left(\int_0^T\int_{\bD} |\res_{sb}(x,t)|^2 ds(x) dt \right)^{\frac{1}{2}}. 
\end{align*}
Applying the integral form of the Gr\"onwall's inequality to the above, we obtain,
\begin{align*}
    &\int_{D}|\hat{u}(x,\bar{T})|^2 dx \\
    &\leq \left(1 + (1+2C_f)Te^{(1+2C_f)T}\right)\left(\int_{D} |\res_{tb}(x)|^2 dx +\int_0^T\int_{D} |\res_{int}|^2 dx dt +  C_{\bD}(u,u^{\ast})T^{\frac{1}{2}}\left(\int_0^T\int_{\bD} |\res_{sb}(x,t)|^2 ds(x) dt \right)^{\frac{1}{2}} \right).
\end{align*}
Integrating over $\bar{T} \in [0,T]$ yields,
\begin{equation}
    \label{eq:hpf1}
    \begin{aligned}
        &\er^2_{G} = \int_0^T \int_{D}|\hat{u}(x,\bar{T})|^2 dx dt \\
        &\leq \left(T + (1+2C_f)T^2e^{(1+2C_f)T}\right)\left(\int_{D} |\res_{tb}(x)|^2 dx +\int_0^T\int_{D} |\res_{int}|^2 dx dt +  C_{\bD}(u,u^{\ast})T^{\frac{1}{2}}\left(\int_0^T\int_{\bD} |\res_{sb}(x,t)|^2 ds(x) dt \right)^{\frac{1}{2}} \right).
    \end{aligned}
\end{equation}
By the definitions of different components of the training error \eqref{eq:hetrain} and applying the estimates \eqref{eq:hquad1}, \eqref{eq:hquad2}, \eqref{eq:hquad3} on the quadrature error yields the desired inequality \eqref{eq:hegenb}.
\end{proof}
\begin{remark}
The estimate \eqref{eq:hegenb} bounds the generalization error in terms of each component of the training error and the quadrature errors. Clearly, each component of the training error can be computed from the loss function \eqref{eq:lf2}, \eqref{eq:hlf}, once the training has been completed. As long as each component of the PINN training error is small, the bound \eqref{eq:hegenb} implies that the generalization error will be small for large enough number of training points. Although, the estimate is not by any means sharp as triangle inequalities and the Gr\"onwall's inequality are used, information can be gleaned from it.  For instance, the error due to the boundary residual has a bigger weight in \eqref{eq:hegenb}, relative to the interior and initial residuals. This is consistent with the observations of \cite{Lag1} and can also be seen in the recent papers such as \cite{DAR1} and suggests that the loss function \eqref{eq:hlf} could be modified such that the boundary residual is penalized more. 
\end{remark}
\begin{remark}
In addition to the training errors, which could depend on the underlying PDE solution, the estimate \eqref{eq:hegenb} shows explicit dependence on the underlying solution through the constant $C_{\bD}$, which is based only on the value of the underlying solution on the boundary. Similarly, the dependence on dimension is only seen through the quadrature error. 
\end{remark}
\begin{remark}
As in subsection \ref{sec:mc}, one has to modify the estimate for the generalization error \eqref{eq:hegenb}, when random training points are used. Defining cumulative generalization error $\bar{\er}_G$, analogously to \eqref{eq:ecgen}, by integrating \eqref{eq:hegen} over all training sets $\train$, identified with the vector $\vtrain$ and the analogous concepts of cumulative training error $\bar{\er}_T$ \eqref{eq:ectrain} and  validation error $\bar{\er}_V$ \eqref{eq:ecval}, we can readily combine the arguments of Lemma \ref{lem:11} and Theorem \ref{thm:heat} to obtain the following estimate on the cumulative generalization error,
\begin{equation}
    \label{eq:hecgenb}
    \begin{aligned}
    \bar{\er}_G^2 &\leq C_1^2\left(\left(\bar{\er}^{tb}_T \right)^2 +\left(\er^{tb}_{TV}\right)^2 +\left(\bar{\er}^{int}_T \right)^2 +\left(\er^{tb}_{TV}\right)^2 +  C_2^2\left(\bar{\er}_{T}^{sb} + \bar{\er}_{TV}^{sb} \right) \right) \\
    &+  C_1^2\left(\frac{std\left(\res_{tb}^2\right)}{N_{tb}^{\frac{1}{2}}}+ \frac{std\left(\res_{int}^2\right)}{N_{int}^{\frac{1}{2}}} +C_2^2\frac{\sqrt{std\left(\res_{sb}^2\right)}}{N_{sb}^{\frac{1}{4}}}      \right),
 \end{aligned}
\end{equation}
with constants $C_{1,2}$ defined in \eqref{eq:hct}, standard deviation of the residuals, defined analogously to \eqref{eq:std} and validation gaps defined by 
$$
\left(\er^{\ell}_{TV}\right)^2 = \left |\left(\bar{\er}^{\ell}_T\right)^2 -  \left(\bar{\er}^{\ell}_V\right)^2\right|,\quad \ell=sb,tb,int.
$$

\end{remark}
\subsection{Numerical experiments.}
In this section, we present numerical experiments for the approximation of solutions of the parabolic equation \eqref{eq:heat} by PINNs, generated with algorithm \ref{alg:PINN}. We focus on the case of the linear heat equation here, i.e, by setting $f \equiv 0$ in \eqref{eq:heat} as one has explicit formulas for the underlying solution and we use these formulas to explicitly compute the generalization error. 
\subsubsection{One-dimensional Heat Equation.}
For the first numerical experiment, we consider the heat equation in one space dimension i.e, $d=1$ in \eqref{eq:heat} with $f\equiv 0$, domain $D=[-1,1]$, $T=1$, and initial data $\bar{u}(x) = -\sin(\pi x)$. Then, the exact solution of the heat equation is given by,
\begin{equation}
    \label{eq:hex1}
    u(x,t) = -\sin(\pi x) e^{-\pi^2 t}.
\end{equation}
Clearly both the initial data and the exact solution are smooth and Theorem \ref{thm:heat} holds. In particular, we can use any grid based quadrature rule to generate training points in algorithm \ref{alg:PINN}. However, to keep the presentation general and allow for very high-dimensional problems later, we simply choose random points for all three components of the training set $\train_{int},\train_{sb},\train_{tb}$. Each set of points is chosen randomly, independently and identically distributed with the underlying uniform distribution. An illustration of these points is provided in figure \ref{fig:2}, where the random points in $\train_{int}$ are shown as blue dots, whereas the points in $\train_{sb}$ and $\train_{tb}$ are shown as black crosses. 
\begin{figure}[htbp]
\centering
\includegraphics[width=0.7\textwidth]{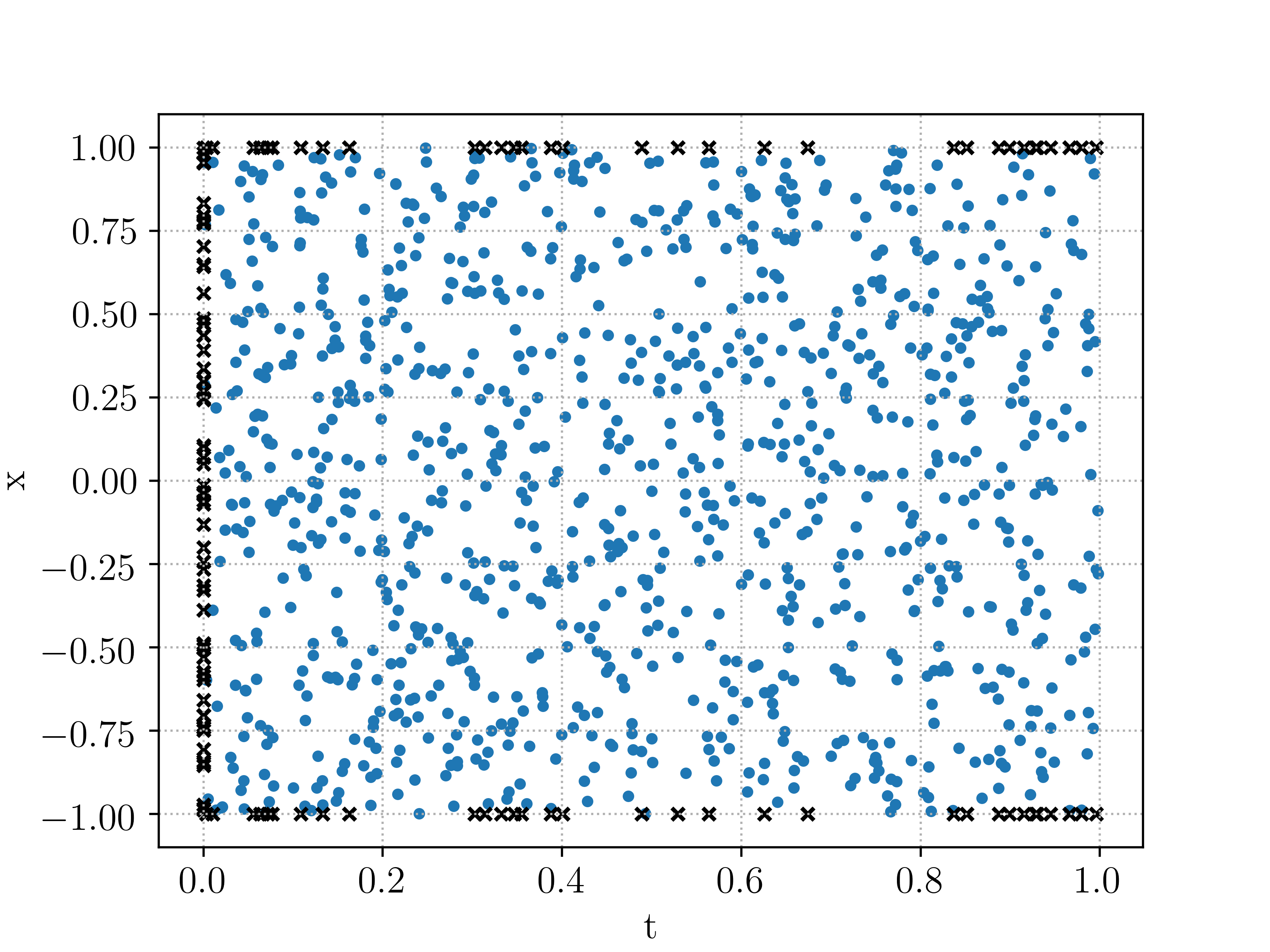}
\caption{An illustration of the training set $\train$ for the one-dimensional heat equation \eqref{eq:heat} with randomly chosen training points. Points in $\train_{int}$ are depicted with blue dots and those in $\train_{tb} \cup \train_{sb}$ are depicted with black crosses.}
\label{fig:2}
\end{figure}
Our first aim in this experiment is to illustrate the estimate \eqref{eq:hecgenb} on the generalization error. To this end, we run algorithm \ref{alg:PINN} with this random training set and with following hyperparameters: we consider a fully connected neural network architecture \eqref{eq:ann1}, with the $\tanh$ activation function, with $4$ hidden layers and $20$ neurons in each layer, resulting in neural networks with $1761$ tuning parameters. Moreover, we use the loss function \eqref{eq:lf2}, \eqref{eq:hlf}, with $\lambda=1$ and with $q=2$ i.e $L^2$-regularization, with regularization parameter $\lambda_{reg}=10^{-6}$. Finally, the optimizer is the second-order LBFGS method. This choice of hyperparameters is consistent with the ensemble training, presented in the next section. 

On this hyperparameter configuration, we vary the number of training points as $N_{int}=$ [1000, 2000, 4000, 8000, 16000], $N_{tb}=N_{sb}=$[8, 16, 32, 64, 256] concurrently, run the algorithm \ref{alg:PINN} to obtain the corresponding trained neural network and evaluate the resulting errors i.e, the cumulative training errors \eqref{eq:ectrain}, and the upper bound in \eqref{eq:hecgenb}. The cumulative generalization error \eqref{eq:ecgen} is computed by evaluating the error of the neural network with respect to the exact solution \eqref{eq:hex1} on a \emph{randomly chosen test set} of $10^5$ points. We compute the averages and standard deviations by taking $K=30$ different random training sets. The results for this procedure are shown in figure \ref{fig:hconv1}. We see from this figure that the cumulative generalization error \eqref{eq:ecgen} is very low to begin with and decays with the number of boundary training points $N_{tb}=N_{sb}$. The effect of the number of interior training points seems to be minimal in this case. Similarly, the computable upper bound \eqref{eq:hecgenb} also decays with respect to increasing the number of boundary training points. In particular, this shows that the constants and the standard deviations of the residuals in the bound \eqref{eq:ecgen} are not blowing up as the number of training points is increased. However, this upper bound does appear to be a significant overestimate as it is almost three orders to magnitude greater than the actual generalization error. This is not surprising as we had used non-sharp estimates such as triangle inequality and Gr\"onwall's inequality rather indiscriminately while deriving \eqref{eq:hecgenb}. Moreover, obtaining sharp bounds on generalization errors is a notoriously hard problem in machine learning \cite{NEYS1,AR1} with overestimates of tens of orders of magnitude. More importantly, both the computed generalization error and the upper bound follow the same decay in the number of training samples. Surprisingly, the training errors \eqref{eq:ectrain} are slightly larger than the computed generalization errors for this example. Note that this observation is still consistent with the bound \eqref{eq:hecgenb}. Given the fact that the training error is defined in terms of residuals and the generalization error is the error in approximating the solution of the underlying PDE by the PINN, there is no reason, a priori, to expect that the generalization error should be greater than the training error.  
\begin{figure}[htbp]
\centering
  \includegraphics[width=0.7\textwidth]{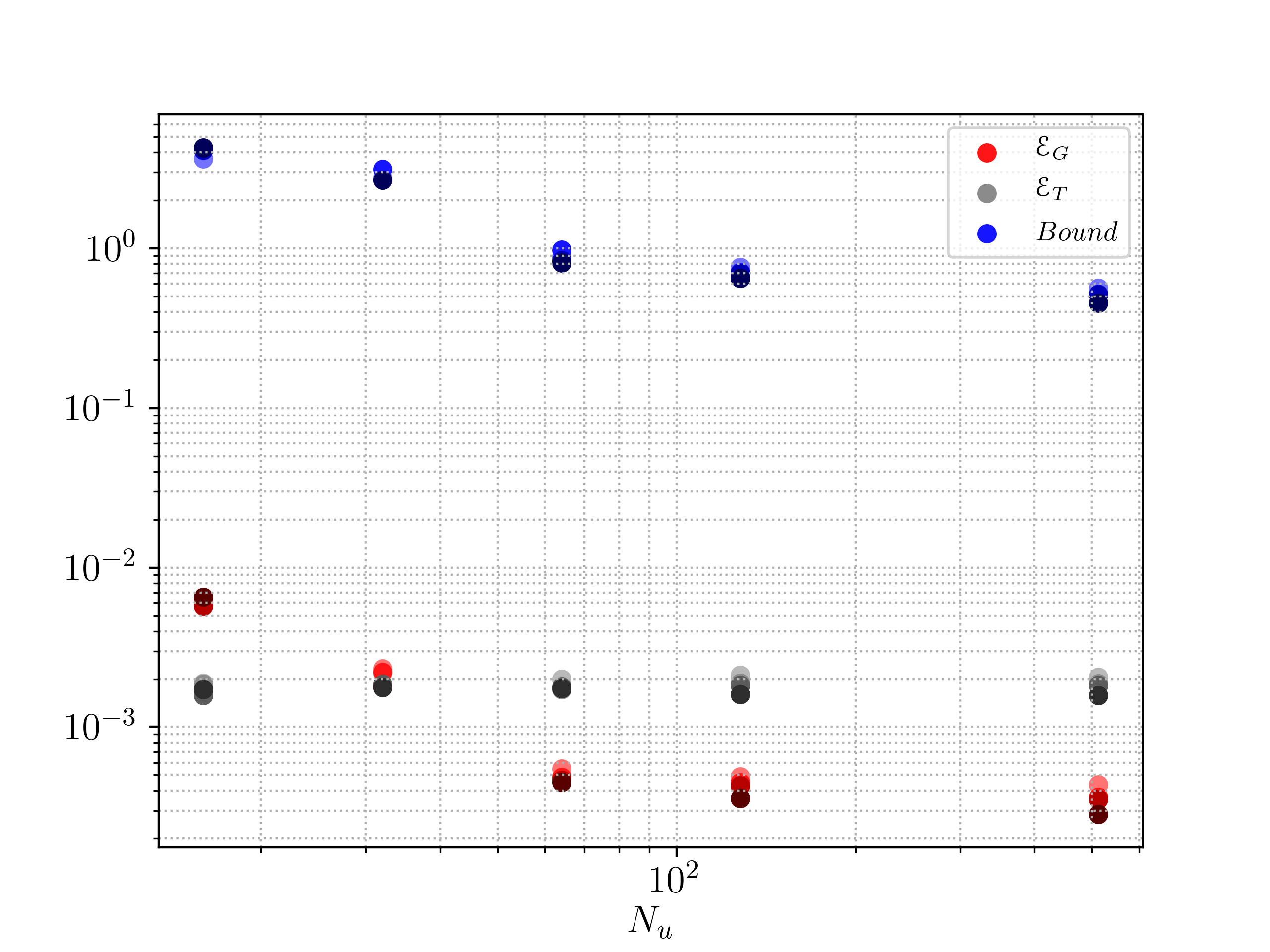}
  \caption{Generalization error, training error and theoretical bound \eqref{eq:hecgenb} VS number of training samples $N_u = N_{sb} + N_{tb}$. Each color gradation corresponds to different values of the number of interior training points $N_{int}$ (from the smallest to the largest).}
\label{fig:hconv1}
\end{figure}

\begin{table}[htbp] 
    \centering
    \renewcommand{\arraystretch}{1.1} 
    
    \footnotesize{
        \begin{tabular}{c  c c c c c } 
            \toprule
            \bfseries   &\bfseries $K-1$  & \bfseries $d$  &\bfseries $q$ & \bfseries $\lambda_{reg}$  &\bfseries $\lambda$\\ 
            \midrule
            \midrule
             1D Heat Equation & 2, 4, 8  & 12, 16, 20 &1, 2& 0, $10^{-6}$, $10^{-5}$, $10^{-4}$, $10^{-3}$& 0.01, 0.1, 1, 10\\
            \midrule
             ND Heat Equation & 4  & 20 & 2& 0, $10^{-6}$, $10^{-5}$& 0.1, 1, 10\\
            \midrule
            Burgers Equation     & 4, 8, 10   & 16, 20, 24 &2& 0, $10^{-6}$, $10^{-5}$& 0.1, 1, 10\\
            \midrule 
             Euler Equations, Taylor Vortex & 4, 8, 12   & 16, 20, 24 &1, 2& 0, $10^{-6}$, $10^{-5}$& 0.1, 1, 10\\
                \midrule 
             Euler Equations, Double Shear Layer & 16, 20, 24  & 32, 40, 48 &1, 2& 0, $10^{-6}$, $10^{-5}$& 0.1, 1, 10\\

            \bottomrule
        \end{tabular}
    \caption{Hyperparameter configurations employed in the ensemble training of PINNs:  $K-1$ is the number of hidden layers,  $d$ is the number of neurons per layer, $q$ and $\lambda_{reg}$ are the exponent in the regularization term \eqref{eq:reg} and the regularization parameter, respectively, $\lambda$ is the scalar balancing the role of  PDE and the data loss.}
        \label{tab:1}
    }
\end{table}
\subsubsection{Ensemble training}
A PINN involves several hyperparameters, some of which are shown in Table \ref{tab:1}. An user is always confronted with the question of which parameter to choose. The theory, presented in this paper and in the literature, offers very little guidance about the choice of hyperparameters. Instead, it is standard practice in machine learning to do a systematic hyperparameter search. To this end, we follow the \emph{ensemble training} procedure of \cite{LMR1}  with a randomly chosen training set ($N_{int} = 1024, N_{sb}=N_{tb}=64$) and compute the marginal distribution of the generalization error with respect to different choices of the hyperparameters (see Table \ref{tab:1}). The ensemble training results in a total number of 360 configurations. For each of them, the model is retrained five times with different starting values of the trainable weights in the optimization algorithm and the one resulting in the smallest value of the training loss is selected. We plot the corresponding histograms, visualizing the marginal generalization error distributions, in figure \ref{fig:hhist}.
\begin{figure}[htbp]
    \begin{subfigure}{.3\textwidth}
        \centering
        \includegraphics[width=1\linewidth]{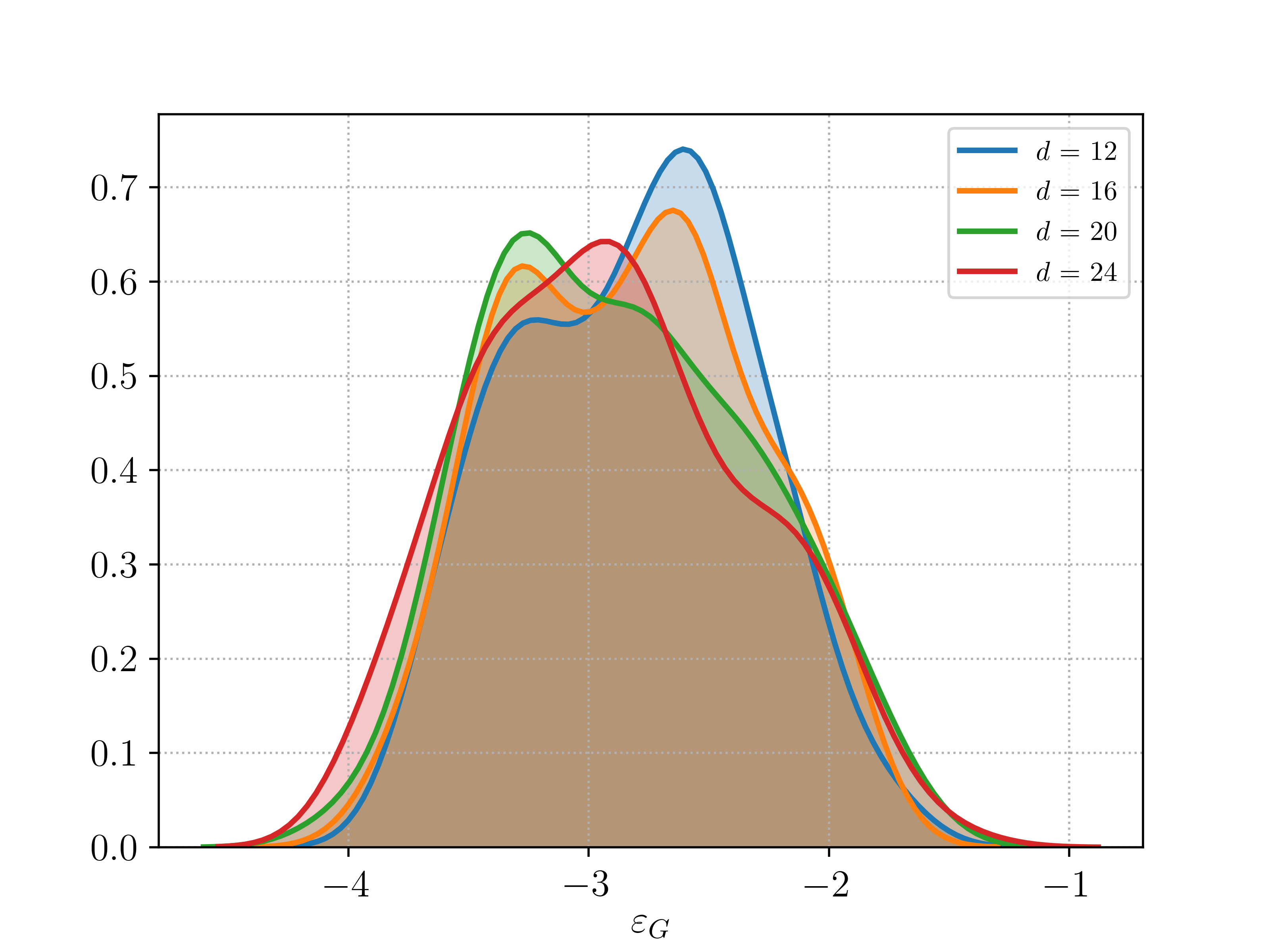}
        \caption{Number of neurons per layer $n$}
    \end{subfigure}
    \begin{subfigure}{.3\textwidth}
        \centering
        \includegraphics[width=1\linewidth]{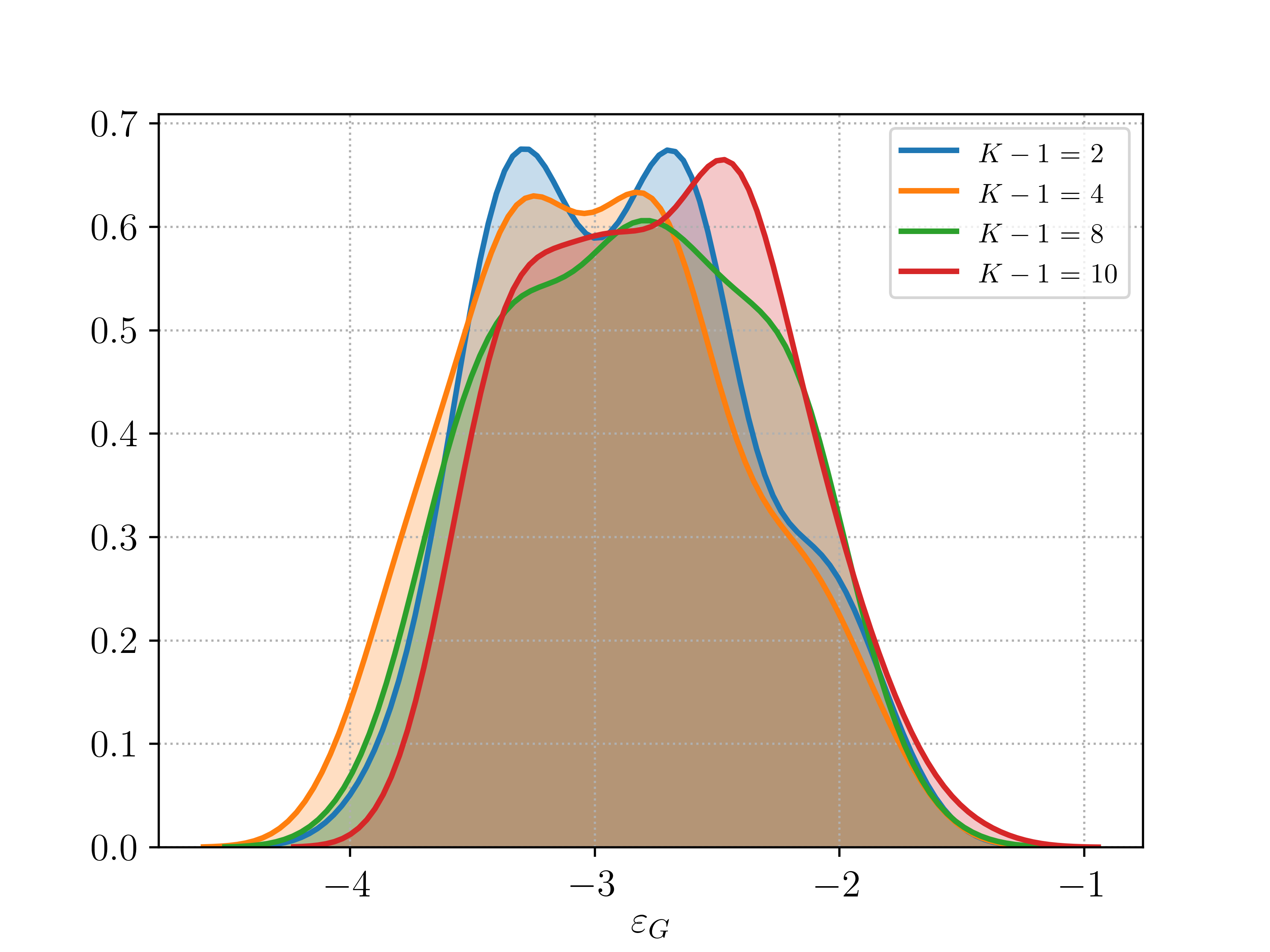}
        \caption{Number of hidden layers $K-1$}
    \end{subfigure}
 \begin{subfigure}{.3\textwidth}
        \centering
        \includegraphics[width=1\linewidth]{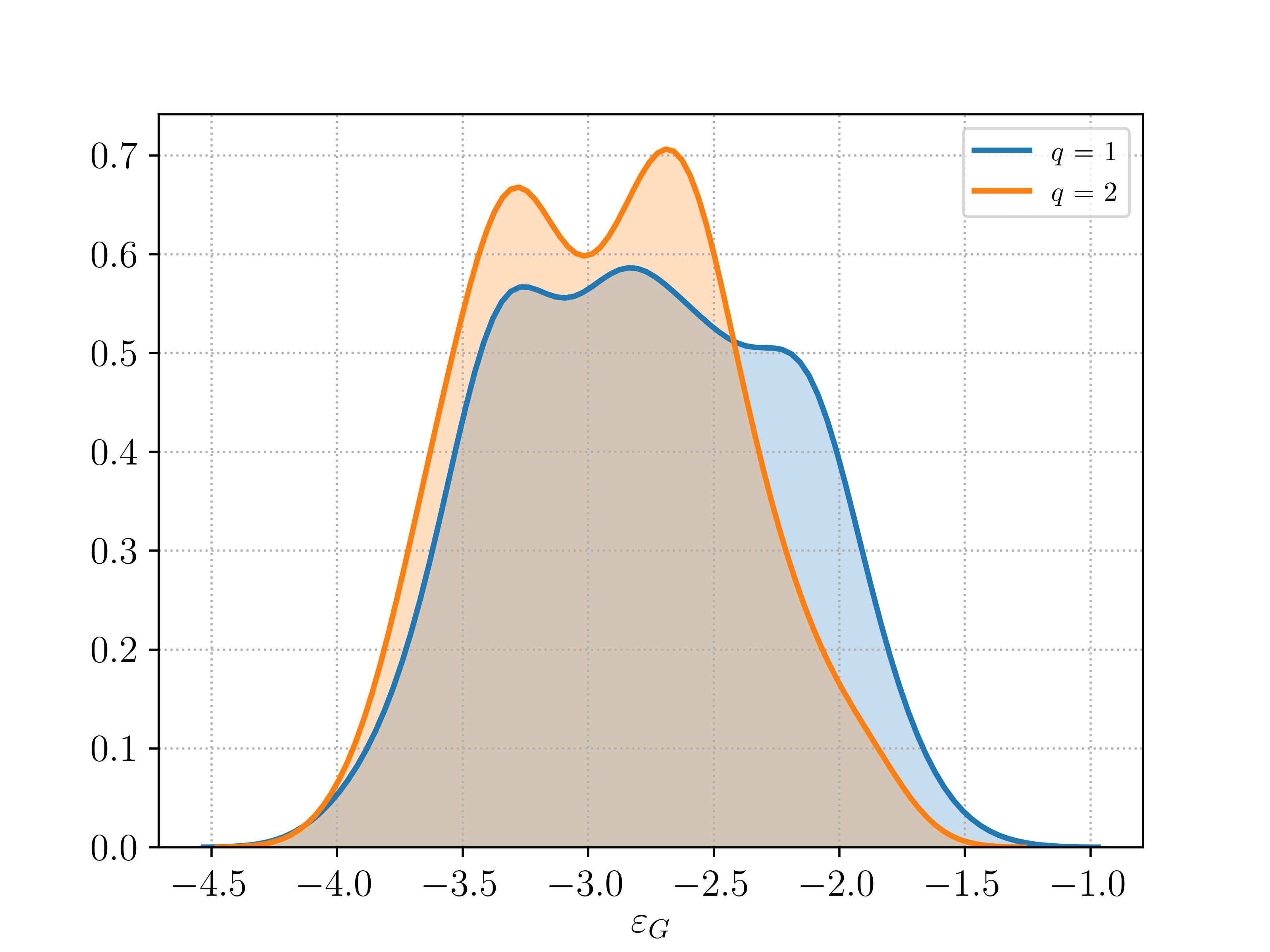}
        \caption{Regularization kernel $q$}
    \end{subfigure}
    
\centering
\begin{subfigure}{.3\textwidth}
        \centering
        \includegraphics[width=1\linewidth]{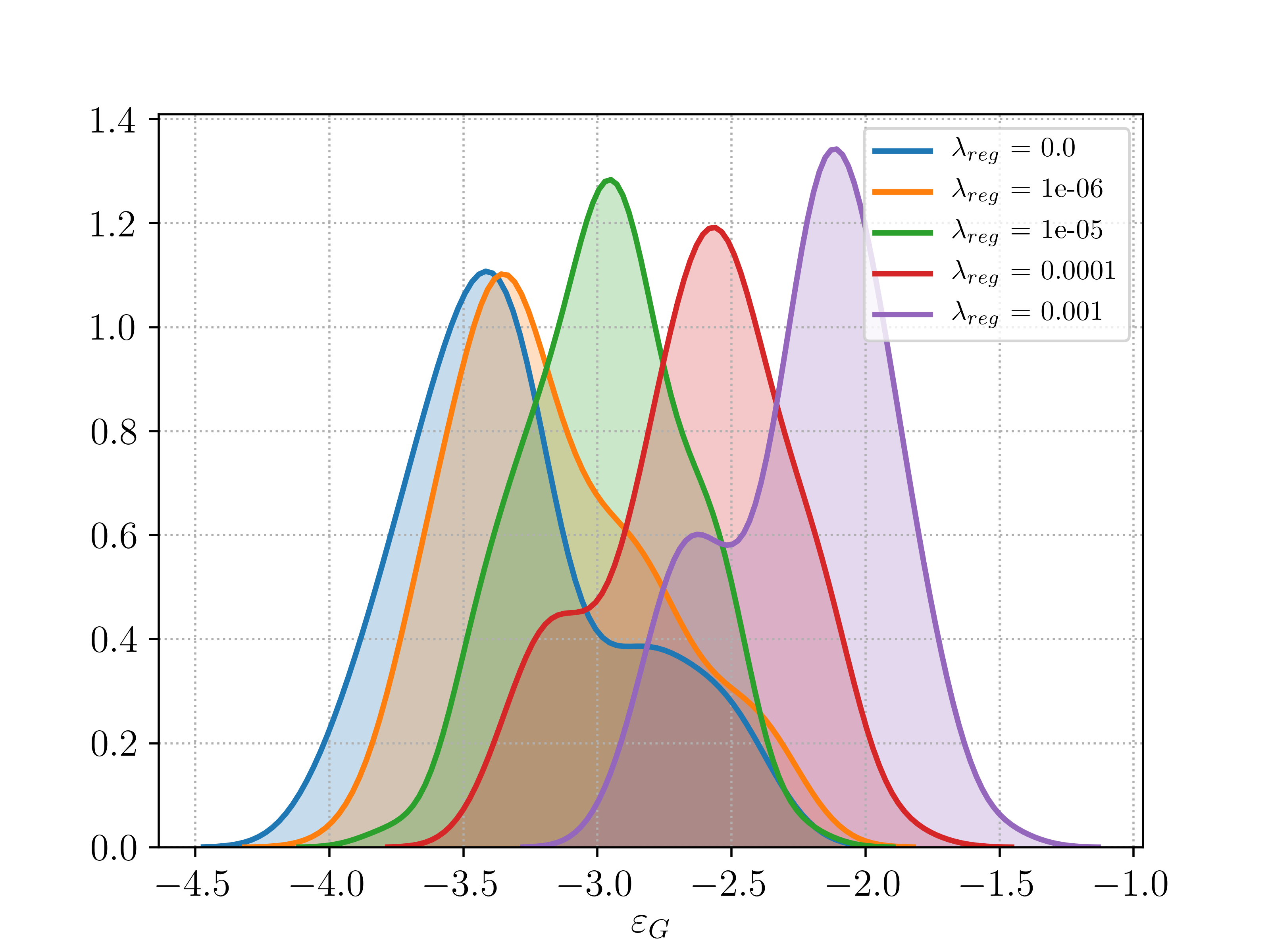}
        \caption{Regularization parameter $\lambda_{reg}$}
    \end{subfigure}
\begin{subfigure}{.3\textwidth}
        \includegraphics[width=1\linewidth]{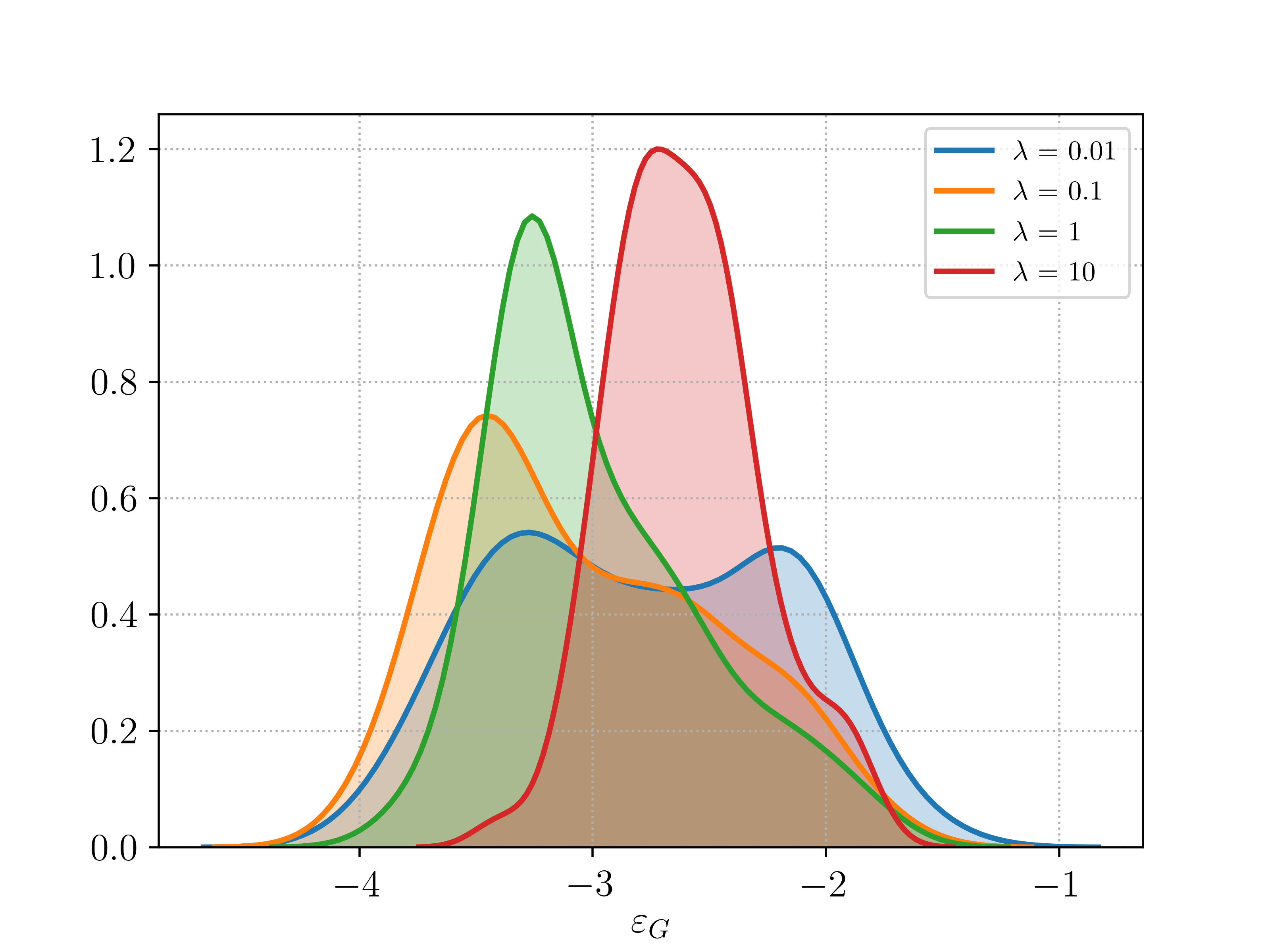}
        \caption{Residual parameter $\lambda$}
    \end{subfigure}
    \caption{Marginal distributions of the generalization error to different network hyperparameters}
    \label{fig:hhist}
\end{figure}
\par As seen from figure \ref{fig:hhist}, there is a large variation in the spread of the generalization error, often two to three orders of magnitude, indicating sensitivity to hyperparameters. However, even the worst case errors are fairly low for this example. Comparing different hyperparameters, we see that there is not much sensitivity to the network architecture (number of hidden layers and number of neurons per layer) and a slight regularization or no regularization in the loss function \eqref{eq:lf2} is preferable to large regularizations. The most sensitive parameter is $\lambda$ in \eqref{eq:hlf} where $\lambda = 1$ or $0.1$ are significantly better than larger values of $\lambda$. This can be explained in terms of the bounds \eqref{eq:hegenb}, \eqref{eq:hecgenb}, where the boundary residual $\res_{sb}$ has a larger weight in error. A smaller value of $\lambda$ enforces this component of error, and hence the overall error, to be small, and we see exactly this behavior in the results.

Finally, in figure \ref{fig:tg}, we plot the total training error \eqref{eq:hetrain} on a logarithmic scale (x-axis) against the generalization error \eqref{eq:hegen} (in log scale) (y-axis) for all the hyperparameter configurations in the ensemble training.  This plot clearly shows that the two errors are highly correlated and validates the fundamental point of the estimates \eqref{eq:hegenb} and \eqref{eq:hecgenb} that if the PINNs are trained well, they generalize very well. In other words, low training errors imply low generalization errors. Moreover, we see from figure \ref{fig:tg} that the generalization error scales as a square root of the total training error, as predicted by the estimate \eqref{eq:hecgenb}.
\subsubsection{Heat equation in several space dimensions}
For this experiment, we consider the linear heat equation in domain $[0,1]^n$ and for the time interval $[0,1]$, for different space dimensions $n$. We consider the initial data $\bar{u}(x) = \frac{\|x\|^2}{n}$. In this case, the explicit solution of the linear heat equation is given by \begin{equation}
    \label{eq:hex2}
    u(x,t)=\frac{\|x\|^2}{n} + 2t. 
\end{equation}
For different values of $n$ ranging up to $n=100$, we generate PINNs with algorithm \ref{alg:PINN}, by selecting randomly chosen training points, with respect to the underlying uniform distribution. Ensemble training, as outlined above, is performed in order to select the best performing hyperparameters among the ones listed in Table \ref{tab:1} and resulting errors are shown in \ref{tab:h2}. In particular, we present the relative percentage (cumulative) generalization error $\er^r_G$, readily computed from definition \eqref{eq:hegen} and normalized with the $L^2$-norm of the exact solution \eqref{eq:hex2}. We see from the table that the generalization errors are very low, less than $0.1\%$ upto 10 spatial dimensions and rise rather slowly (approximately linearly) with dimension, resulting in a low generalization error of $2.6\%$, even for $100$ space dimensions. This shows the ability of PINNs to possibly overcome the curse of dimensionality, at least for the heat equation. Note that we have not used \emph{any explicit solution formulas} such as the Feynman-Kac formulas in algorithm \ref{alg:PINN}. Still, PINNs were able to obtain low enough errors, comparable to supervised learning based neural networks that relied on the availability of an explicit solution formula \cite{HEJ1,Jent1}. This experiment illustrates the ability of PINNs to overcome the \emph{curse of dimensionality}, at least with random training points. 

\begin{table}[htbp] 
    \centering
    \renewcommand{\arraystretch}{1.1} 
    
    \footnotesize{
        \begin{tabular}{ c c c c c c c c c c c} 
            \toprule
            \bfseries $n$  &\bfseries $N_{int}$  & \bfseries $N_{sb}$& \bfseries $N_{tb}$  &\bfseries $K-1$ & \bfseries $d$  &\bfseries $L^1$-reg. &\bfseries $L^2$-reg. &\bfseries $\lambda$ &\bfseries $\er_T$ &\bfseries $\er_G^r$ \\ 
            \midrule
            \midrule
            1            & 65536 & 32768& 32768&4&20& 0.0& 0.0& 0.1 &$2.8\cdot 10^{-5}$& 0.0035\% \\
            \midrule 
              5  & 65536 & 32768& 32768&4 &20& 0.0& 0.0& 1.0 &0.0002& 0.016\%\\
                \midrule 
              10   & 65536 & 32768& 32768 &4 &20& 0.0& 0.0& 0.1 &0.0003& 0.03\%\\
                \midrule 
              20  & 65536 & 32768& 32768 &4 &20& 0.0& $10^{-6}$& 0.1 &0.006& 0.79\%\\
               \midrule 
              50  & 65536 & 32768& 32768 &4 &20& 0.0& $10^{-6}$& 0.1 &0.0056& 1.5\%\\
               \midrule 
              100  & 65536 & 32768& 32768 &4 &20& 0.0& 0.0& 0.1 &0.0035& 2.6\%\\

            \bottomrule
        \end{tabular}
    \caption{Best performing hyperparameters configurations for the multi-dimensional heat equation, for different values of the dimensions $n$.}
    \label{tab:h2}
    }
\end{table}

\begin{figure}[h!]
\centering
  \includegraphics[width=0.7\textwidth]{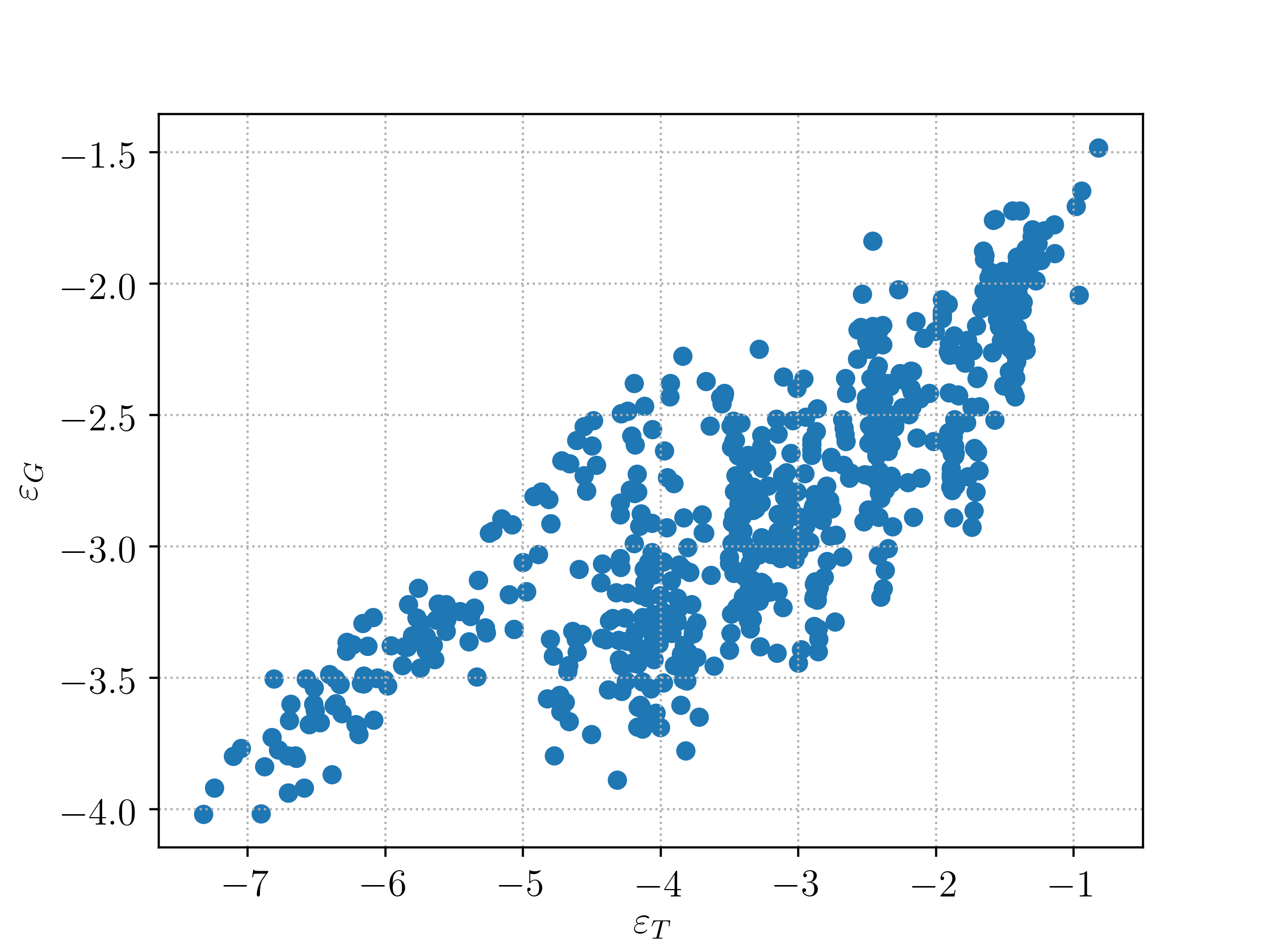}
  \caption{Log of Training error (X-axis) vs Log of  Generalization error (Y-axis) for each hyperparameter configuration during ensemble training. }
\label{fig:tg}
\end{figure}

\section{Viscous scalar conservation laws}
\label{sec:4}
\subsection{The underlying PDE}
In this section, we consider the following one-dimensional version of \emph{viscous scalar conservation laws} as a model problem for quasilinear, convection-dominated diffusion equations,
\begin{equation}
\label{eq:vscl}
\begin{aligned}
u_t + f(u)_x &= \nu u_{xx}, \quad \forall x\in (0,1),~t \in [0,T], \\
u(x,0) &= \bar{u}(x), \quad \forall x \in (0,1). \\
u(0,t) &= u(1,t) \equiv 0, \quad \forall t \in [0,T].
\end{aligned}
\end{equation}
Here, $\bar{u} \in C^k([0,1])$, for any $k \geq 1$, is the initial data and we consider zero Dirichlet boundary conditions. Note that $0 < \nu \ll 1$ is the viscosity coefficient. The flux function is denoted by $f\in C^k(\R;\R)$.

We emphasize that \eqref{eq:vscl} is a model problem that we present here for notational and expositional simplicity. The following results can be readily extended in the following directions:
\begin{itemize}
    \item Several space dimensions. 
    \item Other boundary conditions such as Periodic or Neumann boundary conditions. 
    \item More general forms of the viscous term, namely $\nu \left(B(u)u_x\right)_x$, for any $B \in C^k(\R;\R)$ with $B(v) \geq c > 0$, for all $v \in \R$ and for some $c$.
\end{itemize}

Moreover, we can follow standard textbooks such as \cite{GRbook} to conclude that as long as $\nu > 0$, there exists a classical solution $u \in C^k([0,T)\times [0,1])$ of the viscous scalar conservation law \eqref{eq:vscl}. 
\subsection{PINNs for \eqref{eq:vscl}}
We realize the abstract algorithm \ref{alg:PINN} in the following concrete steps,
\subsubsection{Training Set.}
Let $D =(0,1)$ and $D_T = (0,1) \times (0,T)$. As in section \ref{sec:3}, we divide the training set $\train = \train_{int} \cup \train_{sb} \cup \train_{tb}$ of the abstract PINNs algorithm \ref{alg:PINN} into the following three subsets, 
\begin{itemize}
    \item Interior training points $\train_{int}=\{y_n\}$ for $1 \leq n \leq N_{int}$, with each $y_n = (x_n,t_n) \in D_T$. These points can the quadrature points, corresponding to a suitable space -time grid-based composite Gauss quadrature rule or generated from a low-discrepancy sequence in $D_T$.
    \item Spatial boundary training points $\train_{sb} = (0,t_n) \cup (1,t_n)$ for $1 \leq n \leq N_{sb}$, and the points $t_n$ chosen either as Gauss quadrature points or low discrepancy sequences in $[0,T]$.
    \item  Temportal boundary training points $\train_{tb} = \{x_n\}$, with $1 \leq n \leq N_{tb}$ and each $x_n \in (0,1)$, chosen either as Gauss quadrature points of low-discrepancy sequences. 
\end{itemize}
\subsubsection{Residuals}
In the algorithm \ref{alg:PINN} for generating PINNs, we need to define appropriate residuals. For the neural network $u_{\theta} \in C^k([0,T]\times [0,1])$, defined by \eqref{eq:ann1}, with a smooth activation function such as $\sigma = \tanh$ and $\theta \in \Theta$ as the set of tuning parameters, we define the following residuals,   
\begin{itemize}
    \item Interior Residual given by,
    \begin{equation}
        \label{eq:bres1}
        \res_{int,\theta}(x,t):= \partial_t (u_{\theta}(x,t)) + \partial_x (f(u_{\theta}(x,t))) - \nu \partial_{xx} (u_{\theta}(x,t)).
    \end{equation}
\item Spatial boundary Residual given by,
\begin{equation}
    \label{eq:bres2}
    \left(\res_{sb,0,\theta}(t),~\res_{sb,1,\theta}(t)\right) := \left(u_{\theta}(0,t),~u_{\theta}(1,t)\right), \quad  \forall t \in (0,T]. 
\end{equation}
\item Temporal boundary Residual given by,
\begin{equation}
    \label{eq:bres3}
    \res_{tb,\theta}(x):= u_{\theta}(x,0) - \bar{u}(x), \quad \forall x \in [0,1]. 
\end{equation}
All the above quantities are well defined for $k \geq 2$ and $\res_{int,\theta} \in C^{k-2}([0,1] \times [0,T]),~ \res_{sb,\theta} \in C^k([0,T]),~ \res_{tb,\theta} \in C^k([0,1])$. 
\end{itemize}
\subsubsection{Loss function}
We use the following loss function to train the PINN for approximating the viscous scalar conservation law \eqref{eq:vscl},
\begin{equation}
    \label{eq:blf}
    J(\theta):= \sum\limits_{n=1}^{N_{tb}} w^{tb}_n|\res_{tb,\theta}(x_n)|^2 + \sum\limits_{n=1}^{N_{sb}} w^{sb}_n|\res_{sb,0,\theta}(t_n)|^2   + \sum\limits_{n=1}^{N_{sb}} w^{sb}_n|\res_{sb,1,\theta}(t_n)|^2  + \lambda \sum\limits_{n=1}^{N_{int}} w^{int}_n|\res_{int,\theta}(x_n,t_n)|^2 .
\end{equation}
Here the residuals are defined by \eqref{eq:bres3}, \eqref{eq:bres2}, \eqref{eq:bres1}. $w^{tb}_n$ are the $N_{tb}$ quadrature weights corresponding to the temporal boundary training points $\train_{tb}$, $w^{sb}_n$ are the $N_{sb}$ quadrature weights corresponding to the spatial boundary training points $\train_{sb}$ and $w^{int}_n$ are the $N_{int}$ quadrature weights corresponding to the interior training points $\train_{int}$. Furthermore, $\lambda$ is a hyperparameter for balancing the residuals, on account of the PDE and the initial and boundary data, respectively. 
\subsection{Estimate on the generalization error.}
As for the semilinear parabolic equation, we will try to estimate the following generalization error for the PINN $u^{\ast} = u_{\theta^\ast}$, generated through algorithm \ref{alg:PINN}, with loss functions \eqref{eq:lf2}, \eqref{eq:blf}, for approximating the solution of the viscous scalar conservation law \eqref{eq:vscl}:
\begin{equation}
    \label{eq:begen}
    \er_{G}:= \left(\int\limits_0^T \int\limits_0^1 |u(x,t) - u^{\ast}(x,t)|^2 dx dt \right)^{\frac{1}{2}}.
\end{equation}
This generalization error will be estimated in terms of the \emph{training error},
\begin{equation}
    \label{eq:betrain}
    \begin{aligned}
    \er^2_{T}&:=
    \lambda\underbrace{\sum\limits_{n=1}^{N_{int}} w^{int}_n|\res_{int,\theta^{\ast}}(x_n,t_n)|^2}_{(\er_T^{int})^2} +\underbrace{\sum\limits_{n=1}^{N_{tb}} + w^{tb}_n|\res_{tb,\theta^{\ast}}(x_n)|^2}_{(\er_T^{tb})^2} \\
    &+ \underbrace{\sum\limits_{n=1}^{N_{sb}} w^{sb}_n|\res_{sb,0,\theta^{\ast}}(t_n)|^2}_{(\er_T^{sb,0})^2} + \underbrace{\sum\limits_{n=1}^{N_{sb}} w^{sb}_n|\res_{sb,1,\theta^{\ast}}(t_n)|^2}_{(\er_T^{sb,1})^2},
\end{aligned}
\end{equation}
readily computed from the training loss \eqref{eq:blf} \emph{a posteriori}. We have the following estimate,
\begin{theorem}
\label{thm:burg}
Let $\nu > 0$ and let $u \in C^k((0,T) \times (0,1))$ be the unique classical solution of the viscous scalar conservation law \eqref{eq:vscl}. Let $u^{\ast} = u_{\theta^{\ast}}$ be the PINN, generated by algorithm \ref{alg:PINN}, with loss function \eqref{eq:blf}. Then, the generalization error \eqref{eq:begen} is bounded by,
\begin{equation}
    \label{eq:begenb}
    \begin{aligned}
    \er_G^2 &\leq \left(T + {\bf C}T^2e^{{\bf C}T}\right)\left[
    \left(\er^{tb}_{T}\right)^2 + \left(\er^{int}_{T}\right)^2
+     2\bar{{\bf C}}_b\left( \left(\er^{sb,0}_{T}\right)^2 +  \left(\er^{sb,1}_{T}\right)^2 \right) +  2\nu{\bf C}_b T^{\frac{1}{2}}\left(\er^{sb,0}_{T}+\er^{sb,1}_{T}\right)
    \right]\\
    &+\left(T + {\bf C}T^2e^{{\bf C}T}\right)\left[C^{tb}_{quad}N^{-\alpha_{tb}}_{tb} + C^{int}_{quad}N^{-\alpha_{int}}_{int} + 2\bar{{\bf C}}_b\left(\left(C^{sb,0}_{quad}+ C^{sb,1}_{quad}\right)N^{-\alpha_{sb}}_{sb}\right) +  2\nu{\bf C}_b T^{\frac{1}{2}}\left(\left(C^{sb,0}_{quad}+ C^{sb,1}_{quad}\right)^{\frac{1}{2}}N_{sb}^{-\frac{\alpha_{sb}}{2}} \right)\right].
    \end{aligned}
\end{equation}
Here, the training errors are defined by \eqref{eq:betrain} and the constants are given by ${\bf C} = 1 + 2 C_{f,u,u^{\ast}}$, with 
\begin{equation}
    \label{eq:bc1}
    \begin{aligned}
C_{f,u,u^{\ast}} &= C\left(\|f\|_{C^2},\|u\|_{W^{1,\infty}},\|u^{\ast}\|_{L^\infty} \right) = \left|f^{\prime\prime}\left(\max\{\|u\|_{L^\infty},\|u^{\ast}\|_{L^\infty}\}\right)\right|\|u_x\|_{L^{\infty}}, \\
{\bf C}_b &= \left(\|u_x\|_{C([0,1]\times[0,T])} + \|u^{\ast}_x\|_{C([0,1]\times[0,T])}\right), \\
\end{aligned}
\end{equation}
$\bar{{\bf C}}_b = \bar{{\bf C}}_b\left(\|f^{\prime}\|_{\infty},\|u^{\ast}\|_{C^0([0,1] \times [0,T])}\right)$ and $C^{tb}_{quad} = C^{tb}_{qaud}\left(\|\res^2_{tb,\theta^{\ast}}\|_{C^k}\right)$, $C^{int}_{quad} = C^{int}_{qaud}\left(\|\res^2_{int,\theta^{\ast}}\|_{C^{k-2}}\right)$, $C^{sb,0}_{quad} = C^{sb,0}_{qaud}\left(\|\res^2_{sb,0,\theta^{\ast}}\|_{C^k}\right)$, $C^{sb,1}_{quad} = C^{sb,1}_{qaud}\left(\|\res^2_{sb,1,\theta^{\ast}}\|_{C^k}\right)$ are constants are appear in the bounds on quadrature error \eqref{eq:hquad1}-\eqref{eq:hquad3}. 
\end{theorem}
\begin{proof}
We drop the $\theta^{\ast}$-dependence of the residuals \eqref{eq:bres1}-\eqref{eq:bres3} for notational convenience in the following. Define the \emph{entropy flux function},
\begin{equation}
    \label{eq:q}
    Q(u) = \int\limits_a^u s f^{\prime}(s) ds,
\end{equation}
for any $a \in \R$. Let $\hat{u} = u^{\ast}-u$ be the error with the PINN. From the PDE \eqref{eq:vscl} and the definition of the interior residual \eqref{eq:bres1}, we have the following identities,
\begin{equation}
    \label{eq:bpf1}
    \begin{aligned}
    \partial_t \left(\frac{(u^{\ast})^2}{2}\right) + \partial_x Q(u^{\ast}) &= \nu u^{\ast}u^{\ast}_{xx} + \res_{int} u^{\ast} \\ 
    \partial_t \left(\frac{u^2}{2}\right) + \partial_x Q(u) &= \nu uu_{xx}
    \end{aligned}
\end{equation}
A straightforward calculation with \eqref{eq:vscl} and \eqref{eq:bres1} yields,
\begin{equation}
    \label{eq:bpf2}
    \partial_t (u\hat{u}) + \partial_x\left(u\left(f(u^{\ast}) - f(u) \right)\right) = \left[f(u^{\ast}) - f(u) - f^{\prime}(u)\hat{u}\right]u_x +\res_{int} u + \nu \left(u\hat{u}_{xx} + \hat{u}u_{xx}\right).
\end{equation}
Subtracting the second equation of \eqref{eq:bpf1} and \eqref{eq:bpf2} from the first equation of \eqref{eq:bpf1} yields,
\begin{equation}
\label{eq:bpf3}
\partial_t S(u,u^{\ast}) + \partial_x H(u,u^{\ast}) = \res_{int} \hat{u} + T_1 + T_2,
\end{equation}
with, 
\begin{align*}
    S(u,u^{\ast})&:= \frac{(u^{\ast})^2}{2} - \frac{u^2}{2} - \hat{u}u = \frac{1}{2}\hat{u}^2,\\
    H(u,u^{\ast})&:= Q(u^{\ast}) - Q(u) - u(f(u^{\ast})-f(u)), \\
    T_1 &= -\left[f(u^{\ast}) - f(u) - f^{\prime}(u)\hat{u}\right]u_x, \\
    T_2 &= \nu \left(u^{\ast}u^{\ast}_{xx} - uu_{xx}-u\hat{u}_{xx} - \hat{u}u_{xx}\right) = \nu \hat{u} \hat{u}_{xx}.  
\end{align*}
As the flux $f$ is smooth, by a Taylor expansion, we see that 
\begin{equation}
    \label{eq:bpf4}
    T_1 = -f^{\prime \prime}(u + \gamma(u^{\ast}-u))\hat{u}^2 u_x,
\end{equation}
for some $\gamma \in (0,1)$. Hence, a straightforward estimate for $T_1$ is given by,
\begin{equation}
    \label{eq:bpf5}
    |T_1| \leq C_{f,u,u^{\ast}} \hat{u}^2, 
\end{equation}
with $C_{f,u,u^{\ast}}$ defined in \eqref{eq:bc1}. Next, we integrate \eqref{eq:bpf3} over the domain $(0,1)$ and integrate by parts to obtain,
\begin{equation}
    \label{eq:bpf6}
    \begin{aligned}
    \frac{d}{dt} \int_0^1 \hat{u}^2(x,t) dx &\leq 2 H(u(0,t),u^{\ast}(0,t)) - 2 H(u(1,t),u^{\ast}(1,t)) \\
    & + {\bf C} \int_0^1 \hat{u}^2(x,t) dx + \int_0^1 \res_{int}^2(x,t) dx, \\
    &-2\nu\int_0^1 \hat{u}^2_x(x,t) dx + 2\nu \left(\hat{u}(1,t)\hat{u}_x(1,t) - \hat{u}(0,t)\hat{u}_x(0,t)\right),
    \end{aligned}
\end{equation}
with the constant,
    ${\bf C} = 1 + 2 C_{f,u,u^{\ast}}$. 
    
Next, for any $\bar{T} \leq T$, we estimate the boundary terms starting with,
\begin{align*}
\int\limits_0^{\bar{T}} \hat{u}(0,t)\hat{u}_x(0,t) dt &= \int\limits_0^{\bar{T}} \res_{sb,0}(t) \left(u^{\ast}_x(0,t) - u_x(0,t)\right) dt \\
&\leq \underbrace{\left(\|u_x\|_{C([0,1]\times[0,T])} + \|u^{\ast}_x\|_{C([0,1]\times[0,T])}\right)}_{{\bf C}_b}T^{\frac{1}{2}}\left(\int_0^T 
\res_{sb,0}^2(t) dt \right)^{\frac{1}{2}}.
\end{align*}
Analogously we can estimate,
\begin{align*}
\int\limits_0^{\bar{T}} \hat{u}(1,t)\hat{u}_x(1,t) dt \leq {\bf C}_b T^{\frac{1}{2}}\left(\int_0^T 
\res_{sb,1}^2(t) dt \right)^{\frac{1}{2}}.
\end{align*}
We can also estimate from \eqref{eq:vscl} and \eqref{eq:hres2} that,
\begin{align*}
    H(u(0,t),u^{\ast}(0,t))&= Q(u^{\ast}(0,t)) - Q(u(0,t)) - u(0,t)(f(u^{\ast}(0,t)) - f(u(0,t))), \\
    &= Q(\res_{sb,0}(t)) - Q(0), \quad {\rm as}~ u(0,t) = 0, \\
    &= Q^{\prime}(\gamma_0\res_{sb,0}(t))\res_{sb,0}(t), \quad {\rm for~some}~\gamma_0 \in (0,1), \\
    &= \gamma_0f^{\prime}(\gamma_0u^{\ast}(0,t)) \res_{sb,0}^2(t), \quad {\rm by}~\eqref{eq:q}, \\ 
    &\leq \bar{{\bf C}}_b \res_{sb,0}^2(t),
\quad 
{\rm with}~ \bar{{\bf C}}_b = \bar{{\bf C}}_b\left(\|f^{\prime}\|_{\infty},\|u^{\ast}\|_{C^0([0,1] \times [0,T])}\right).
\end{align*}
Analogously, we can estimate,
\begin{align*}
   H(u(1,t),u^{\ast}(1,t)) \leq  \bar{{\bf C}}_b \res_{sb,1}^2(t).
\end{align*}
For any $\bar{T} \leq T$, integrating \eqref{eq:bpf6} over the time interval $[0,\bar{T}]$ and using the above inequalities on the boundary terms, together with the definition of the residual \eqref{eq:bres3} yields,
\begin{equation}
    \label{eq:bpf7}
    \begin{aligned}
    \int_0^1 \hat{u}^2(x,\bar{T}) dx &\leq \Cont + {\bf C}\int_0^{\bar{T}} \int_0^1 \hat{u}^2(x,t) dx dt, \\
    \Cont &= \int_0^1 \res_{tb}^2(x) dx + \int_0^T \int_0^1 \res_{int}^2(x,t) dx dt \\
    &+ 2 \bar{{\bf C}}_b\left[\int_0^T \res_{sb,0}^2(t) dt + \int_0^T \res_{sb,1}^2(t) dt\right] + 2\nu {\bf C}_b T^{\frac{1}{2}}\left[\left(\int_0^T \res_{sb,0}^2(t) dt\right)^{\frac{1}{2}} +\left(\int_0^T \res_{sb,1}^2(t) dt\right)^{\frac{1}{2}}\right] .
    \end{aligned}
\end{equation}
By applying the integral form of the Gr\"onwall's inequality to \eqref{eq:bpf7} for any $\bar{T} \leq T$ and integrating again over $\bar{T}$, together with the definition of the generalization error \eqref{eq:begen}, we obtain,
\begin{equation}
    \label{eq:bpf8}
\er_G^2 \leq \left(T + {\bf C}T^2e^{{\bf C}T}\right)\Cont.
\end{equation}
Using the bounds \eqref{eq:hquad1}-\eqref{eq:hquad3} on the quadrature errors and the definition of $\Cont$ in \eqref{eq:bpf7}, we obtain,
\begin{align*}
    \Cont &\leq \sum\limits_{n=1}^{N_{tb}} w^{tb}_n|\res_{tb}(x_n)|^2 + C^{tb}_{qaud}\left(\|\res_{tb}\|_{C^k}\right) N_{tb}^{-\alpha_{tb}} \\
    &+ \sum\limits_{n=1}^{N_{int}} w^{int}_n|\res_{int}(x_n,t_n)|^2 + C^{int}_{qaud}\left(\|\res_{int}\|_{C^{k-2}}\right) N_{int}^{-\alpha_{int}}, \\
    &+ 2\bar{{\bf C}}_b \left[\sum\limits_{n=1}^{N_{sb}} w^{sb}_n|\res_{sb,0}(t_n)|^2 + \sum\limits_{n=1}^{N_{sb}} w^{sb}_n|\res_{sb,1}(t_n)|^2 +  \left(C^{sb}_{qaud}\left(\|\res_{sb,0}\|_{C^k}\right)+ C^{sb}_{qaud}\left(\|\res_{sb,1}\|_{C^k}\right)\right) N_{sb}^{-\alpha_{sb}}             \right] \\
    &+ 2\nu{\bf C}_b T^{\frac{1}{2}}\left[\left(\sum\limits_{n=1}^{N_{sb}} w^{sb}_n|\res_{sb,0}(t_n)|^2\right)^{\frac{1}{2}} + \left(\sum\limits_{n=1}^{N_{sb}} w^{sb}_n|\res_{sb,1}(t_n)|^2 \right)^{\frac{1}{2}}+  \left(C^{sb}_{qaud}\left(\|\res_{sb,0}\|_{C^k}\right)+ C^{sb}_{qaud}\left(\|\res_{sb,1}\|_{C^k}\right)\right)^{\frac{1}{2}} N_{sb}^{-\frac{\alpha_{sb}}{2}}             \right].
\end{align*}
From definition of training errors \eqref{eq:betrain} and \eqref{eq:bpf8} and the above inequality, we obtain the desired estimate \eqref{eq:begenb}. 
\end{proof}
\begin{remark}
The estimate \eqref{eq:begenb} is a concrete realization of the abstract estimate \eqref{eq:egenb}, with training error decomposed into 4 parts, the constants, associated with the PDE, are given by $C_{f,u,u^{\ast}},{\bf C}_b$ \eqref{eq:bc1} and the constants due to the quadrature errors are also clearly delineated.
\end{remark}
\begin{remark}
\label{rem:burg}
A close inspection of the estimate \eqref{eq:begenb} reveals that at the very least, the classical solution $u$ of the PDE \eqref{eq:vscl} needs to be in $L^{\infty}((0,T);W^{1,\infty}((0,1)))$ for the rhs of \eqref{eq:begenb} to be bounded. This indeed holds as long as $\nu > 0$. However, it is well known (see \cite{GRbook} and references therein) that if $u^{\nu}$ is the solution of \eqref{eq:vscl} for viscosity $\nu$, then for some initial data,
\begin{equation}
    \label{eq:bbup}
    \|u^{\nu}\|_{L^{\infty}((0,T);W^{1,\infty}((0,1)))} \sim \frac{1}{\sqrt{\nu}}.
\end{equation}
Thus, in the limit $\nu \rightarrow 0$, the constant $C_{f,u,u^{\ast}}$ can blow up (exponentially in time) and the bound \eqref{eq:begenb} no longer controls the generalization error. This is not unexpected as the whole strategy of this paper relies on pointwise realization of residuals. However, the
zero-viscosity limit of \eqref{eq:vscl}, leads to a scalar conservation law with discontinuous solutions (shocks) and the residuals are measures that do not make sense pointwise. Thus, the estimate \eqref{eq:begenb} also points out the limitations of a PINN for approximating discontinuous solutions. 
\end{remark}
\subsection{Numerical experiments}
\begin{figure}[h!]
    \begin{subfigure}{.49\textwidth}
        \centering
        \includegraphics[width=1\linewidth]{{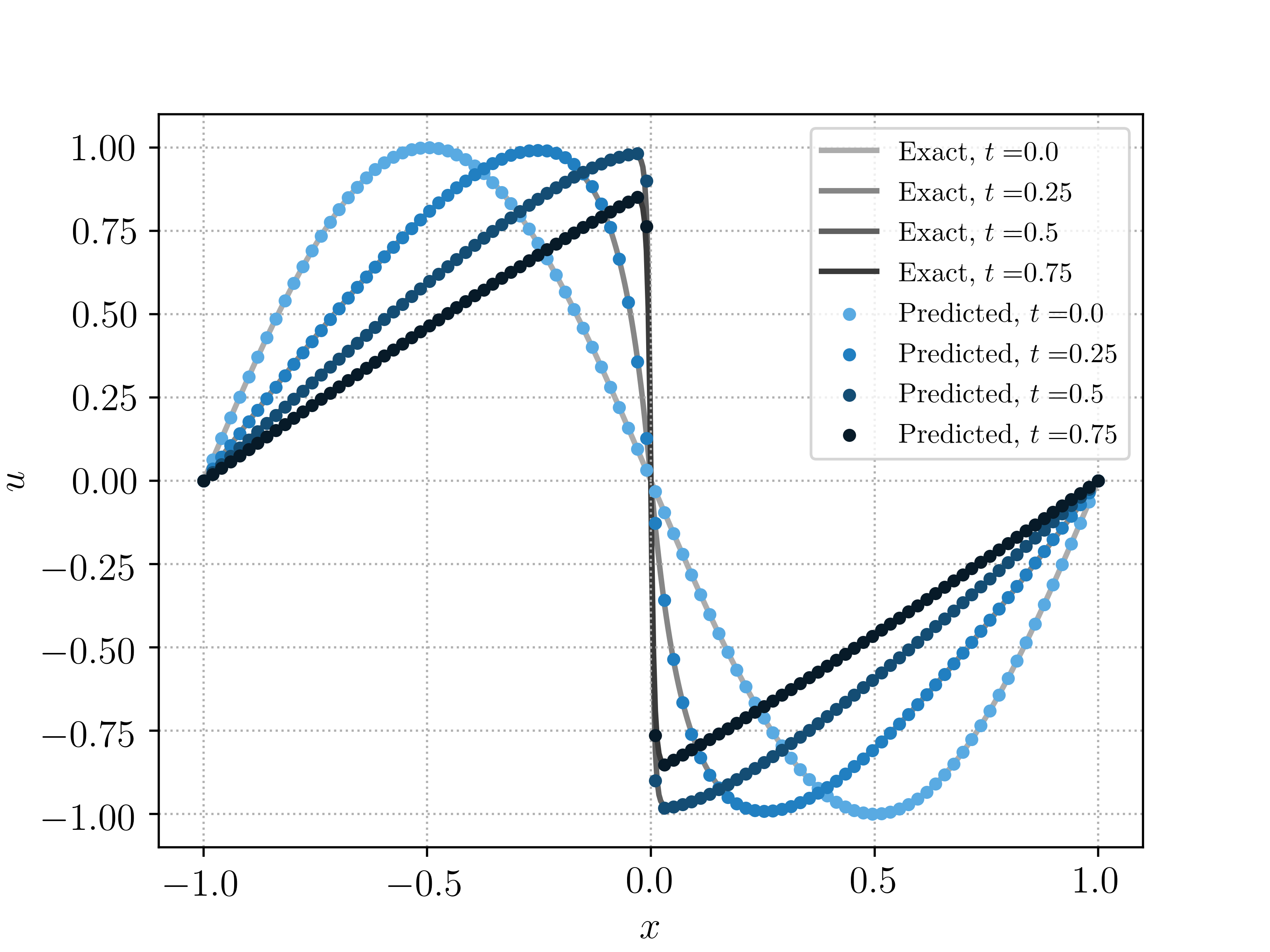}}
        \caption{$\nu = \frac{0.01}{\pi}$, $\er^r_G=0.010$}
    \end{subfigure}
    \begin{subfigure}{.49\textwidth}
        \centering\
        \includegraphics[width=1\linewidth]{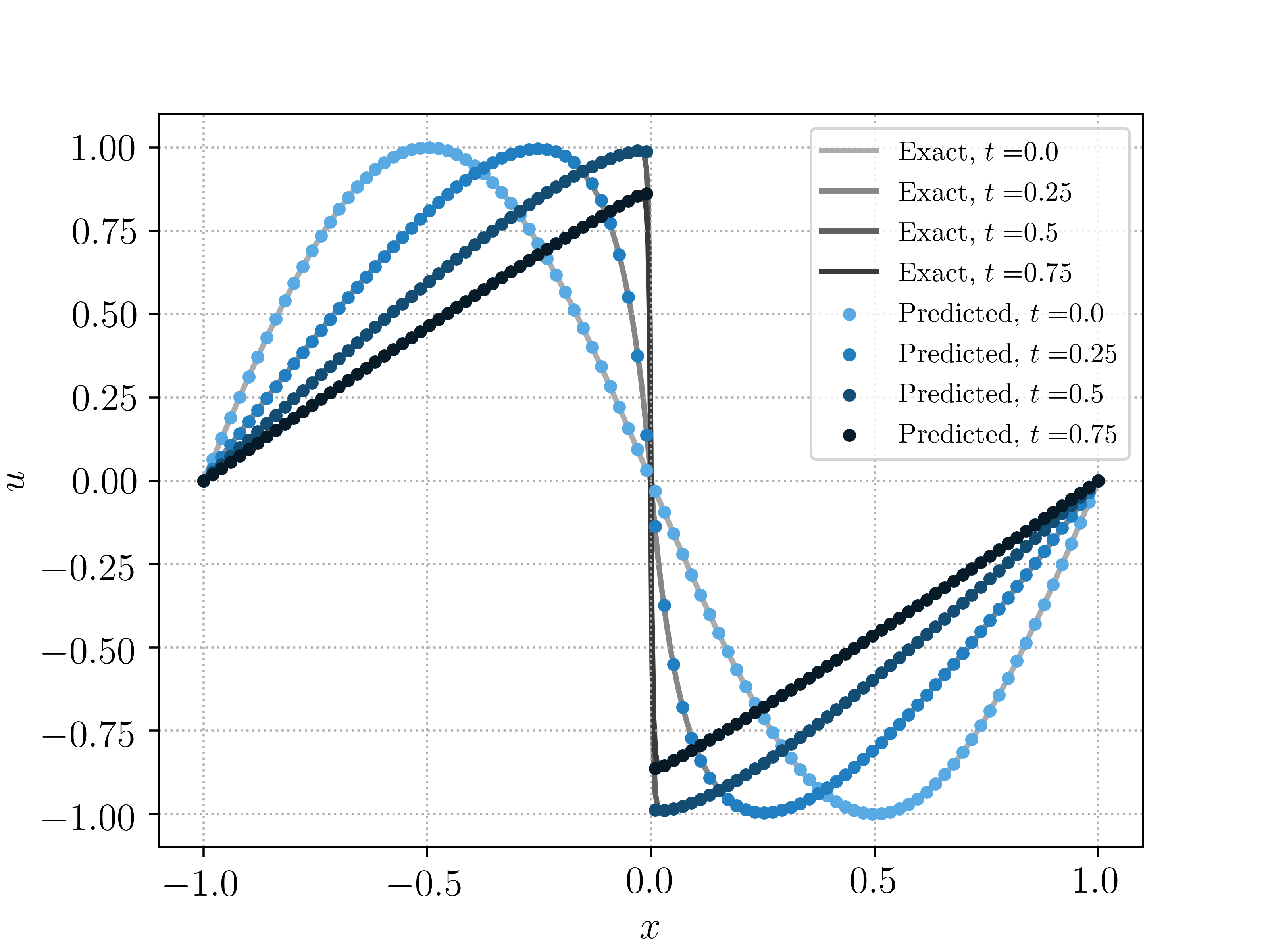}
        \caption{$\nu = \frac{0.005}{\pi}$, $\er^r_G=0.012$}
    \end{subfigure}
    
    \begin{subfigure}{.49\textwidth}
        \centering\
        \includegraphics[width=1\linewidth]{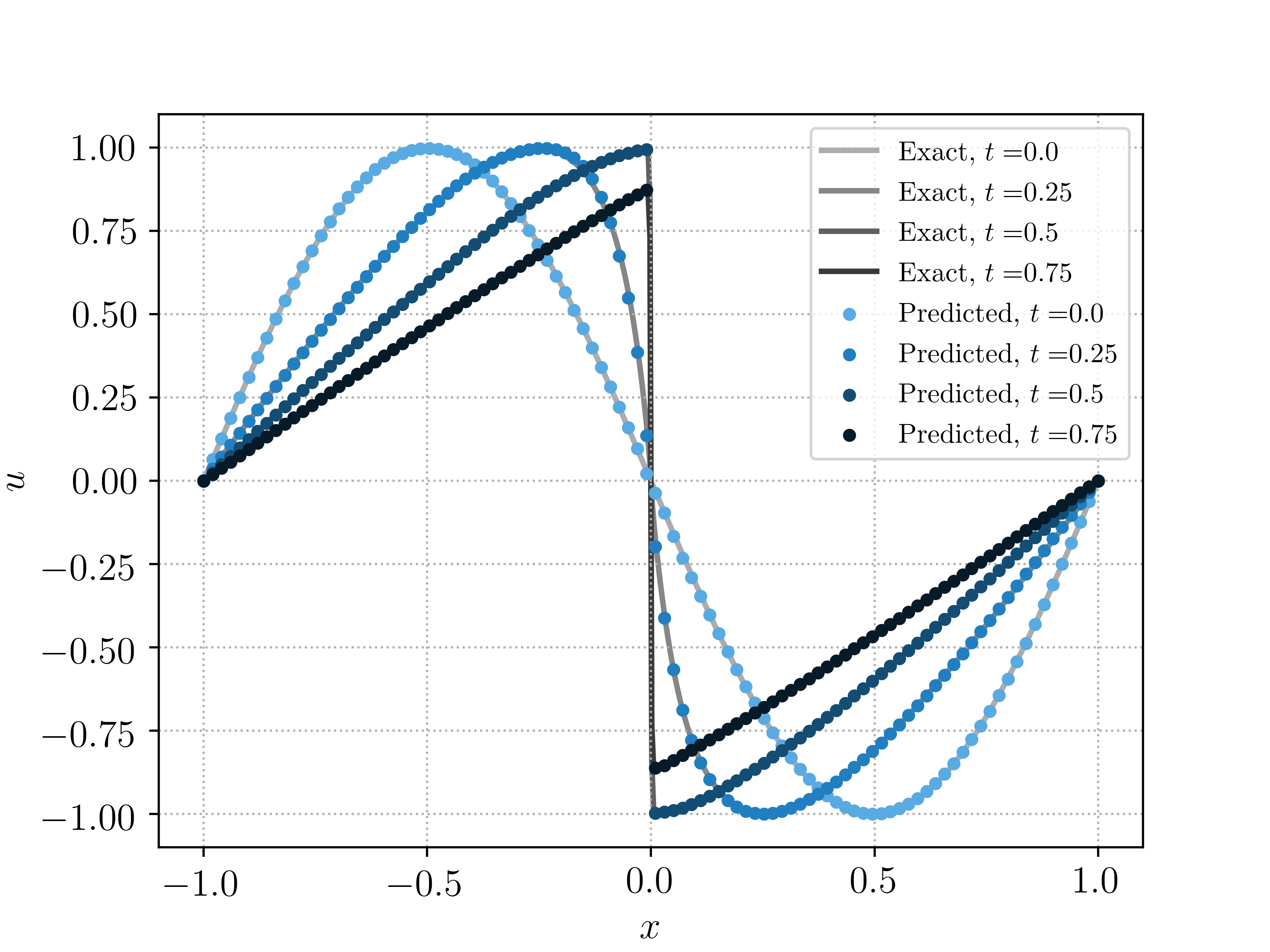}
        \caption{$\nu = \frac{0.001}{\pi}$, $\er^r_G=0.11$}
    \end{subfigure}
    \begin{subfigure}{.49\textwidth}
        \centering\
        \includegraphics[width=1\linewidth]{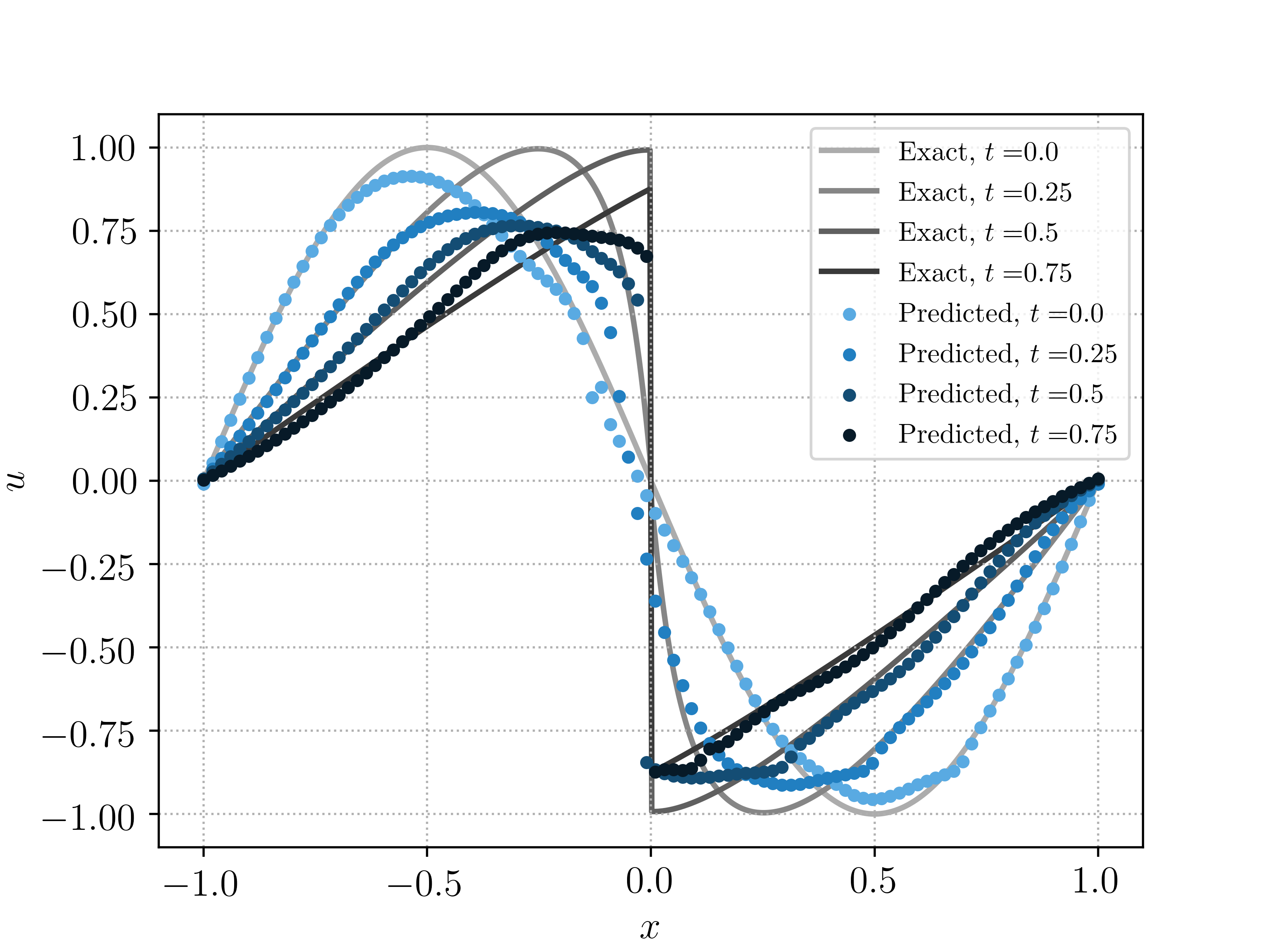}
        \caption{$\nu = 0$, $\er^r_G=0.23$}
    \end{subfigure}
    \caption{Burgers equation with discontinuos solution for different values of $\nu$}
    \label{fig:burg1}
\end{figure}
We consider the viscous scalar conservation law \eqref{eq:vscl}, but in the domain $D_T = [-1,1]\times[0,1]$, with initial conditions,
$
\bar{u}(x) = -\sin(\pi x)$ and zero Dirichlet boundary conditions. We choose the flux function $f(u) = \frac{u^2}{2}$, resulting in the well-known \emph{viscous Burgers'} equation. 

This problem is considered for $4$ different values of the viscosity parameter $\nu = \frac{c}{\pi}$, with $c= $ 0.01, 0.005, 0.001, 0.0, and with $N_{int}=8192$, $N_{tb}=256$, $N_{sb}=256$ points. All the training points are chosen as low-discrepancy Sobol sequences on the underlying domains. An ensemble training procedure, based on the hyperparameters presented in Table \ref{tab:1}, is performed and the best performing hyperparameters i.e. those that led to the smallest training errors, are chosen and presented in Table \ref{tab:burg}. 

In figure \ref{fig:burg1}, we present the \emph{reference} solution field $u(\cdot,t)$, at different time snapshots, of the viscous Burgers' equation computed with a simple upwind finite volume scheme and forward Euler time integration with $2\times10^{6}$ Cartesian  grid points in space-time, and the predicted solution $u^{\ast}(\cdot,t)$ of the PINN, generated with algorithm \ref{alg:PINN}, corresponding to the best performing hyperparameters (see Table \ref{tab:burg}), for different values of the viscosity coefficient. From this figure, we observe that for the viscosity coefficients, corresponding to $c=0.01,0.005$, the approximate solution, predicted by the PINN, approximates the underlying exact solution, which involves self steepening of the initial sine wave into a steady sharp profile at the origin, very well. This is further reinforced by the very low (relative percentage) generalization errors of approximately $1\%$, presented in Table \ref{tab:burg}. However, this efficient approximation is no longer the case for the inviscid problem i.e $\nu = 0$. As seen from figure \ref{fig:burg1} (bottom right), the PINN fails to resolve the solution, which in this case, consists of a steady shock at the origin. In fact, this failure to approximate is already seen with a viscosity coefficient of $\nu = \frac{0.001}{\pi}$. For this very low viscosity coefficient, we see that the relative generalization error is approximately $11 \%$. The generalization error rises to $23 \%$ for the inviscid Burgers' equation. This increase in error appears consistent with the bound \eqref{eq:begenb}, combined with the blow-up estimate \eqref{eq:bbup} for the derivatives of the viscous Burgers' equation. As the viscosity $\nu \rightarrow 0$, the bounds in the rhs of \eqref{eq:begenb} can increase exponentially, which appears to be the case here.

\begin{figure}[h!]
    \begin{subfigure}{.49\textwidth}
        \centering
        \includegraphics[width=1\linewidth]{{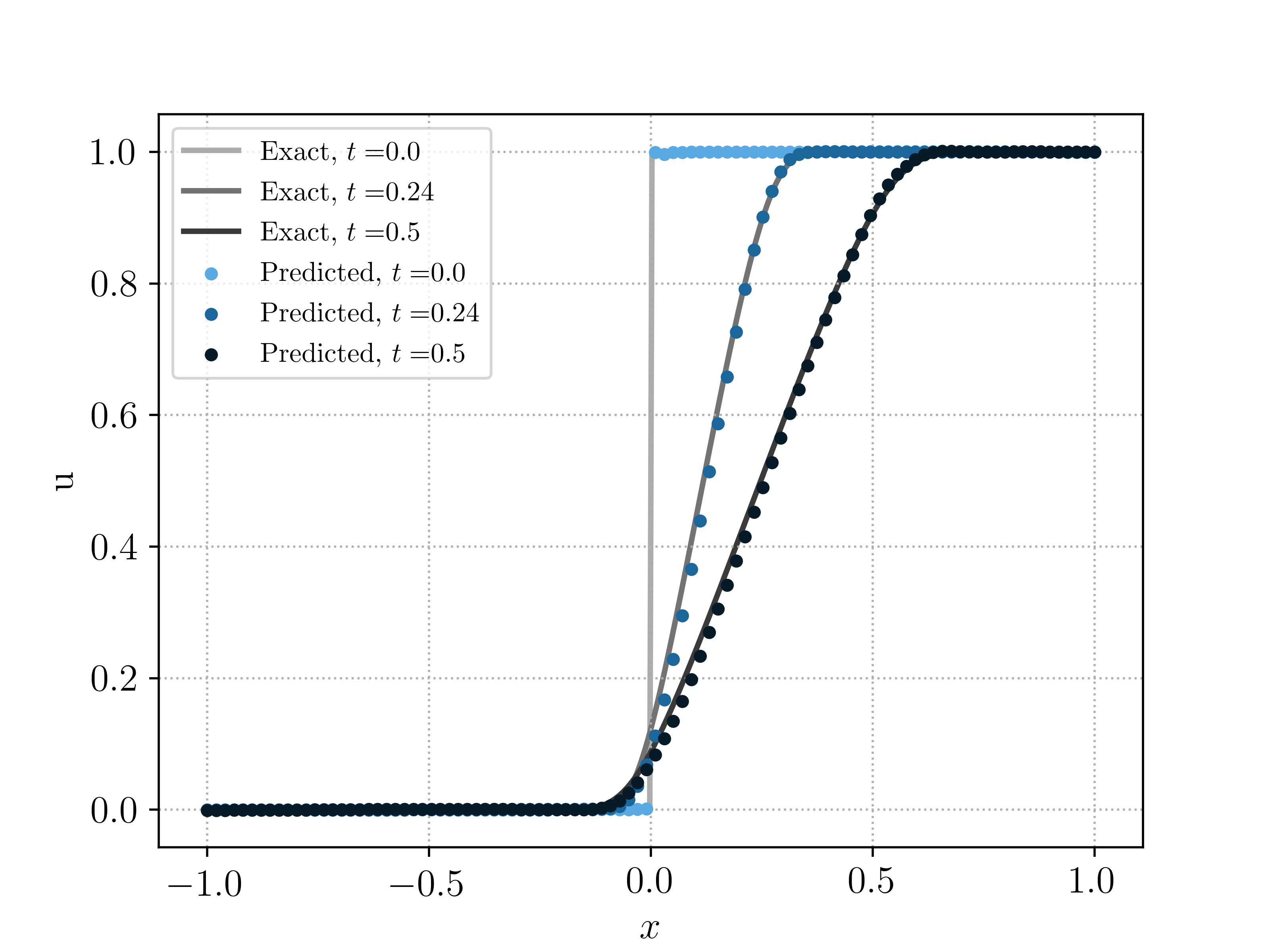}}
        \caption{$\nu = \frac{0.01}{\pi}$, $\er^r_G=0.022$}
    \end{subfigure}
    \begin{subfigure}{.49\textwidth}
        \centering\
        \includegraphics[width=1\linewidth]{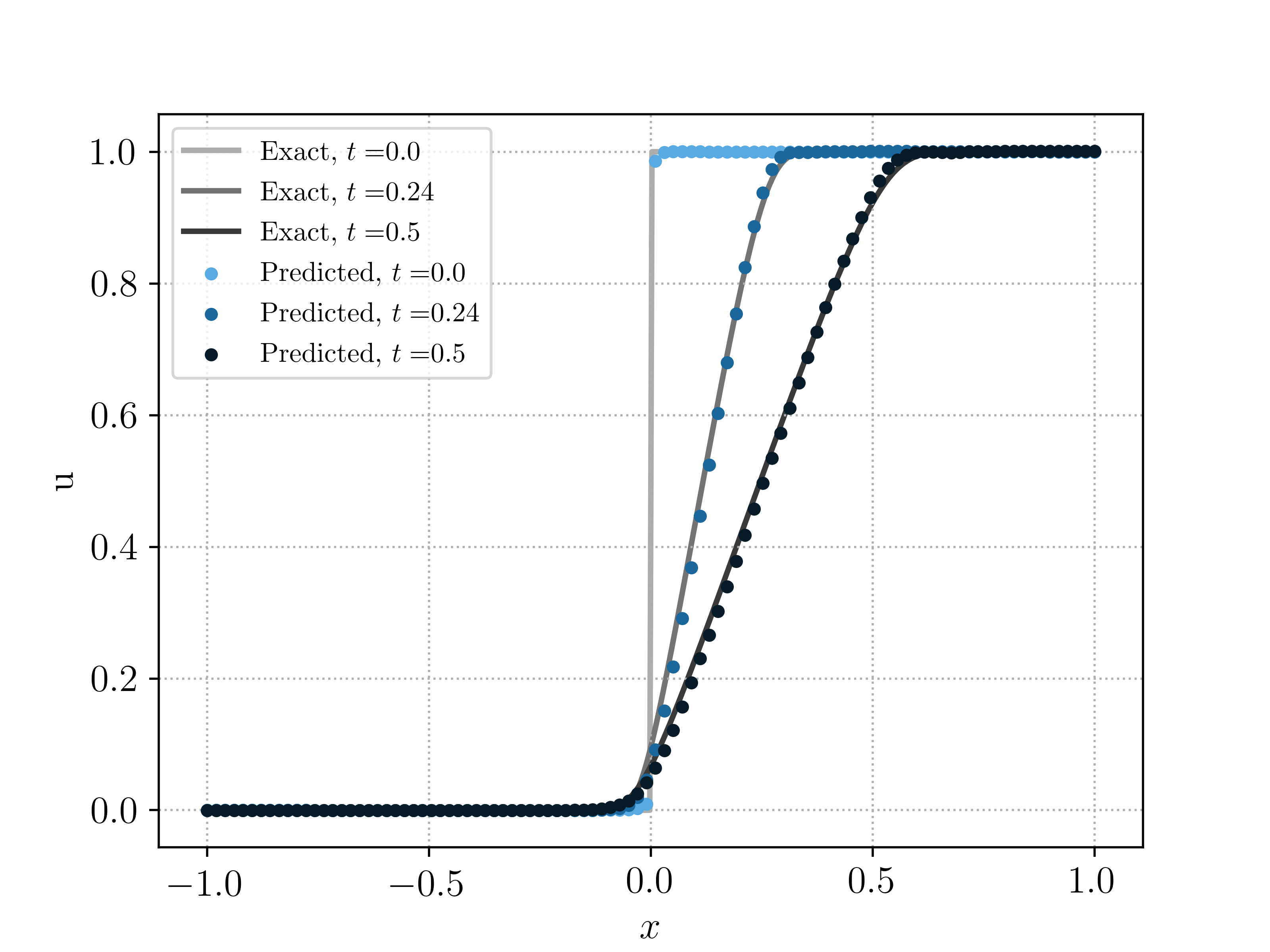}
        \caption{$\nu = \frac{0.005}{\pi}$, $\er^r_G=0.018$}
    \end{subfigure}
    
    \begin{subfigure}{.49\textwidth}
        \centering\
        \includegraphics[width=1\linewidth]{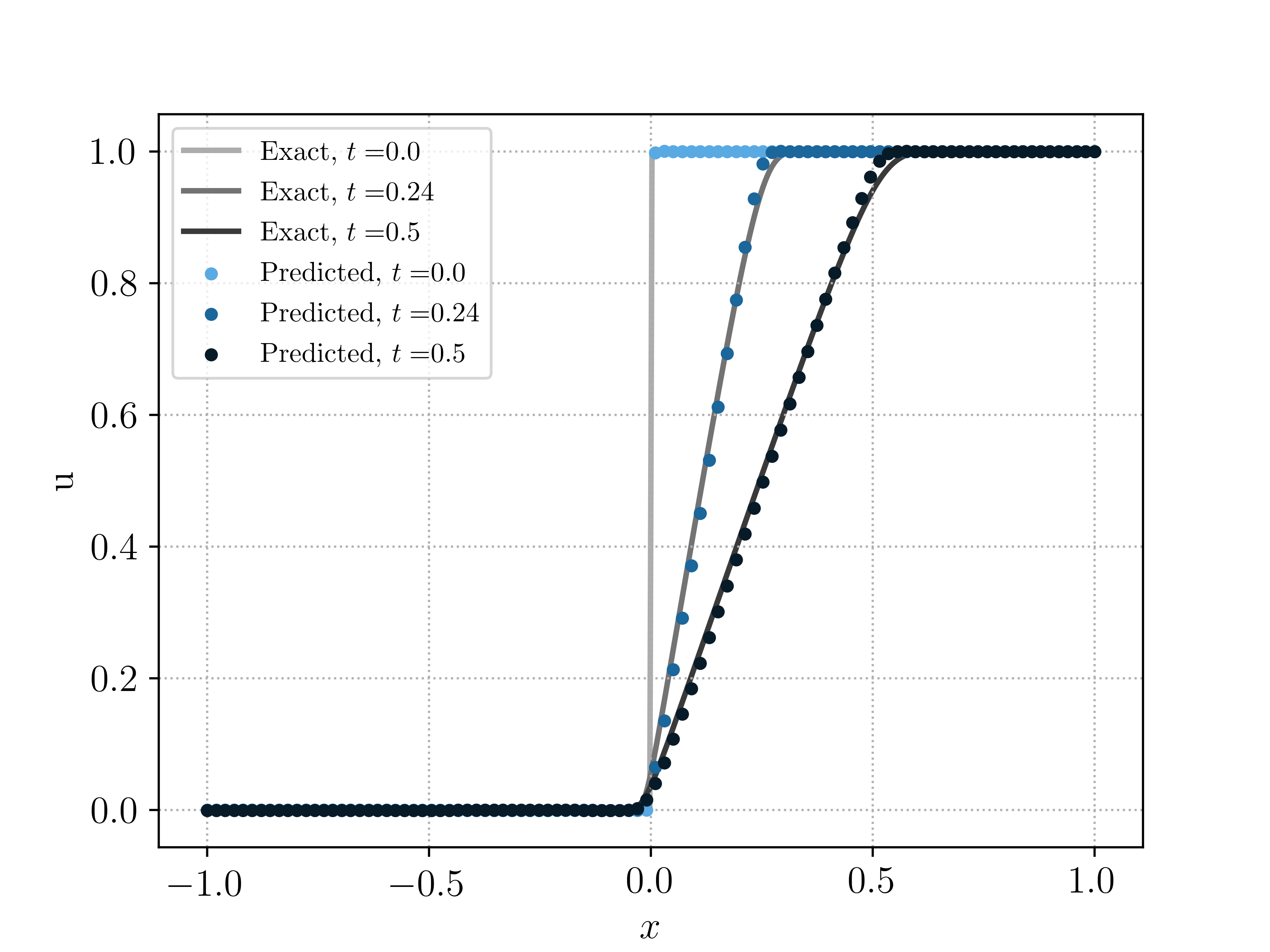}
        \caption{$\nu = \frac{0.001}{\pi}$, $\er^r_G=0.016$}
    \end{subfigure}
    \begin{subfigure}{.49\textwidth}
        \centering\
        \includegraphics[width=1\linewidth]{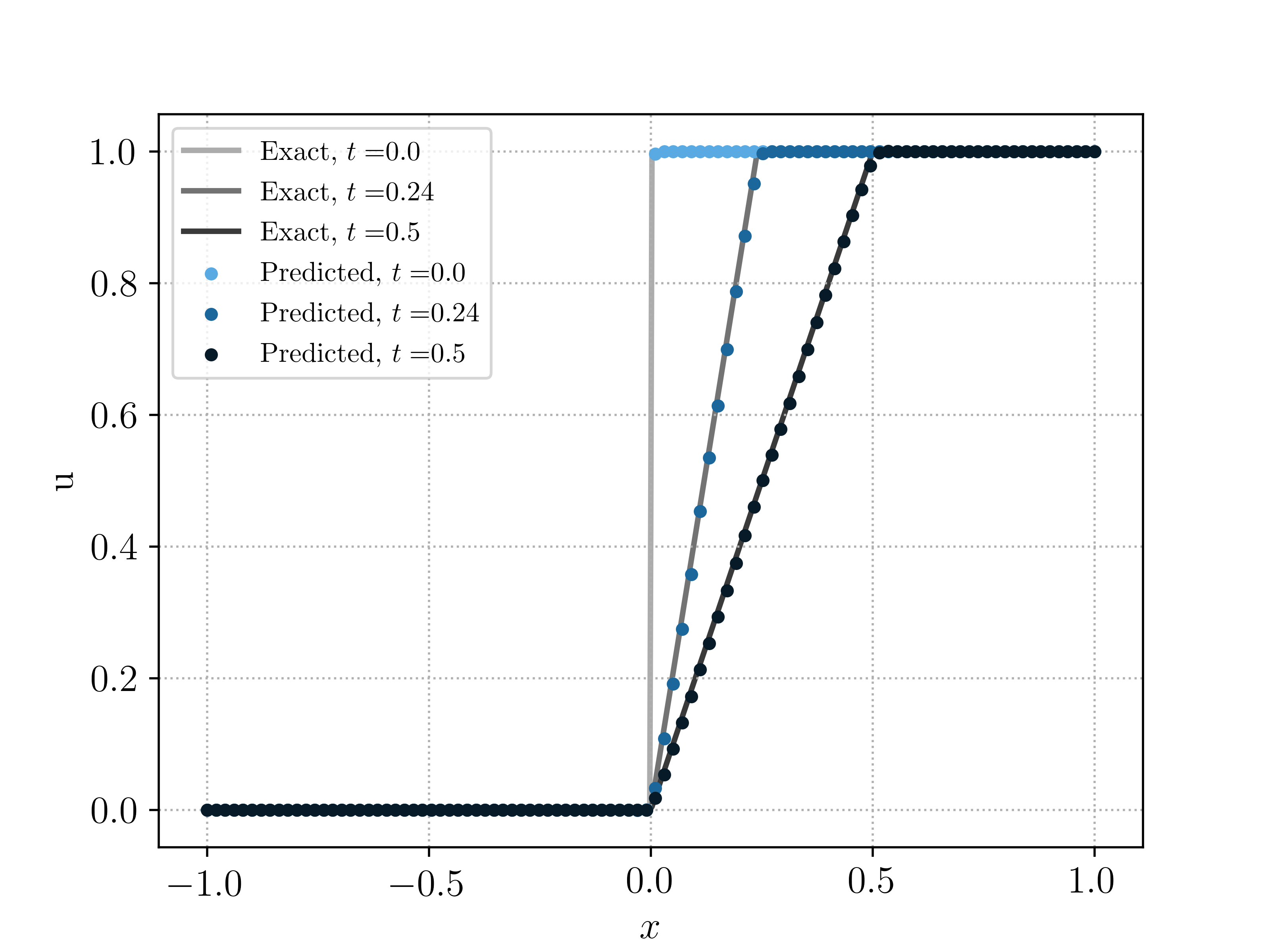}
        \caption{$\nu = 0$, $\er^r_G=0.012$}
    \end{subfigure}
    \caption{Burgers equation with rarefaction wave for different values of $\nu$}
    \label{fig:burg2}
\end{figure}
To further test the bound \eqref{eq:begenb}'s ability to explain performance of PINNs for the viscous Burgers' equation, we consider the following initial and boundary conditions,
\[   
\bar{u}(x) = 
     \begin{cases}
       0, &\quad\text{if } x\leq 0 \\
       1, &\quad\text{if } x> 0\\
     \end{cases}\quad \forall ~x \in [-1,1],
\]
\begin{equation}
     u(t,-1) =0, ~ u(t,1)=1, ~ \forall ~t \in [0,0.5].
\end{equation}
Given the discontinuity in the initial data we train the PINNS with a larger number of boundary training samples $N_{tb}=512$ and $N_{sb}=512$, while leaving $N_{int}= 8192$, unchanged. As in the previous experiments, training sets are Sobol sequences and an ensemble training is preformed to configure the network architecture. The results are summarized in Table \ref{tab:rar} and figure \ref{fig:burg2}. In this case, the exact solution is a so-called rarefaction wave (see figure \ref{fig:burg2}
for the reference solution, computed in the manner, analogous to the previous numerical experiment) and the gradient of the solution remains bounded, uniformly as the viscosity coefficient $\nu \rightarrow 0$. Hence, from the bound \eqref{eq:begenb}, we expect that PINNs will efficiently approximate the underlying solution for all values of the viscosity coefficient. This is indeed verified in the solution snapshots, presented figure \ref{fig:burg2}, where we observe that the PINN approximates the reference solution quite well, for all values of the viscosity coefficient. This behavior is further verified in table \ref{tab:rar}, where we see that the generalization error \eqref{eq:begen}, remains low (less than $2\%$) for all the values of viscosity and in fact, reduces slightly as $\nu \rightarrow 0$, completely validating the error estimate \eqref{eq:begenb}.

\begin{table}[htbp] 
    \centering
    \renewcommand{\arraystretch}{1.1} 
    
    \footnotesize{
        \begin{tabular}{ c c c c c c c c c c c} 
            \toprule
            \bfseries $\nu$  &\bfseries $N_{int}$  & \bfseries $N_{sb}$& \bfseries $N_{tb}$  &\bfseries $K-1$ & \bfseries $d$  &\bfseries $L^1$-reg. &\bfseries $L^2$-reg. &\bfseries $\lambda$ &\bfseries $\er_T$ &\bfseries $\er_G^r$ \\ 
            \midrule
            \midrule
            $0.01/pi$            & 8192   & 256 & 256 &8&20& 0.0& 0.0& 0.1 & 0.0005& 1.0\% \\
            \midrule 
              $0.005/pi$  & 8192  & 256 & 256 &10&20& 0.0& 0.0& 0.1 &0.00075& 1.2\% \\
                \midrule 
              $0.001/pi$   & 8192  & 256& 256&10 &20& 0.0& $10^{-6}$& 0.1 &0.009& 11.0\%\\
                \midrule 
              $0.0$   & 8192 & 256& 256 &8 &24& 0.0& $10^{-5}$& 0.1 &0.08& 23.0\%\\

            \bottomrule
        \end{tabular}
    \caption{Best performing \textit{Neural Network} configurations for the Burgers equation with shock, for different values of the parameter $\nu$. }
        \label{tab:burg}
    }
\end{table}
\begin{table}[htbp] 
    \centering
    \renewcommand{\arraystretch}{1.1} 
    
    \footnotesize{
        \begin{tabular}{ c c c c c c c c c c c} 
            \toprule
            \bfseries $\nu$  &\bfseries $N_{int}$  & \bfseries $N_{sb}$& \bfseries $N_{tb}$  &\bfseries $K-1$ & \bfseries $d$  &\bfseries $L^1$-reg. &\bfseries $L^2$-reg. &\bfseries $\lambda$ &\bfseries $\er_T$ &\bfseries $\er_G^r$ \\ 
            \midrule
            \midrule
            $0.01/pi$            & 8192   & 512 & 512 &4&20& 0.0& 0.0& 0.1 & 0.0043& 2.2\% \\
            \midrule 
              $0.005/pi$  & 8192  & 512 & 512 &4&20& 0.0& 0.0& 0.1 &0.0034& 1.8\% \\
                \midrule 
              $0.001/pi$   & 8192  & 512& 512&4 &16& 0.0& 0.0& 0.1 &0.00048& 1.6\%\\
                \midrule 
              $0.0$   & 8192 & 512& 512 &4 &20& 0.0& 0.0& 0.1 &0.00033& 1.2\%\\

            \bottomrule
        \end{tabular}
    \caption{Best performing \textit{Neural Network} configurations for the Burgers equation with rarefaction wave, for different values of the parameter $\nu$. }
        \label{tab:rar}
    }
\end{table}
\section{Incompressible Euler equations}
\label{sec:5}
\subsection{The underlying PDE}
The motion of an inviscid, incompressible fluid is modeled by the incompressible Euler equations \cite{MBbook}. We consider the following form of these PDEs,
\begin{equation}
    \label{eq:ie}
    \begin{aligned}
    \bu_t + \left(\bu \cdot \nabla\right)\bu + \nabla p &= \f, \quad (x,t) \in D \times (0,T), \\
    {\rm div}(\bu) &= 0, \quad (x,t) \in D \times (0,T), \\
    \bu\cdot\nl &= 0, \quad (x,t) \in \bD \times (0,T), \\
    \bu(x,0) &= \bar{\bu}(x), \quad x \in D.
    \end{aligned}
\end{equation}
Here, $D \subset \R^d$, for $d=2,3$ is an open, bounded, connected subset with smooth $C^1$ boundary $\bD$, $D_T = D \times (0,T)$, $\bu: D_T \mapsto \R^d$ is the velocity field, $p: D_T \mapsto \R$ is the pressure that acts as a Lagrange multiplier to enforce the divergence constraint and $\f \in C^1(D_T;\R^d)$ is a forcing term. We use the \emph{no penetration} boundary conditions here with $\nl$ denoting the unit outward normal to $\bD$. 

Note that we have chosen to present this form of the incompressible Euler equations for simplicity of exposition. The analysis, presented below, can be readily but tediously extended to the following,
\begin{itemize}
\item Other boundary conditions such as periodic boundary conditions on the torus ${\mathbb T}^d$.
\item The \emph{Navier-Stokes equations}, where we add the \emph{viscous term} $\nu \Delta \bu$ to the first equation in \eqref{eq:ie}, with either periodic boundary conditions or the so-called \emph{no slip} boundary conditions i.e, $\bu \equiv 0$, for all $x \in \bD$ and for all $t \in (0,T]$.
\end{itemize}
\subsection{PINNs}
We describe the algorithm \ref{alg:PINN} for this PDE in the following steps,
\subsubsection{Training set}
We chose the training set $\train \subset D_T$ with $\train = \train_{int} \cup \train_{sb} \cup \train_{tb}$, with interior, spatial and temporal boundary training sets, chosen exactly as in section \ref{sec:tset}, either as quadrature points for a (composite) Gauss rule or as low-discrepancy sequences on the underlying domains. 
\subsubsection{Residuals}
For the neural networks $(x,t) \mapsto \left(\bu_{\theta}(x,t),p_{\theta}(x,t)\right) \in C^k((0,T)\times D) \cap  C([0,T]\times \bar{D})$, defined by \eqref{eq:ann1}, with a smooth activation function and $\theta \in \Theta$ as the set of tuning parameters, we define the residual $\res$ in algorithm \ref{alg:PINN}, consisting of the following parts,
\begin{itemize}
    \item \emph{Velocity residual} given by,
    \begin{equation}
        \label{eq:ires1}
        \res_{\bu,\theta}(x,t):= (\bu_{\theta})_t + \left(\bu_{\theta} \cdot \nabla\right)\bu_{\theta} + \nabla p_{\theta}  - \f, \quad (x,t) \in D \times (0,T),
    \end{equation}
    \item \emph{Divergence residual} given by,
    \begin{equation}
        \label{eq:ires2}
        \res_{div,\theta}(x,t):= {\rm div}(\bu_{\theta}(x,t)), \quad (x,t) \in D \times (0,T),
    \end{equation}
  \item  \emph{Spatial boundary Residual} given by,
\begin{equation}
    \label{eq:ires3}
    \res_{sb,\theta}(x,t):= \bu_{\theta}(x,t)\cdot \nl, \quad \forall x \in \bD, ~ t \in (0,T]. 
\end{equation}
\item \emph{Temporal boundary Residual} given by,
\begin{equation}
    \label{eq:ires4}
    \res_{tb,\theta}(x):= \bu_{\theta}(x,0) - \bar{\bu}(x), \quad \forall x \in D. 
\end{equation}
\end{itemize}
As the underlying neural networks have the required regularity, the residuals are well-defined. 
\subsubsection{Loss function}
We consider the following loss function for training PINNs to approximate the incompressible Euler equation \eqref{eq:ie},
\begin{equation}
    \label{eq:ilf}
    J(\theta):= \sum\limits_{n=1}^{N_{tb}} w^{tb}_n|\res_{tb,\theta}(x_n)|^2 + \sum\limits_{n=1}^{N_{sb}} w^{sb}_n|\res_{sb,\theta}(x_n,t_n)|^2 + \lambda \left( \sum\limits_{n=1}^{N_{int}} w^{int}_n|\res_{\bu,\theta}(x_n,t_n)|^2+\sum\limits_{n=1}^{N_{int}} w^{int}_n|\res_{div,\theta}(x_n,t_n)|^2 \right).
\end{equation}
Here the residuals are defined by \eqref{eq:ires1}-\eqref{eq:ires4}. $w^{tb}_n$ are the $N_{tb}$ quadrature weights corresponding to the temporal boundary training points $\train_{tb}$, $w^{sb}_n$ are the $N_{sb}$ quadrature weights corresponding to the spatial boundary training points $\train_{sb}$ and $w^{int}_n$ are the $N_{int}$ quadrature weights corresponding to the interior training points $\train_{int}$. Furthermore, $\lambda$ is a hyperparameter for balancing the residuals, on account of the PDE and the initial and boundary data, respectively. 
\subsection{Estimate on the generalization error.}
We denote the PINN, obtained by the algorithm \ref{alg:PINN}, for approximating the incompressible Euler equations, as $\bu^{\ast}= \bu_{\theta^{\ast}}$, with $\theta^{\ast}$ being a (approximate,local) minimum of the loss function \eqref{eq:lf2},\eqref{eq:ilf}. We consider the following generalization error,
\begin{equation}
    \label{eq:iegen}
    \er_{G}:= \left(\int\limits_0^T \int\limits_D \|\bu(x,t) - \bu^{\ast}(x,t)\|^2 dx dt \right)^{\frac{1}{2}},
\end{equation}
with $\|\cdot\|$ denoting the Euclidean norm in $\R^d$. Note that we only consider the error with respect to the velocity field $\bu$ in \eqref{eq:iegen}. Although the pressure $p$ in \eqref{eq:ie} is approximated by the neural network $p^{\ast} = p_{\theta^{\ast}}$, we recall that the pressure is a Lagrange multiplier, and not a primary variable in the incompressible Euler equations.  Hence, we will not consider pressure errors here.

As in section \ref{sec:2}, we will bound the generalization error in terms of the following \emph{training errors},
\begin{equation}
    \label{eq:ietrain}
    \er_T^2:= \underbrace{\sum\limits_{n=1}^{N_{tb}} w^{tb}_n|\res_{tb,\theta^{\ast}}(x_n)|^2}_{\left(\er_T^{tb}\right)^2} + \underbrace{\sum\limits_{n=1}^{N_{sb}} w^{sb}_n|\res_{sb,\theta^{\ast}}(x_n,t_n)|^2}_{\left(\er_T^{sb}\right)^2} + \underbrace{\sum\limits_{n=1}^{N_{int}} w^{int}_n|\res_{\bu,\theta^{\ast}}(x_n,t_n)|^2}_{\left(\er_T^{\bu}\right)^2}+\lambda\underbrace{\sum\limits_{n=1}^{N_{int}} w^{int}_n|\res_{div,\theta^{\ast}}(x_n,t_n)|^2}_{\left(\er_T^{d}\right)^2}.
\end{equation}
As in the previous sections, the training errors can be readily computed \emph{a posteriori} from the loss function \eqref{eq:lf2}, \eqref{eq:ilf}. 

We have the following bound on the generalization error in terms of the training error,
\begin{theorem}
\label{thm:euler}
Let $\bu \in C^1((0,T)\times D) \cap C([0,T]\times \bar{D})$ be the classical solution of the incompressible Euler equations \eqref{eq:ie}. Let $\bu^{\ast} = \bu_{\theta^{\ast}}, p^{\ast} = p_{\theta^{\ast}}$ be the PINN generated by algorithm \ref{alg:PINN}, then the resulting generalization error \eqref{eq:iegen} is bounded as,\begin{equation}
    \label{eq:iegenb}
    \begin{aligned}
        \er_G^2 &\leq \left(T+C_{\infty}T^2e^{C_{\infty}T}\right) \left[\left(\er^{tb}_T\right)^2 + \left(\er^{\bu}_T\right)^2 +  C_0T^{\frac{1}{2}}\left(\er_T^{div} + \er_T^{sb}\right)      \right] \\
        &+ \left(T+C_{\infty}T^2e^{C_{\infty}T}\right) \left[C^{tb}_{qaud} N_{tb}^{-\alpha_{tb}} +C^{int,\bu}_{qaud} N_{int}^{-\alpha_{int}} + \left(C^{int,div}_{qaud}\right)^{\frac{1}{2}}  N_{int}^{-\frac{\alpha_{int}}{2}} + \left(C^{sb}_{qaud}\right)^{\frac{1}{2}}  N_{sb}^{-\frac{\alpha_{sb}}{2}}\right].
    \end{aligned}
\end{equation}
Here, the training errors are defined in \eqref{eq:ietrain} and the constants are given by,
\begin{equation}
    \label{eq:cons}
    \begin{aligned}
    C_0 &= C\left(\|\bu\|_{C^0([0,T]\times \bar{D})},\|\bu^{
    \ast}\|_{C^0([0,T]\times \bar{D})}, \|p\|_{C^0([0,T]\times \bar{D})}, \|p^{\ast}\|_{C^0([0,T]\times \bar{D})}\right), \\
    C_{\infty} &= 1 + 2C_d \|\nabla \bu\|_{L^{\infty}(D_T)}, 
\end{aligned}
\end{equation}
with $C_d$ only depending on dimension $d$ and $C^{tb}_{quad} = C^{tb}_{qaud}\left(\|\res^2_{tb,\theta^{\ast}}\|_{C^k}\right)$, $C^{int,\bu}_{quad} = C^{int}_{qaud}\left(\|\res^2_{\bu,\theta^{\ast}}\|_{C^{k-1}}\right)$, $C^{int,div}_{quad} = C^{int}_{qaud}\left(\|\res^2_{div,\theta^{\ast}}\|_{C^{k-1}}\right)$ and $C^{sb}_{quad} = C^{sb}_{qaud}\left(\|\res^2_{sb,\theta^{\ast}}\|_{C^k}\right)$ are the constants associated with the quadrature errors \eqref{eq:hquad1}-\eqref{eq:hquad3}. 
\end{theorem}
\begin{proof}
We will drop explicit dependence of all quantities on the parameters $\theta^{\ast}$ for notational convenience. We denote the difference between the underlying solution $\bu$ of \eqref{eq:ie} and PINN $\bu^{\ast}$ as $\hat{\bu} = \bu^{\ast} - \bu$. Similarly $\hat{p} = p^{\ast}-p$. Using the PDE \eqref{eq:ie} and the definitions of the residuals \eqref{eq:ires1}-\eqref{eq:ires4}, a straightforward calculation yields the following PDE for the $\hat{\bu}$,
\begin{equation}
    \label{eq:iehat}
    \begin{aligned}
    \hat{\bu}_t + \left(\hat{\bu} \cdot \nabla\right)\hat{\bu} + \left(\bu \cdot \nabla\right)\hat{\bu} + \left(\hat{\bu} \cdot \nabla\right)\bu+ \nabla \hat{p} &= \res_{\bu}, \quad (x,t) \in D \times (0,T), \\
    {\rm div}(\hat{\bu}) &= \res_{div}, \quad (x,t) \in D \times (0,T), \\
    \hat{\bu}\cdot\nl &= \res_{sb}, \quad (x,t) \in \bD \times (0,T), \\
    \bu(x,0) &= \res_{tb}, \quad x \in D.
    \end{aligned}
\end{equation}
We take a inner product of the first equation in \eqref{eq:iehat} with the vector $\hat{\bu}$ and use the following vector identities,
\begin{align*}
    \hat{\bu}\cdot \partial_t \hat{\bu} &= \partial_t \left(\frac{\|\hat{\bu}\|^2}{2}\right), \\
    \hat{\bu} \cdot \left(\left(\hat{\bu} \cdot \nabla\right)\hat{\bu}\right) &= \left(\hat{\bu} \cdot \nabla\right)\left(\frac{\|\hat{\bu}\|^2}{2}\right), \\
    \hat{\bu} \cdot \left(\left(\bu \cdot \nabla\right)\hat{\bu}\right) &= \left(\bu \cdot \nabla\right)\left(\frac{\|\hat{\bu}\|^2}{2}\right), 
\end{align*}
yields the following identity,
\begin{align*}
\partial_t \left(\frac{\|\hat{\bu}\|^2}{2}\right) +     \left(\hat{u} \cdot \nabla\right)\left(\frac{\|\hat{\bu}\|^2}{2}\right)
+\left(\bu \cdot \nabla\right)\left(\frac{\|\hat{\bu}\|^2}{2}\right) +
\hat{\bu} \cdot \left(\left(\hat{\bu} \cdot \nabla\right)\bu\right)+\left(\hat{\bu} \cdot \nabla\right)\hat{p} = \hat{\bu}\cdot\res_{\bu}.
\end{align*}
Integrating the above identity over $D$ and integrating by parts, together with \eqref{eq:ie} and \eqref{eq:iehat} yields,
\begin{equation}
    \label{eq:ipf1}
\begin{aligned}
    \frac{d}{dt} \int_D \left(\frac{\|\hat{\bu}\|^2}{2}\right) dx &= \int_{D} \res_{div} \left(\frac{\|\hat{\bu}\|^2}{2}+ \hat{p}\right) dx - \int_{\bD} \res_{sb}\left(\frac{\|\hat{\bu}\|^2}{2}+ \hat{p}\right) ds(x) \\
    &-\int_D \hat{\bu} \cdot \left(\left(\hat{\bu} \cdot \nabla\right)\bu\right) dx + \int_D\hat{\bu}\cdot\res_{\bu} dx.  
\end{aligned}
\end{equation}
It is straightforward to obtain the following inequality, 
\begin{align*}
     \int_D \hat{\bu} \cdot \left(\left(\hat{\bu} \cdot \nabla\right)\bu\right) dx \leq C_d \|\nabla \bu\|_{\infty} \int_D \|\hat{\bu}\|^2 dx,
\end{align*}
with the constant $C_d$ only depending on dimension and $\|\bu\|_{\infty} = \|\bu\|_{L^{\infty}(D_T)}$.

Using the above estimate and estimating \eqref{eq:ipf1} yields,
\begin{equation}
    \label{eq:ipf2}
    \frac{d}{dt} \int_D \|\hat{\bu}\|^2 dx \leq C_0 \left[ \left(\int_D \left(\res_{div}\right)^2 dx\right)^{\frac{1}{2}} + \left(\int_{\bD} \left(\res_{sb}\right)^2 ds(x)\right)^{\frac{1}{2}}\right] + C_{\infty}  \int_D \|\hat{\bu}\|^2 dx + \int_D \res^2_{\bu} dx,
\end{equation}
with constants given by \eqref{eq:cons}.

For any $\bar{T} \leq T$, we integrate \eqref{eq:ipf2} over time and use some simple inequalities to obtain,
\begin{equation}
    \label{eq:ipf3}
    \begin{aligned}
    \int_D \|\hat{\bu}(x,\bar{T})\|^2 dx &\leq \Cont + C_{\infty}  \int_0^{\bar{T}}\int_D \|\hat{\bu}(x,t)\|^2 dx dt, \\
    \Cont&= \int_D \res_{tb}^2 dx + \int_0^{T} \int_D \res^2_{\bu} dx dt, \\
     &+ C_0T^{\frac{1}{2}}\left[ \left(\int_0^T\int_D \left(\res_{div}\right)^2 dxdt \right)^{\frac{1}{2}} + \left(\int_0^T\int_{\bD} \left(\res_{sb}\right)^2 ds(x)dt \right)^{\frac{1}{2}}\right] .
    \end{aligned}
\end{equation}
Now by using the integral form of the Gr\"onwall's inequality in \eqref{eq:ipf3} and integrating again over $[0,T]$ results in,
\begin{equation}
    \label{eq:ipf4}
    \er_G^2 \leq \left(T+C_{\infty}T^2e^{C_{\infty}T}\right)\Cont.
\end{equation}
Using the bounds \eqref{eq:hquad1}-\eqref{eq:hquad3} on the quadrature errors and the definition of $\Cont$ in \eqref{eq:ipf3}, we obtain,
\begin{align*}
    \Cont &\leq \sum\limits_{n=1}^{N_{tb}} w^{tb}_n|\res_{tb}(x_n)|^2 + C^{tb}_{qaud}\left(\|\res_{tb}\|_{C^k}\right) N_{tb}^{-\alpha_{tb}} \\
    &+ \sum\limits_{n=1}^{N_{int}} w^{int}_n|\res_{\bu}(x_n,t_n)|^2 + C^{int}_{qaud}\left(\|\res_{\bu}\|_{C^{k-1}}\right) N_{int}^{-\alpha_{int}}, \\
    &+ C_0T^{\frac{1}{2}} \left[\left(\sum\limits_{n=1}^{N_{int}} w^{int}_n|\res_{div}(x_n,t_n)|^2\right)^{\frac{1}{2}} +  \left(C^{int}_{qaud}\left(\|\res_{div}\|_{C^{k-1}}\right)\right)^{\frac{1}{2}} N_{int}^{-\frac{\alpha_{int}}{2}} \right] \\
    &+ C_0T^{\frac{1}{2}} \left[\left(\sum\limits_{n=1}^{N_{sb}} w^{sb}_n|\res_{sb}(x_n,t_n)|^2\right)^{\frac{1}{2}} +  \left(C^{sb}_{qaud}\left(\|\res_{sb}\|_{C^{k}}\right)\right)^{\frac{1}{2}} N_{sb}^{-\frac{\alpha_{sb}}{2}} \right]
\end{align*}
From definition of training errors \eqref{eq:ietrain} and \eqref{eq:ipf4} and the above inequality, we obtain the desired estimate \eqref{eq:iegenb}.

\end{proof}
\begin{remark}
The bound \eqref{eq:iegenb} explicitly requires the existence of a classical solution $\bu$ to the incompressible Euler equations, with a minimum regularity of $\nabla \bu \in L^{\infty}(D \times (0,T))$. Such solutions do exist as long as we consider the incompressible Euler equations in \emph{two space dimensions} and with sufficiently smooth initial data \cite{MBbook}. However, in three space dimensions, even with smooth initial data, the existence of smooth solutions is a major open question. It is possible that the derivative blows up and the constant $C_{\infty}$ in \eqref{eq:iegenb} is unbounded, leading to a loss of control on the generalization error. In general, complicated solutions of the Euler equations are characterized by strong vorticity, resulting in large values of the spatial derivative. The bound \eqref{eq:iegenb} makes it clear that the generalization error with PINNs can be large for such problems.  
\end{remark}
\subsection{Numerical Experiments}
We present experiments for the incompressible Euler equations in two space dimensions, i.e $d=2$ in \eqref{eq:ie}. Moreover, we will use the \emph{CELU} function, given by,
\begin{equation}
\label{eq:celu}
    CELU(x)=\max(0,x)+\min\big(0,\exp(x)-1\big),
\end{equation}
as the activation function $\sigma$ in \eqref{eq:ann1}, \cite{barron2017continuously}. The CELU function results in better approximation than the hyperbolic tangent, for the Euler equations.
\begin{figure}[h!]
    \begin{subfigure}{.49\textwidth}
        \centering\captionsetup{width=.775\linewidth}
        \includegraphics[width=1\linewidth]{{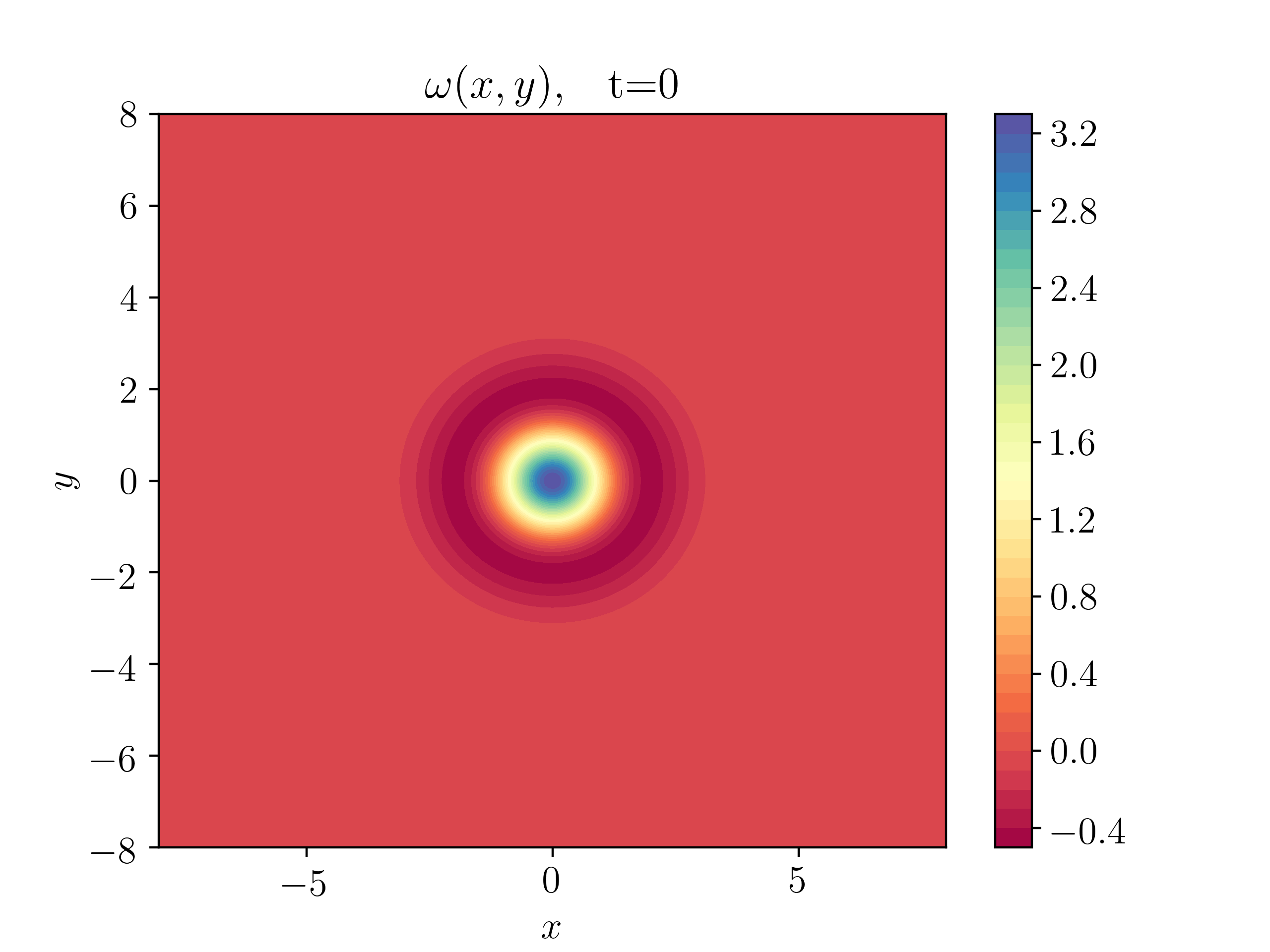}}
        \caption{Exact vorticity at $T=0$}
    \end{subfigure}
    \begin{subfigure}{.49\textwidth}
        \centering\captionsetup{width=.775\linewidth}
        \includegraphics[width=1\linewidth]{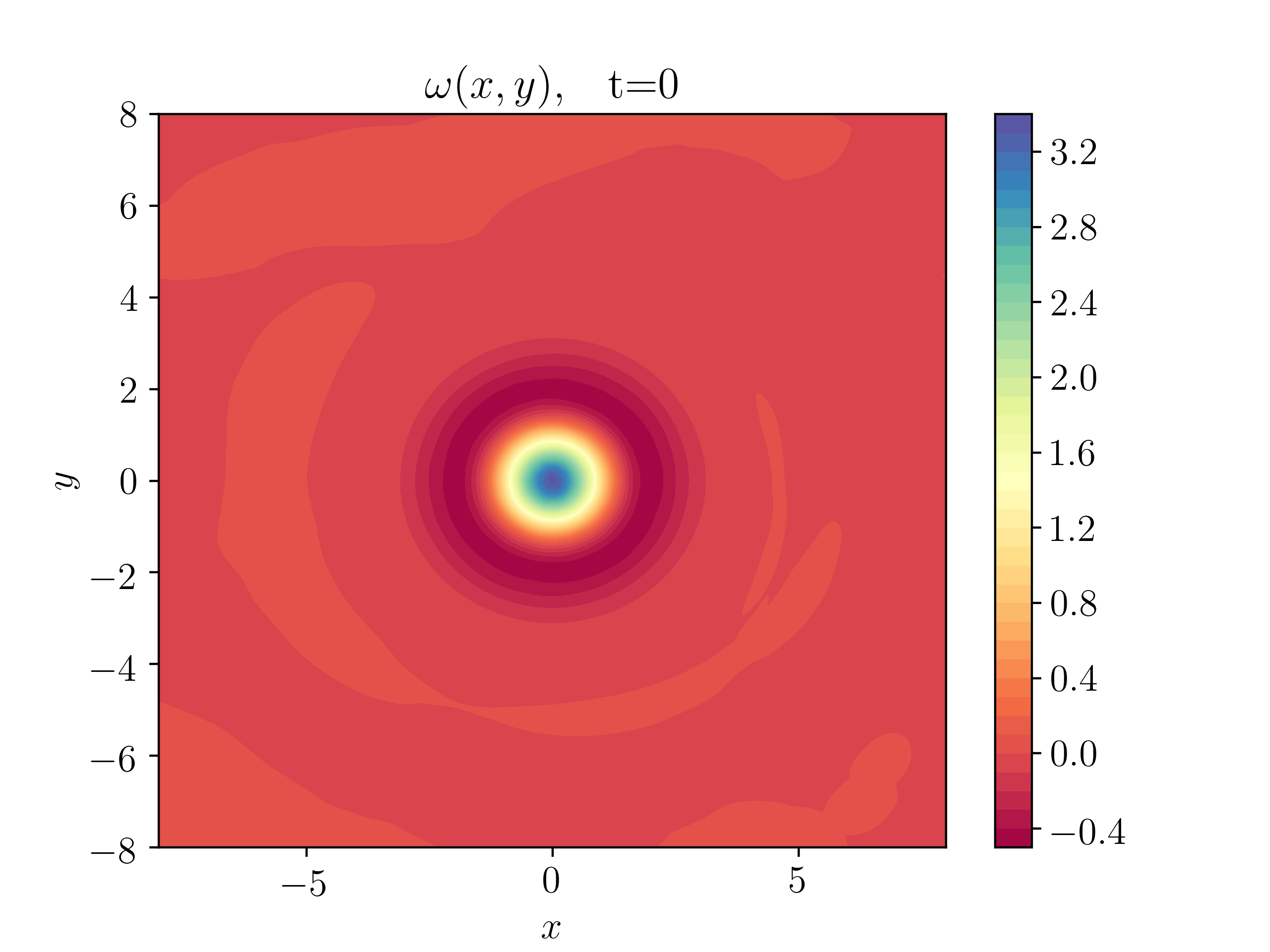}
        \caption{Approximate (PINN) vorticity at $T=0$}
    \end{subfigure}
    
    \begin{subfigure}{.49\textwidth}
        \centering\captionsetup{width=.775\linewidth}
        \includegraphics[width=1\linewidth]{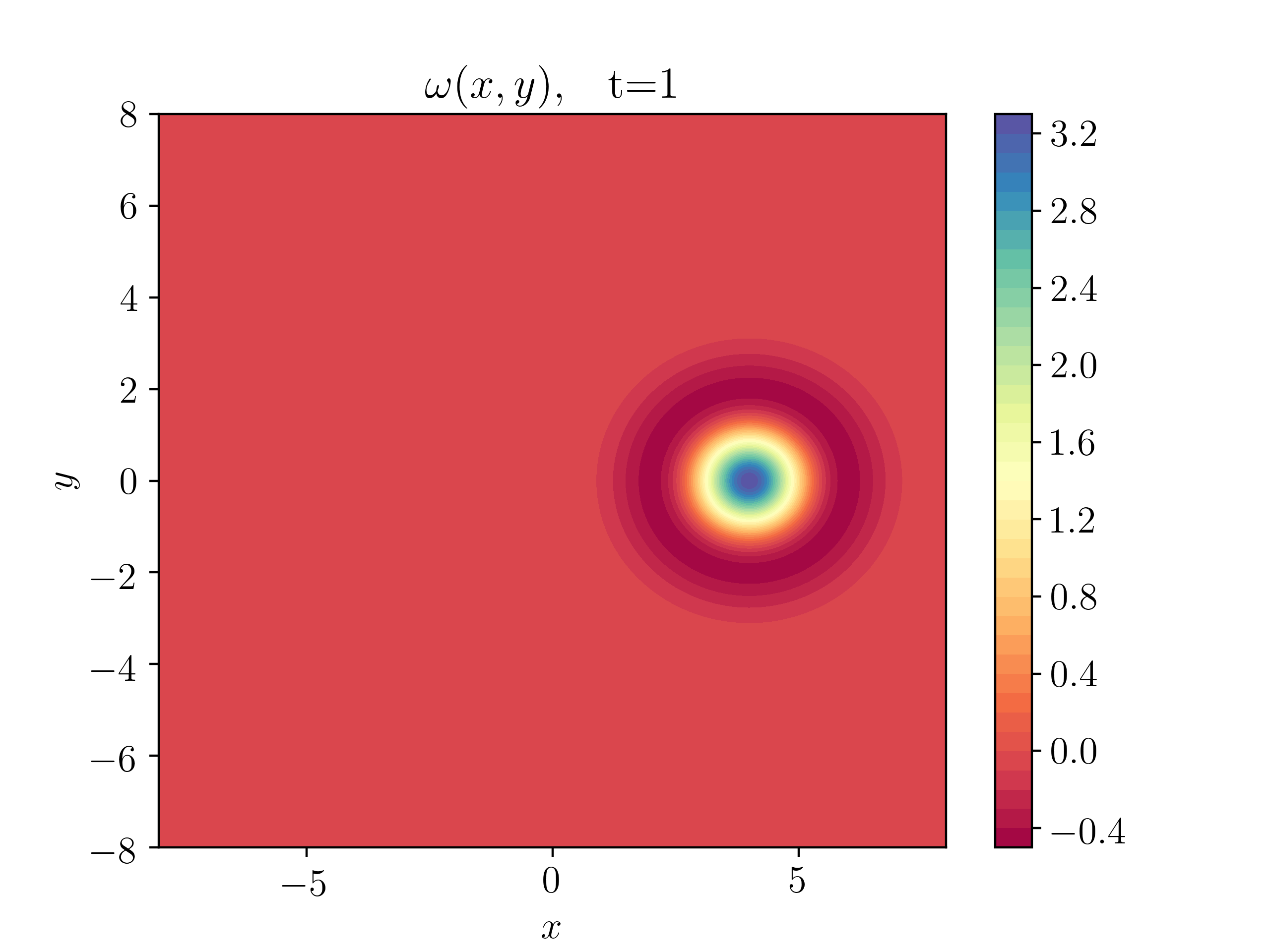}
        \caption{Exact vorticity at $T=1$}
    \end{subfigure}
    \begin{subfigure}{.49\textwidth}
        \centering\captionsetup{width=.775\linewidth}
        \includegraphics[width=1\linewidth]{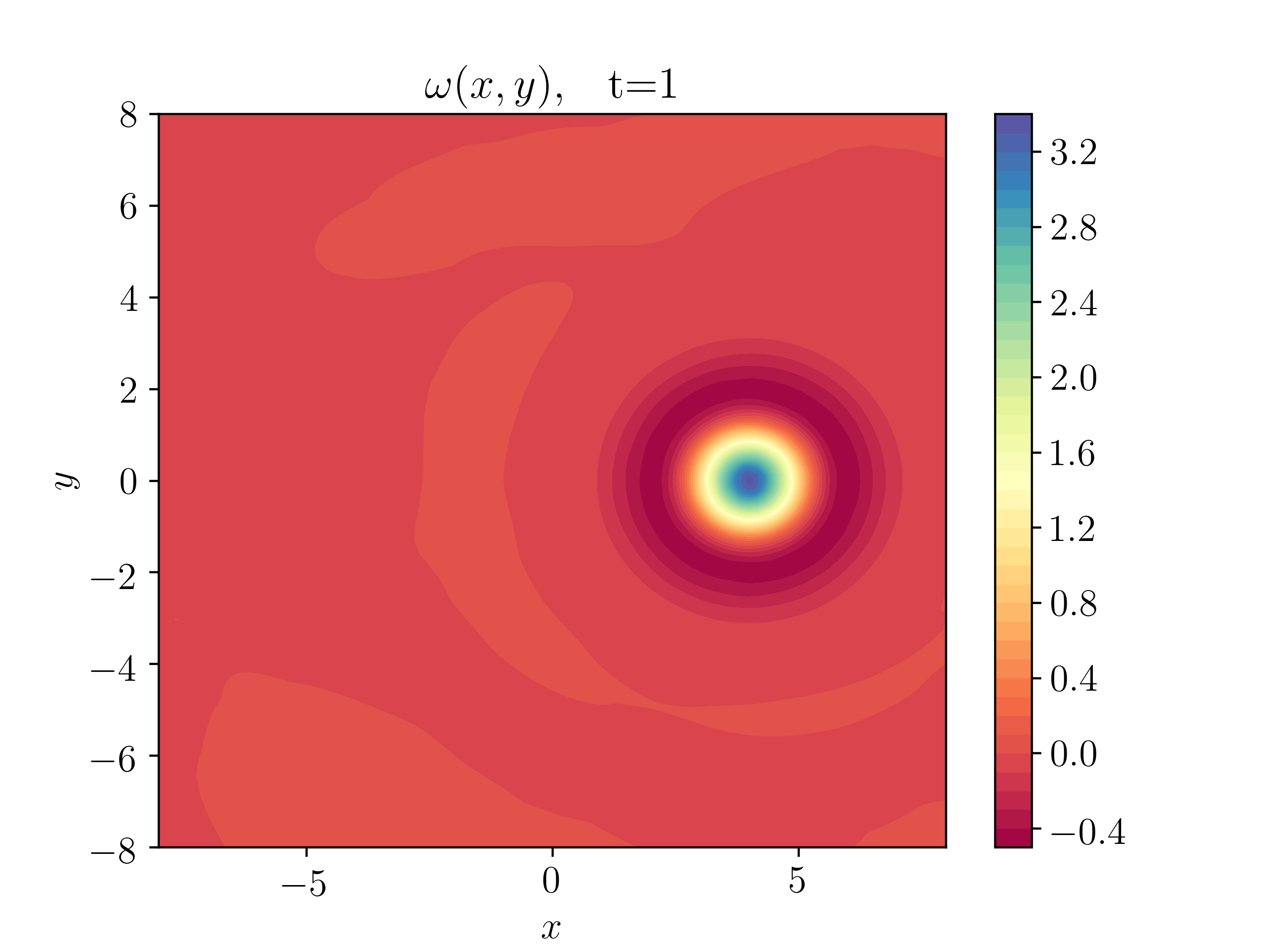}
        \caption{Approximate (PINN) vorticity at $T=1$}
    \end{subfigure}
    \caption{Exact and PINN solutions to the Taylor Vortex}
    \label{fig:euler1}
\end{figure}
\subsubsection{Taylor Vortex}
In the first numerical experiment, we consider the well-known Taylor Vortex, in a computational domain $D = [-8,8]^2$ with periodic boundary conditions and with the initial conditions, 
\begin{equation}
\label{eq:TV_ic}
\begin{aligned}
    u_0(x,y) &= -y e^{\frac{1}{2}(1-x^2-y^2)} + a_x,& \quad& (x,y)\in[-8,8]^2,\\
    v_0(x,y) &= xe^{\frac{1}{2}(1-x^2-y^2)} + a_y, & \quad& (x,y)\in[-8,8]^2,
\end{aligned}
\end{equation}
with $a_x=4$ and $a_y=0$.

In this case, one can obtain the following exact solution,
\begin{equation}
\label{eq:tv_exact}
\begin{aligned}
    u(t, x,y) &= -(y-a_yt) e^{\frac{1}{2}\big[1-(x-a_xt)^2-(y-a_yt)^2\big]}  + a_x,\\
    v(t, x,y) &= (x-a_xt)e^{\frac{1}{2}\big[1-(x-a_xt)^2-(y-a_yt)^2\big]}  + a_y.
\end{aligned}
\end{equation}

We will generate the training set with $N_{int}=8192$, $N_{tb} = N_{sb} = 256$ points, chosen as low-discrepancy Sobol sequences on the underlying domains. An ensemble training procedure is performed, as described in the previous section, and resulted in the hyperparameter configuration presented in Table \ref{tab:euler}. 

To visualize the solution, we follow standard practice and compute the vorticity $\omega = {\rm curl} (\bu)$ and present the exact vorticity and the one obtained from the PINN, generated by algorithm \ref{alg:PINN} in figure \ref{fig:euler1}. We remark that the vorticity can be readily computed from the PINN $\bu^{\ast}$ by automatic differentiation. We see from the figure, that the PINN, approximates the flow field very well, both initially as well as at later times, with small numerical errors. This good quality of approximation is further reinforced by the generalization error \eqref{eq:iegen}, computed from \eqref{eq:tv_exact} with $10^5$ uniformly distributed random points, and presented in Table \ref{tab:euler}. We see that the generalization error for the best hyperparameter configuration is only $0.012\%$, indicating very high accuracy of the approximation for this test problem. 
\begin{table}[htbp] 
    \centering
    \renewcommand{\arraystretch}{1.1} 
    
    \footnotesize{
        \begin{tabular}{ c c c c c c c c c c c} 
            \toprule
            \bfseries   &\bfseries $N_{int}$  & \bfseries $N_{sb}$& \bfseries $N_{tb}$  &\bfseries $K-1$ & \bfseries $d$  &\bfseries $L^1$-reg. &\bfseries $L^2$-reg. &\bfseries $\lambda$ &\bfseries $\er_T$&\bfseries $\er_G^r$ \\ 
            \midrule
            \midrule
            Taylor Vortex &8192   & 256 & 256 &12&24& 0.0& 0.0& 1 &0.0003& 0.012\% \\
            \midrule 
            Double Shear Layer           & 65536   & 16384 & 16384 &24&48& 0.0& 0.0& 0.1 &0.0025& 3.8\% \\

            \bottomrule
        \end{tabular}
    \caption{Best performing \textit{Neural Network} configurations for the Taylor Vortex and Double Shear Layer problem. Low-discrepancy Sobol points are used for every reported numerical example.}
        \label{tab:euler}
    }
\end{table}
\subsubsection{Double shear Layer}
We consider the two-dimensional Euler equations \eqref{eq:ie} in the computational domain $D = [0,2\pi]^2$ with periodic boundary conditions and consider initial data with the underlying vorticity, shown in figure \ref{fig:euler2} (Top Left). This vorticity, corresponds to a velocity field that has been evolved with a standard second-order finite difference projection method, with the well-known double shear layer initial data \cite{BCG1}, evolved till $T=1$. We are interested in determining if we can train a PINN to match the solution for later times. 

To this end, we acknowledge that the underlying solution is rather complicated (see figure \ref{fig:euler2} Top row) for the corresponding reference vorticity, and consists of fast moving sharp vortices. Moreover, the vorticity is high, implying from the bound \eqref{eq:iegenb}, that the generalization errors with PINNs can be high in this case. Hence, we consider training sets with larger number of points than the previous experiment, by setting $N_{int}=65536$ and $N_{tb}=N_{sb}=16384$. The ensemble training procedure resulted in hyperparameters presented in Table \ref{tab:euler}. 

We present the approximate vorticity computed with the PINN, together with the exact vorticity, in figure \ref{fig:euler2}, at three different times. From the figure, we see that the vorticity is approximated by the PINN quite well. However, the sharp vortices are smeared out and this is particularly apparent at later times. This is not surprising as the underlying solution is much more complicated in this case. Moreover, we have trained the PINN to approximate the velocity field, rather than the vorticity, and the generalization error \eqref{eq:iegen} is still quite low at $3.8 \%$ (see Table \ref{tab:euler}). 
\begin{figure}[h!]
    \begin{subfigure}{.3\textwidth}
        \centering\captionsetup{width=.775\linewidth}
        \includegraphics[width=1\linewidth]{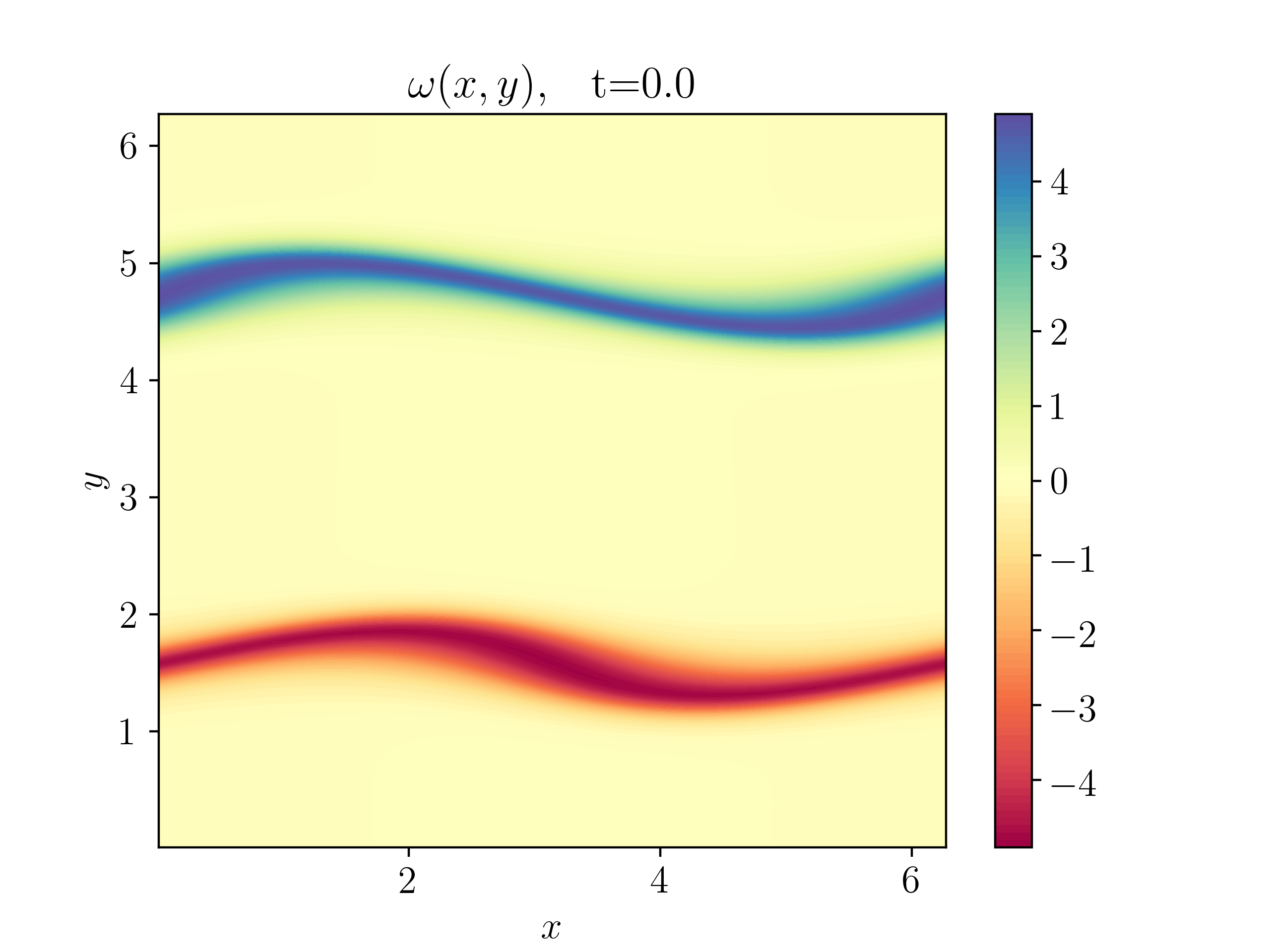}
        \caption{Reference, $T=0$}
    \end{subfigure}
    \begin{subfigure}{.3\textwidth}
        \centering\captionsetup{width=.775\linewidth}
        \includegraphics[width=1\linewidth]{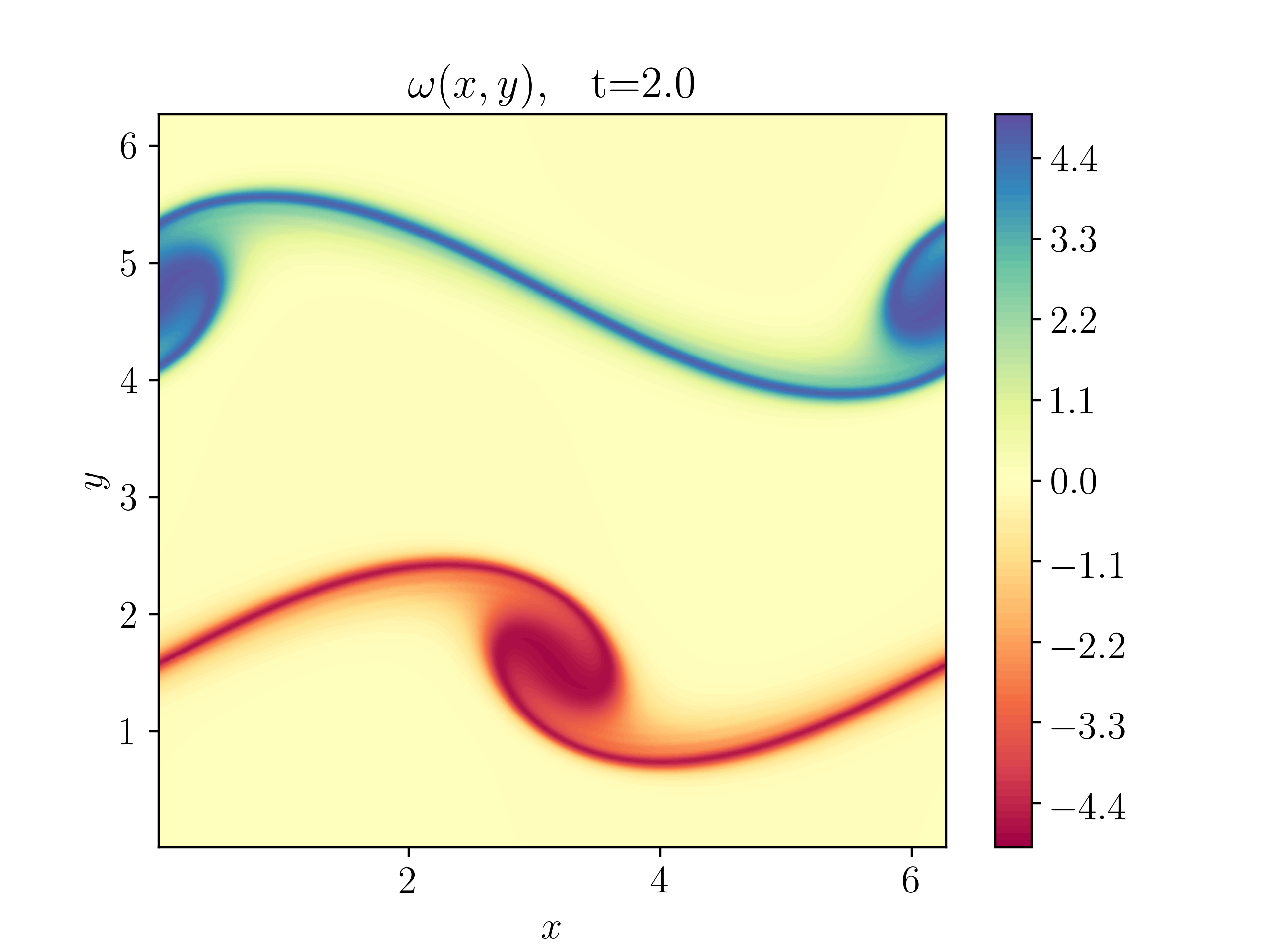}
        \caption{Reference, $T=2$}
    \end{subfigure}
    \begin{subfigure}{.3\textwidth}
        \centering\captionsetup{width=.775\linewidth}
        \includegraphics[width=1\linewidth]{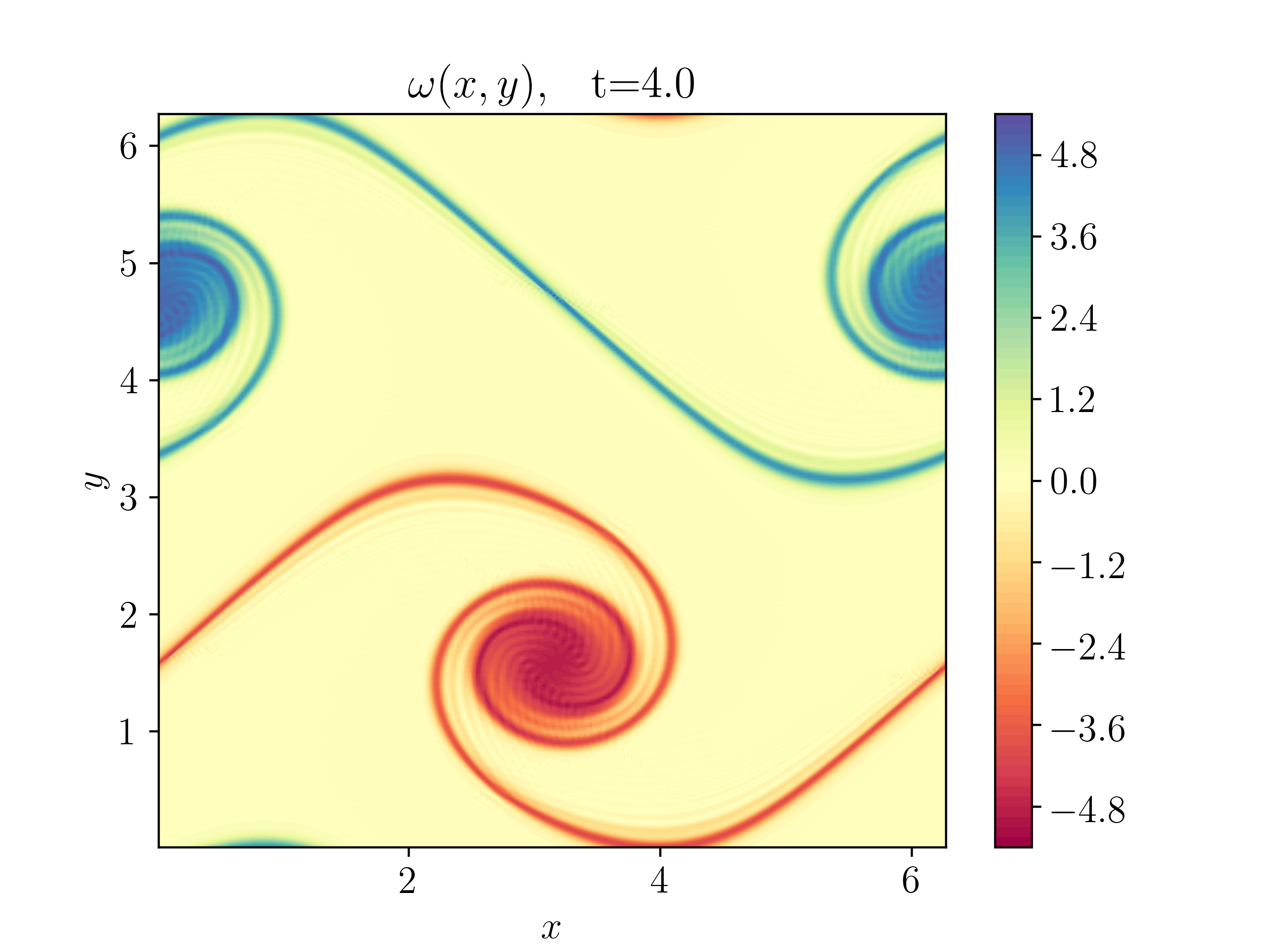}
        \caption{Reference, $T=4$}
    \end{subfigure}
    
    \begin{subfigure}{.3\textwidth}
        \centering\captionsetup{width=.775\linewidth}
        \includegraphics[width=1\linewidth]{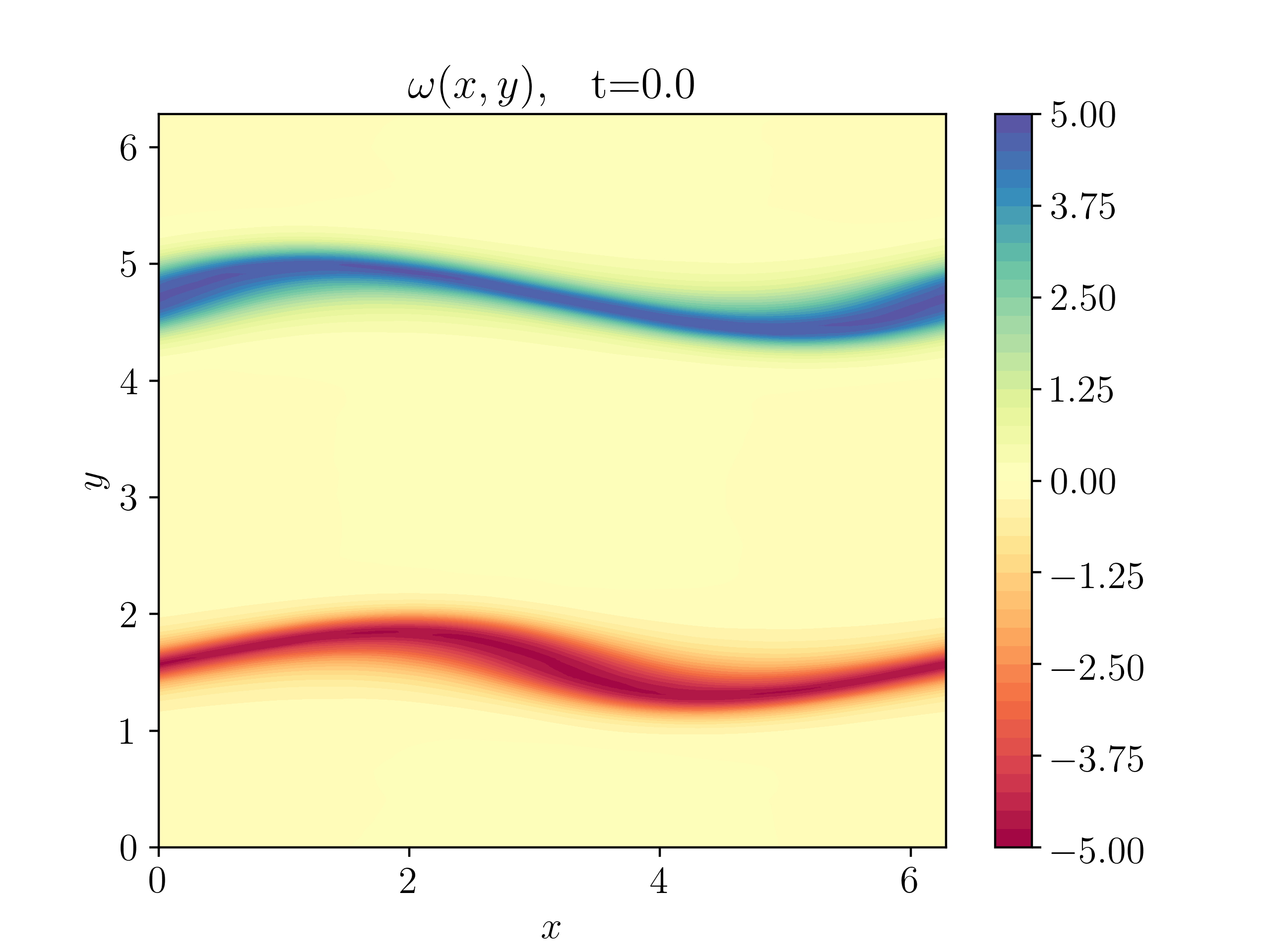}
        \caption{PINN, $T=0$}
    \end{subfigure}
    \begin{subfigure}{.3\textwidth}
        \centering\captionsetup{width=.775\linewidth}
        \includegraphics[width=1\linewidth]{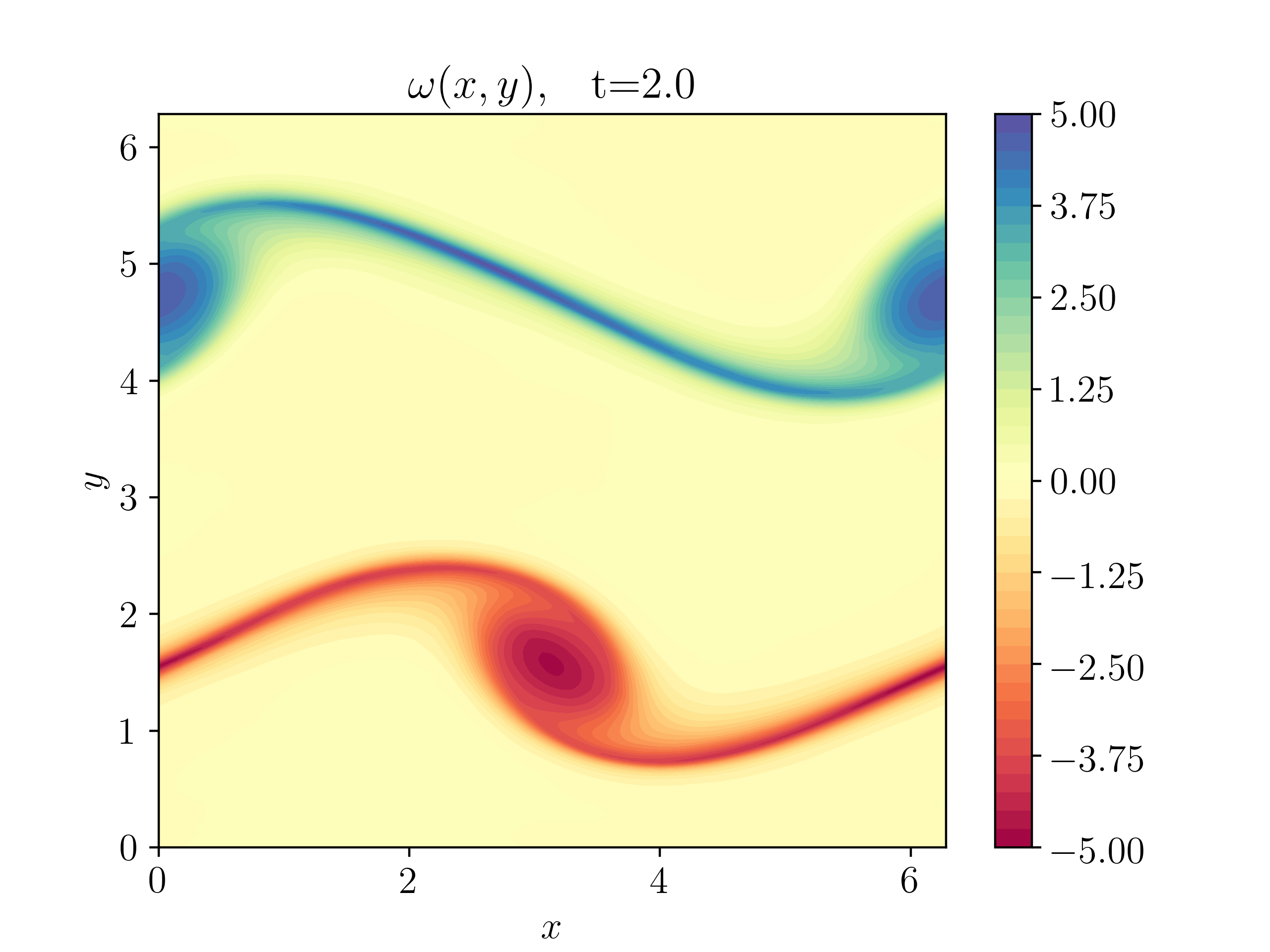}
        \caption{PINN, $T=2$}
    \end{subfigure}
    \begin{subfigure}{.3\textwidth}
        \centering\captionsetup{width=.775\linewidth}
        \includegraphics[width=1\linewidth]{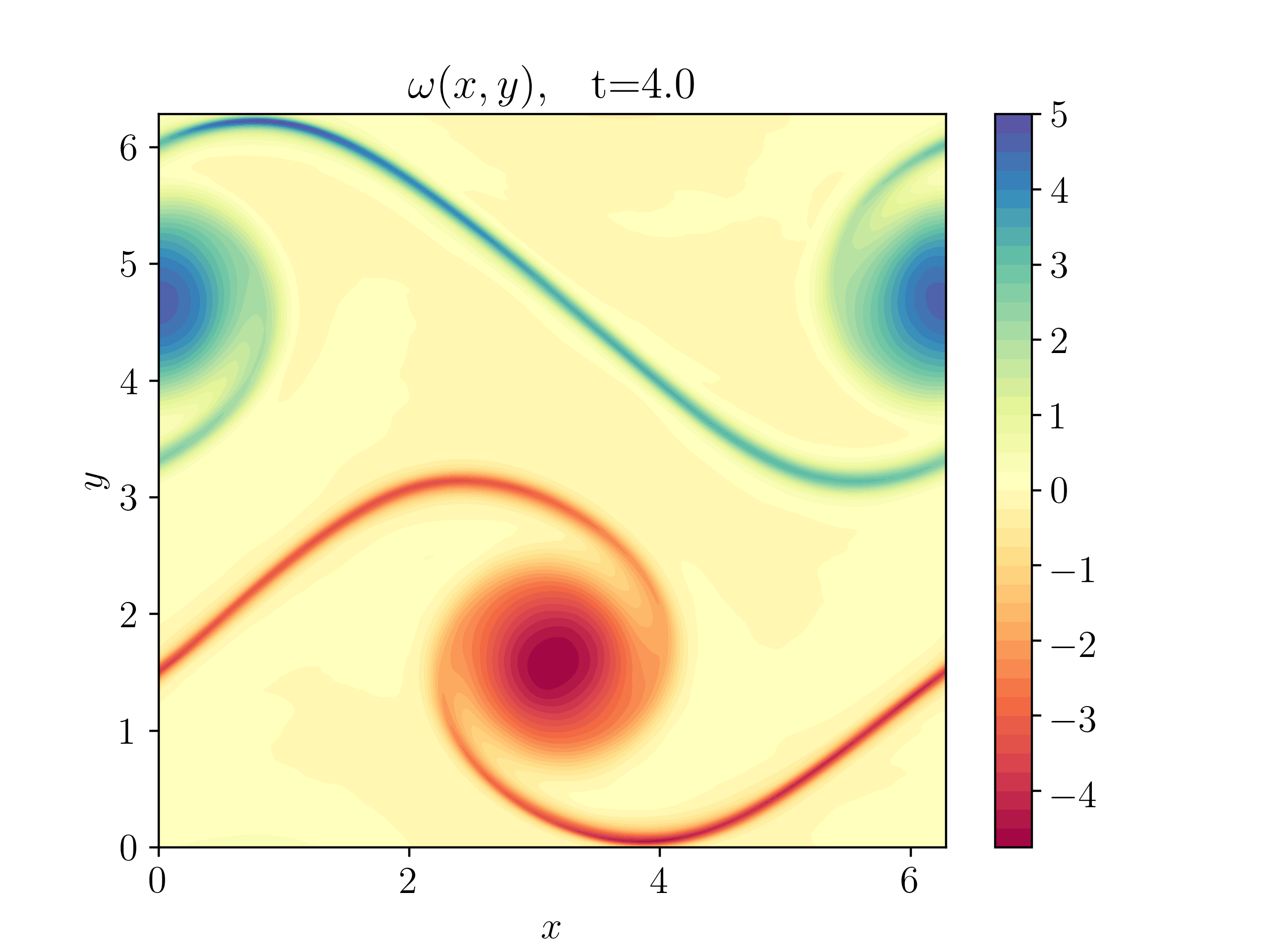}
        \caption{PINN, $T=4$}
    \end{subfigure}
\caption{Reference (Top Row) and PINN generated (Bottom Row) vorticities for the double shear layer problem at different times}
    \label{fig:euler2}
\end{figure}

\section*{Code} The building and the training of PINNs, together with the ensemble training for the selection of the model hyperparameters, are performed with a collection of Python scripts, realized with the support of PyTorch \url{https://pytorch.org/}. The scripts can be downloaded from \url{https://github.com/mroberto166/PinnsSub}.

\section{Discussion}
\label{sec:6}
Physics informed neural networks (PINNs), originally proposed in \cite{Lag1}, have recently been extensively used to numerically approximate solutions of PDEs in different contexts \cite{KAR1,KAR2,KAR4} and references therein. The main aim of this paper was to suggest possible explanations for this efficient approximation by PINNs. In particular, one needs to explain why minimization the PDE residual at collocation points will lead to bounds (control) on the overall approximation or generalization error.

In this article, we introduce an abstract framework, where an abstract nonlinear PDE \eqref{eq:pde} is approximated by a PINN, generated with the algorithm \ref{alg:PINN}. We repeat that the key point of this algorithm is to minimize the PDE residual \eqref{eq:res1}, at training points chosen as the quadrature points, corresponding to an underlying quadrature rule \eqref{eq:quad}. The resulting generalization error is bounded in the abstract error estimate \eqref{eq:egenb}, in terms of the training error, the number of training samples and involves constants that stem from a stability estimates \eqref{eq:assm2} for the underlying PDE as well as in the quadrature error. From our bound \eqref{eq:egenb}, one can conclude that as long as the training error is small, the number of quadrature points is sufficiently large and one has some control on the constants related to the stability of PDE and accuracy of the quadrature rule, the resulting generalization error will be small.  A key role in our error analysis is played by the stability of solutions of the underlying PDE to inputs that we leveraged to derive the abstract error bound \eqref{eq:egenb}. 

We illustrate our approach with three representative examples for PDEs, i.e, semilinear parabolic PDEs, viscous scalar conservation laws and the incompressible Euler equations of fluid dynamics. For each of these examples, the abstract framework was worked out and resulted in the bounds \eqref{eq:hegenb}, \eqref{eq:begenb}, \eqref{eq:iegenb} on the PINN generalization errors. All the bounds were of the form \eqref{eq:egenb}, in terms of the training and quadrature errors and with constants relying on stability (and regularity) estimates for \emph{classical} solutions of the underlying PDE. 

We also presented numerical experiments to validate the proposed theory. The numerical results were consistent with the derived estimates on the generalization error. For the heat equation, the results showed that PINNs are able to approximate the solutions accurately, even for very high ($100$) dimensional problems. For the viscous scalar conservation laws, we observed very accurate approximations with PINNs, for all values of the viscosity, as long as the underlying solution was at least Lipschitz continuous. However, as expected from the estimate \eqref{eq:begenb}, the accuracy deteriorated for the inviscid problem, when shocks were formed. Results with the two-dimensional incompressible Euler equations were also consistent with the derived error estimate \eqref{eq:iegenb}. 

With the exception of the very recent paper \cite{DAR1}, this article is one of the first rigorous investigations on approximations of PDEs by PINNs. Given that \cite{DAR1} also considers this theme, it is instructive to compare the two articles and highlight differences. In \cite{DAR1}, the authors focus on consistent approximation by PINNs by estimating the PINN loss in terms of the number of randomly chosen training samples and training error, corresponding to a \emph{loss function}, regularized with H\"older norms of PDE residual, (Theorem 2.1 of \cite{DAR1}). Under assumptions of vanishing training error and uniform bounds on the residual in H\"older spaces, the authors prove convergence of the resulting PINN to the classical solution of the underlying PDE as the number of training samples increase. They illustrate their abstract framework on linear elliptic and parabolic PDEs. 

Although similar in spirit, there are major differences in our approach compared to that of \cite{DAR1}. First, we do not need any additional regularization terms and directly estimate the generalization error from the stability estimate \eqref{eq:assm2} and loss function \eqref{eq:lf2}. Second, given the abstract formalism of section \ref{sec:2}, in particular the error estimate \eqref{eq:egenb}, we can cover very general PDEs, including non-linear PDEs.  In fact, any PDE with a stability estimate of the form \eqref{eq:assm2} is covered by our approach. We believe that this provides a unified explanation, in terms of stability estimates and quadrature rules, for the robust performance of PINNs for approximating a large number of linear and non-linear PDEs.

Finally, our current work has certain advantages as well as limitations and forms the foundation for the following extensions,
\begin{itemize}
    \item We provided a very abstract framework for deriving error estimates for PINNs here and illustrated this approach for some representative linear and non-linear PDEs. Given the generality and universality of the proposed formalism, it is possible to extend the estimates for approximating a very wide class of PDEs that satisfy stability estimates of the form \eqref{eq:assm2} and have classical solutions. In addition to the examples here, such PDEs include all well-known linear PDEs such as linear elliptic PDEs, wave equations, Maxwell's equations, linear elasticity, Stokes equations, Helmholtz equation etc as well as nonlinear PDEs such as Schr\"odinger equations, semilinear elliptic equations, dispersive equations such as KdV and many others. Extension of PINNs to these equations can be performed readily. 
    \item A major limitation of the estimate \eqref{eq:egenb} on the generalization error lies in the fact that the rhs in this estimate involves the training error \eqref{eq:train}. We are not able to estimate this error rigorously. However, this is standard in machine learning \cite{MLbook2}, where the training error is computed \emph{a posteriori} (see the numerical results for training errors in respective experiments). The training error stems from the solution of a non-convex optimization problem in very high dimensions and there are no robust estimates on this problem. Once such estimates become available, they can be readily incorporated in the rhs of \eqref{eq:egenb}. For the time being, the estimate \eqref{eq:egenb} should be interpreted as \emph{as long as the PINN is trained well, it generalizes well}. This is also extensively verified in the presented numerical experiments.
    \item Although never explicitly assumed, stability estimates of the abstract form \eqref{eq:assm2}, at least for nonlinear PDEs, amounts to regularity assumptions on the underlying solutions. For instance, the subspace $Z$ in \eqref{eq:assm2} is often a space of functions with sufficiently high Sobolev (or H\"older) regularity. The issue of regularity of the underlying PDE solutions is brought to the fore in the numerical example (see figure \ref{fig:burg1} and table \ref{tab:burg}) for the viscous Burgers' equation. Here, we clearly need the underling PDE solution to be Lipschitz continuous, for the PINN to appproximate it. Can PINNs approximate rough solutions of PDEs? The techniques introduced here may not suffice for this purpose and non-trivial extensions are needed.
    \item Another key implicit assumption in our error bounds was on the dependence of the constants $C_{pde}$ and $C_{quad}$ on the underlying PINN. Hence, these constants can also depend on the number of training points $N$. In particular, $C_{quad}$ in \eqref{eq:egenb} can change with changing $N$. In general, one needs to estimate the constants $C_{pde}$ and $C_{quad}$ for each specific PDE (and the underlying quadrature rule) in terms of the weights of the resulting neural networks. These detailed calculations are out of scope of the current paper and we refer the interested reader to \cite{DRM1} for detailed calculations of these constants in the case of linear Kolmogorov PDEs. Nevertheless, in our numerical experiments, we consistently observe that these constants are controlled in practice and the training and generalization errors are highly correlated (see figure \ref{fig:tg} for an illustration). This correlation also provides users with a practical criteria to select PINNs during ensemble training as low training error is a good indication of low generalization error.  
    \item Finally, we have focused on the forward problems for PDEs in this paper. One of the most attractive features of PINNs is their ability to solve Inverse problems \cite{KAR2,KAR4,KAR6}, with the same computational costs and efficiency as the forward problem. We prove generalization error estimates for PINNs, approximating inverse problems for PDEs in the companion paper \cite{MM2}.

\end{itemize}

\section*{Acknowledgements.} The research of SM and RM was partially supported by European Research Council Consolidator grant ERCCoG 770880: COMANFLO.

\bibliographystyle{abbrv}
\bibliography{MMpaperI}

\end{document}